\documentclass[12pt]{amsart}
\usepackage[utf8]{inputenc}
\usepackage{amsthm}
\usepackage[T2A]{fontenc}
\usepackage[table]{xcolor}

\usepackage[english]{babel}
\usepackage{graphicx}
\usepackage{amsfonts}
\usepackage{amsmath}
\usepackage{amssymb}
\usepackage{amscd}
\usepackage[all,cmtip,matrix, arrow]{xy}
\usepackage{amsthm}
\usepackage{tikz}
\usepackage{mathtools}
\usepackage{eucal}
\usepackage{bm}
\usepackage{scalerel}
\usepackage{stackrel}

\usepackage{hyperref}
\usepackage{color}

\newcommand{\PP}{\mathbb{P}}

\newcommand{\LL}{\mathbb{L}}
\newcommand{\RR}{\mathbb{R}}

\newcommand{\cO}{\mathcal{O}}
\newcommand{\cS}{\mathcal{S}}
\newcommand{\cK}{\mathcal{K}}
\newcommand{\cU}{\mathcal{U}}
\newcommand{\cD}{\mathcal{D}}
\newcommand{\cE}{\mathcal{E}}
\newcommand{\cI}{\mathcal{I}}
\newcommand{\cF}{\mathcal{F}}
\newcommand{\cL}{\mathcal{L}}

\newcommand{\cG}{\mathcal{G}}
\newcommand{\cm}{\mathfrak{m}}

\newcommand{\cC}{\mathcal{C}}
\newcommand{\bC}{\mathbf{C}}
\newcommand{\bQ}{\mathbf{Q}}
\newcommand{\hQ}{\mathbf{Q}}
\newcommand{\tcC}{\tilde{\mathcal{C}}}
\newcommand{\cR}{\mathcal{R}}
\newcommand{\cT}{\mathcal{T}}

\newcommand{\CG}{\mathbf{CG}}
\newcommand{\rG}{\mathbb{G}}
\newcommand{\qQ}{\mathsf{Q}}
\newcommand{\cQ}{\mathcal{Q}}

\newcommand{\Gad}{\mathbb{G}_2^{\mathrm{ad}}}

\newcommand{\br}{[\ -,-\ ]}
\newcommand{\cB}{\mathcal{B}}
\newcommand{\cP}{\mathcal{P}}

\newcommand{\tpi}{\widetilde{\pi}_2}

\newcommand{\tg}{\tilde{g}}

\newcommand{\tE}{\widetilde{E}}
\newcommand{\tcE}{\tilde{\mathcal{K}}}

\DeclareMathOperator{\OGr}{OGr}
\DeclareMathOperator{\Lie}{LieGr}

\newcommand{\Db}{\mathrm{D}^{b}}

\DeclareMathOperator{\IGr}{IGr}
\DeclareMathOperator{\Gr}{Gr}

\newcommand{\restr}[1]{|_{#1}}

\newcommand{\conv}{\,\lrcorner\,}

\newtheorem{lemma}{Lemma}[section]
\newtheorem{theorem}[lemma]{Theorem}
\newtheorem{proposition}[lemma]{Proposition}

\newtheorem{corollary}[lemma]{Corollary}

\theoremstyle{definition}
\newtheorem{definition}[lemma]{Definition}

\newtheorem{remark}[lemma]{Remark}
\newtheorem{example}[lemma]{Example}

\title{On the Derived Category of the Cayley Grassmannian}
\author{Lyalya Guseva}
\thanks{This work was supported by the Russian Science Foundation under grant no. 19-11-00164}
\address{Steklov Mathematical Institute of Russian Academy of Sciences, Moscow, Russia}
\email{lyalya.guseva1994@gmail.com}

\begin{document}
\maketitle{}
\begin{abstract}
We construct a full exceptional collection consisting of vector bundles in the derived category of coherent sheaves on the so-called Cayley Grassmannian, the subvariety of the Grassmannian~$\Gr(3, 7)$ parameterizing 3-subspaces that are annihilated by a general 4-form. The main step in the proof
of fullness is a construction of two self-dual vector bundles which is obtained from two operations with quadric bundles that might be interesting in themselves.
\end{abstract}
\section{Introduction}
The bounded derived category of coherent sheaves $\Db(X)$ is an important invariant of a smooth projective variety $X$.
In general, its structure is quite difficult to describe. One of the approaches to this problem is to split $\Db(X)$ into smaller pieces, which leads to the notion of a semiorthogonal decomposition.

In the important case where $\Db(X)$ possesses a full exceptional collection~$(E_{1}, E_{2}, \ldots, E_{m})$ the decomposition is as simple as possible: in this case every object of $\Db(X)$ admits a unique filtration with $i$-th subquotient being a direct sum of shifts of the objects $E_i$. Therefore, an exceptional collection can be considered as a kind of basis for the bounded derived category. The study of exceptional collections goes back to works of Beilinson~\cite{6}  and Kapranov~\cite{7}, where it was shown that projective spaces and, more generally, Grassmannians admit full exceptional collections consisting of vector bundles.

Having a full exceptional collection is a very restrictive condition on a variety $X$. For example, the Grothendieck group of the variety must necessarily be free, and the Hodge numbers $h^{p,q}(X)$ should be zero whenever $p\ne q$. 

Since the work of Beilinson and Kapranov, it has been conjectured that the bounded derived category
of a rational homogeneous variety admits a full exceptional collection. Even stronger, Lunts conjectures that the bounded derived category of a cellular variery possesses a full exceptional collection, see \cite[Conjecture 1.2]{EL}. While the general conjecture remains open, it is interesting to provide new examples of full exceptional collections that could shed some light upon the general case.

The aim of this article is to construct a full exceptional collection in the bounded derived category of coherent sheaves on the so called Cayley Grassmannian $\CG$, the subvariety of the Grassmannian~$\Gr(3,7)$ parametrizing 3-dimensional subspaces that are annihilated by a general 4-form. The Cayley Grassmannian $\CG$ is a spherical variety with respect to an action of the exceptional simple Lie group $\rG_2$. In fact,~$\CG$ is a smooth projective symmetric $\rG_2$-variety and it first appeared in \cite{R}, where such varieties with Picard number one were classified. Most of the varieties in this classification are homogeneous under their full automorphism group, or are just hyperplane sections of homogeneous spaces. One of the remarkable properties of $\CG$ is that it is not homogeneous, but as we prove in this paper it still possesses a full exceptional collection consisting of vector bundles. 

The geometry and cohomology of $\CG$ were studied in \cite{M} and \cite{BM}. In particular, in \cite{BM} the semisimplicity of the small quantum cohomology ring of $\CG$ was proved. Thus, our result confirms in this case Dubrovin’s conjecture predicting that the semisimplicity of the quantum cohomology ring implies the existence of a full exceptional collection. In \cite{BM} an exceptional collection closely related to ours was proposed, but it was not proved that the collection generates the whole derived category. 

It is worth mentioning that the K{\"u}chle fivefolds of type ($c5$) --- subvarieties of the Grassmannian~$\Gr(3, 7)$ parameterizing 3-subspaces that are isotropic for a given 2-form and are annihilated by a given 4-form --- are embedded into the Cayley Grassmannian $\CG$. This was the first motivation for the author to investigate the structure of $\Db(\CG)$.

\subsection{The Cayley Grassmannian}
There are several description of the Cayley Grassmannian $\CG$.

The first description explains the name. We consider the complexified nonassociative 8-dimensional Cayley algebra $\mathbb{O}$. The Cayley Grassmannian~$\CG$ can be defined as the set of 4-dimensional subalgebras of $\mathbb{O}$; it is a closed subvariety in the Grassmannian $\Gr(4,\mathbb{O})\simeq \Gr(4,8)$ of 4-dimensional vector subspaces in a 8-dimensional vector space $\mathbb{O}$. All these subalgebras contain the unit element $e$ of $\mathbb{O}$, so we can instead define $\CG$ as a closed subvariety in $\Gr(3,\mathbb{O}/(\mathbb{C}\cdot e))\simeq \Gr(3,7)$, that parametrizes the imaginary parts of the four-dimensional subalgebras of $\mathbb{O}$. 

The second description makes sense over any algebraically closed field $\Bbbk$ of characteristic 0. We consider the Grassmannian $\Gr(3,V)$ of 3-dimensional vector subspaces in a 7-dimensional vector space~$V$. Let us denote by $\cU\subset V\otimes \cO$ and $\cU^{\perp}\subset V^{\vee}\otimes \cO$ the tautological vector bundles on $\Gr(3,V)$, of ranks 3 and 4 respectively. A global section of $\cU^{\perp}(1)$ is given by a skew-symmetric 4-form $\lambda\in \Lambda^4V^{\vee}$. By definition, the Cayley Grassmannian $\CG$ is the zero locus of a general global section~$\lambda$ of $\cU^{\perp}(1)$. Explicitly, $\CG$ parametrizes 3-dimensional vector subspaces of $V$ annihilated by $\lambda$, i.e., $U\subset V$ such that $\lambda(u_1,u_2,u_3, -)=0$ for any $u_1,u_2,u_3\in U$. From this description we immediately deduce that the Cayley Grassmannian is a smooth Fano eightfold of index 4.

The equivalence of these two descriptions comes from the fact that the stabilizer
of a general 4-form on $V$ is isomorphic to the algebraic group $\rG_2$, that is the automorphism group of the octonions, see~\cite{M} for more details.

Also the Cayley Grassmannian can be described as the Hilbert scheme of conics on the adjoint homogeneous variety~$\Gad$, see Subsection \ref{section:Hilbert} for the details. 

Throughout the paper we will mostly use the second description of $\CG$. In particular, we will denote~by
\begin{equation*}
    \lambda\in \Lambda^4V^{\vee}
\end{equation*}
a general skew-symmetric 4-form that defines $\CG$.

\subsection{The main result} 

To state the main result, we need to make some preparations. The exceptional collection on $\CG$ that we construct is a Lefschetz collection, \cite{2,K06,8}. Recall that a Lefschetz exceptional collection with respect to an ample line bundle $\cO(1)$ is just an exceptional collection which consists of several blocks, each of them is a sub-block of the previous one twisted by $\cO(1)$, see Definition~\ref{def:lefschetz} for more details. The collection on $\CG$ that we are going to describe consists of four blocks with respect to the Pl{\"u}cker line bundle $\cO(1)$ of $\Gr(3,V)$ restricted to $\CG$. The common part of these four blocks (the rectangular part of the Lefschetz collection, see Definition~\ref{def:lefschetz}) consists of three vector bundles $(\cO, \cU^{\vee}, \Lambda^2\cU^{\vee})$.

To describe the nonrectangular part of the exceptional collection on $\CG$ we need to define two additional vector bundles. The first one is  $$\Sigma^{2,1}\cU^\vee \coloneqq (\cU^\vee \otimes \Lambda^2 \cU^{\vee})/\Lambda^3\cU^{\vee}.$$
To describe the second vector bundle note that on $\CG$ we have an embedding of vector bundles $$i_{\lambda}\colon\Lambda^2\cU\hookrightarrow \Lambda^2\cU^{\perp}$$ given by $\lambda$, see Lemma \ref{lemma:embedding} for more details. So on $\CG$ we can define the quotient bundle $\Lambda^2\cU^{\perp}/ \Lambda^2\cU$. For the exceptional collection we will need its dual
\begin{equation*} 
  \cR\coloneqq(\Lambda^2\cU^{\perp}/ \Lambda^2\cU)^{\vee}.
\end{equation*}
The main result of this paper is the following theorem.
\begin{theorem} \label{Theorem:introduction:main}
The collection of $15$ vector bundles on $\CG$
\begin{equation} \label{collection}
   \{ \underbrace{\cO, \cU^{\vee}, \Lambda^2\cU^{\vee}, \cR,\Sigma^{2,1}\cU^{\vee};}_{\mbox{block $1$}}\ 
    \underbrace{\cO(1), \cU^{\vee}(1), \Lambda^2\cU^{\vee}(1),\cR(1);}_{\mbox{block $2$}}
    \underbrace{\cO(2), \cU^{\vee}(2), \Lambda^2\cU^{\vee}(2);}_{\mbox{block $3$}} \underbrace{\cO(3), \cU^{\vee}(3), \Lambda^2\cU^{\vee}(3)}_{\mbox{block $4$}} \}
\end{equation}
is a full Lefschetz collection with respect to $\cO(1)$.
\end{theorem}
We will denote by $\mathfrak{E}(i)$ the collection consisting of three vector bundles:
\begin{equation} \label{mathfrak_E}
    \mathfrak{E}(i)\coloneqq (\cO(i),\cU^{\vee}(i), \Lambda^2\cU^{\vee}(i)).
\end{equation}

Let us sketch the idea of the proof of Theorem \ref{Theorem:introduction:main}. First, it will be more convenient for us to prove the fullness of a slightly different collection \eqref{ec}. We cover $\CG$ with the family of subvarieties~$\CG_f \stackrel{i_f}{\hookrightarrow}\CG$ defined as zero loci of sufficiently general global sections $f\in \mathrm{H}^0(\CG,\cU^{\vee})$. It turns out that $\CG_f$ is isomorphic to the isotropic Grassmannian $\mathrm{IGr}(3,6)$ and $\Db(\CG_f)$ posseses a full exceptional collection, so using standard arguments from \cite{2} we reduce the problem to the checking that some objects lie in the subcategory $\cD\subset \Db(\CG)$ generated by the collection \eqref{ec}: it is enough to show that $S^2\cU^{\vee}(m)\in \cD$ for~$m=0,1,2$ and that $\Sigma^{2,1}\cU^{\vee},\Sigma^{2,1}\cU^{\vee}(1) \in \cD$. This check is the most interesting part of the proof, so let us describe it more precisely.

First, we present two general constructions with quadric bundles. Roughly speaking, the first construction shows that we can glue two quadric bundles with isomorphic cokernel sheaves into a quadric bundle that determines self-dual isomorphism, see Proposition \ref{lemma:quadric_bundles_cokernel} for the details; and the second one allows to construct a new quadric bundle from a quadric bundle with the cokernel sheaf supported on a Cartier divisor, see Proposition \ref{quadric_construction}. Using these contructions we obtain several interesting $\rG_2$-equivariant quadric bundles on $\CG$.  Using again Proposition \ref{lemma:quadric_bundles_cokernel} we glue the obtained quadric bundles into the following self-dual vector bundles on $\CG$
\begin{equation} \label{self-dualities}
  \cE_{10}(1)\simeq \cE_{10}^{\vee} \quad \mbox{and} \quad \cE_{16}(-1)\simeq \cE_{16}^{\vee},
\end{equation}
where $\cE_{10}$ is an extension of $\cU^{\perp}$ by $S^2\cU$ and $\cE_{16}$ is an extension of $\Lambda^2\cU_3^{\vee} \oplus \cO(1)$ by $\cU_3^{\perp}\otimes \Lambda^2\cU_3^{\vee}$. Using \eqref{self-dualities} and some standard exact sequences we prove that $S^2\cU^{\vee}(m)\in \cD$ for~$m=0,1,2$ and that~$\Sigma^{2,1}\cU^{\vee},\Sigma^{2,1}\cU^{\vee}(1) \in \cD$.
\subsection{Organization of the paper}
The work is organized as follows. In Section \ref{preliminaries} we discuss background material on
semiorthogonal decompositions and recall some basic facts about the algebraic group~$\rG_2$.
In Section~\ref{section:quadric_constructions} we present general constructions with quadric bundles that will allow us to prove the self-dualities \eqref{self-dualities}.
In Section~\ref{section:G2action} we collect some useful results concerning the $\rG_2$-action on the Cayley Grassmannian.
In Section \ref{section:selfdualities} using the general constructions from Section~\ref{section:quadric_constructions} we describe some special quadric bundles on $\CG$ and prove the necessary self-dualities. 
In Section~\ref{section:fullness} we give a proof of fullness of the constructed collection.
In Appendix \ref{section:geometric_construction} we present several well-known to specialists geometric constructions for the Cayley Grassmannian: we show that~$\CG$ is isomorphic to the Hilbert scheme of conics on $\Gad$ and describe the Hilbert scheme of lines on $\CG$; these results are not directly related to the main theorem but are quite interesting by themselves.
In Appendix~\ref{section:computations} we collect all necessary cohomology computations for $\CG$ including exceptionality of the collection~\eqref{collection}, and describe the residual category for \eqref{collection}.

{\bf Acknowledgements:}
I am very grateful to my advisor, Alexander Kuznetsov, for many important suggestions, which led to significant improvements. I would also like to thank Sasha Petrov and Grzegorz Kapustka for useful discussions, Vladimir Arutyunov for writing the Sage script that helped to instruct our search for the exceptional collection, and Artem Avilov for his constant support. 

\section{Preliminaries} \label{preliminaries}
\subsection{Notation and conventions} \label{notations}
Let $\Bbbk$ be an algebraically closed field of characteristic zero. For any $\Bbbk$-vector space $W$ we denote by $\wedge$ the wedge product operation for skew forms and polyvectors, and by~$\conv$ the convolution operation
$$\Lambda^pW\otimes \Lambda^qW^{\vee}\stackrel{\conv}{\longrightarrow}\Lambda^{p-q}W\qquad (\text{if}\ p\ge q),$$
induced by the natural pairing $W\otimes W^{\vee}\to \Bbbk$.

If $p=n=\mathrm{dim}(W)$ and $0\ne \omega \in \Lambda^nW$, the convolution with $\omega$ gives an isomorphism
$$\Lambda^qW^{\vee}\simeq \Lambda^{n-q}W, \quad \xi \mapsto \xi^{\vee}:=\omega \conv \xi.$$
This isomorphism is canonical up to rescaling (since $\omega$ is unique up to rescaling). We say that $\xi^{\vee}$ is the \textbf{dual} of $\xi$.

We say that a $q$-form $\xi\in \Lambda^q W^{\vee}$ \textbf{annihilates} a $k$-dimensional subspace $U\subset W$, if $k\le q$  and $\xi \conv \Lambda^k U =0$. Analogously, we say that $U\subset W$ is \textbf{isotropic} for $\xi$ if $q\le k$ and $\Lambda^k U\conv \xi =0$, i.e. $\xi\restr{U}=0$.

We denote by $\Gr(k,W)$ the Grassmannian of $k$-dimensional vector subspaces in $W$. The
tautological vector subbundle of rank $k$ on $\Gr(k,W)$ is denoted by $\cU_k \subset W \otimes \cO_{\Gr(k,W)}$. The quotient
bundle is denoted simply by $\cQ_k$, and for its dual we use the notation $\cU_k^{\perp}:= \cQ_k^{\vee}$. We often use the following tautological exact sequences on $\Gr(k,W)$:
\begin{gather} \label{tautological_exact_sequence}
    0\to \cU_k \to W\otimes \cO \to \cQ_k \to 0\\ \label{tautological_exact_sequence_1}
    0\to \cU^{\perp}_k \to W^{\vee}\otimes \cO \to \cU^{\vee}_k \to 0.
\end{gather}
The point of the Grassmannian corresponding to a
subspace $U_k\subset W$ is denoted by~$[U_k]$, or even just~$U_k$. We recall that the line bundle
$\mathrm{det}\ \cU^{\vee}_k \simeq  \mathrm{det}\ \cQ_k$ is the ample generator of $\mathrm{Pic}(\Gr(k,W))$; we will denote it by $\cO(H_k)$.

Since we will mostly work with $\Gr(3,V)$ for a fixed 7-dimensional vector space $V$, we abbreviate its tautological and dual quotient
bundles by $\cU$ and $\cU^{\perp}$ respectively, unless this leads to a confusion. The line bundle $\cO(H_3)$ will be denoted by $\cO(1)$.

We will frequently use the following natural isomorphisms on $\Gr(3,V)$:
\begin{align} \label{preliminaries_isomorphisms}
   & \Lambda^2\cU\simeq \cU^{\vee}(-1);\\ \label{preliminaries_isomorphisms2}
   & \Lambda^2 \cU^{\perp} \simeq \Lambda^2 \cQ(-1).
\end{align}

The canonical line bundle of a variety $X$ will be denoted by $\omega_{X}$.
\subsection{Semiorthogonal decompositions and exceptional collections} \label{subsection:semiorthogonal}
We recall some well-known facts about semiorthogonal decompositions. Let $\mathcal{T}$ be a $\Bbbk$-linear triangulated category.
\begin{definition}
A sequence of full triangulated subcategories $\mathcal{A}_1, \ldots ,\mathcal{A}_m \in \mathcal{T}$ is \textbf{semiorthogonal}
if for all $1\le i {}<{} j\le m$ and all $G\in \mathcal{A}_i$, $H\in \mathcal{A}_j$ one has $$\mathrm{Hom}_{\mathcal{T}} (H,G) = 0.$$
Let $\langle \mathcal{A}_1, \ldots ,\mathcal{A}_m\rangle$ denote
the smallest full triangulated subcategory in $\mathcal{T}$ containing all $\mathcal{A}_i$.
If $\mathcal{A}_i$ are semiorthogonal and $\langle \mathcal{A}_1, \ldots ,\mathcal{A}_m \rangle = \mathcal{T}$,
we say that the subcategories $\mathcal{A}_i$ form a \textbf{semiorthogonal decomposition} of $\mathcal{T}$.
\end{definition}

\begin{definition}
An object $E$ of  $\mathcal{T}$ is \textbf{exceptional} if $\mathrm{Ext}^{\bullet}(E, E) = \Bbbk$
(i.e., $E$ is simple and has no non-trivial self-extensions).
\end{definition}

If $E$ is exceptional, the minimal triangulated subcategory $\langle E \rangle$ of $\cT$ containing $E$
is equivalent to~$\Db(\Bbbk)$, the bounded derived category of $\Bbbk$-vector spaces,
via the functor $\Db(\Bbbk) \to \cT$ that takes a graded vector space $W$ to $W \otimes E \in \cT$.

\begin{definition}
A sequence of objects $E_{1},\ldots, E_{m}$ in  $\mathcal{T}$ is an \textbf{exceptional collection}
if each $E_{i}$ is exceptional and $\mathrm{Ext}^{\bullet}(E_{i}, E_{j} ) = 0$ for all $i > j$.
A collection $(E_{1},E_{2},\ldots ,E_{m})$ is \textbf{full} if the minimal triangulated subcategory of $\mathcal{T}$
containing $(E_{1},E_{2},\ldots, E_{m})$ coincides with $\mathcal{T}$.
\end{definition}
A full exceptional collection in $\cT$ is a special case of a semiorthogonal decomposition of $\cT$ with~$\mathcal{A}_k=\langle E_{k}\rangle$.

For an exceptional object $E \in \cT$ we denote by $\LL_E$ and $\RR_E$ the \textbf{left} and \textbf{right mutation} functors through $E$,
which are defined by taking an object $G \in \cT$ to
\begin{equation} \label{definition:mutation}
\mathbb{L}_{E}(G):=\mathrm{Cone}(\mathrm{Ext}^{\bullet}(E, G)\otimes E \to G),
\qquad\text{and}\qquad
\mathbb{R}_{E}(G):= \mathrm{Cone}(G \to \mathrm{Ext}^{\bullet}(G,E)^{\vee} \otimes E)[-1],
\end{equation}
where the morphisms are given by the evaluation and coevaluation, respectively.

It is well known (see \cite{3}) that if $(E,E')$ is an exceptional pair then $(E',\mathbb{R}_{E'}(E))$ and $(\mathbb{L}_{E}(E'),E)$
are also exceptional pairs each of which generates the same subcategory in $\mathcal{T}$ as the initial pair $(E,E').$

More generally, if $(E_1, \ldots , E_m)$ is an exceptional collection of arbitrary length in $\mathcal{T}$ then
one can define the left and right mutations of an object through the category $\langle E_1, \ldots , E_n\rangle$
as the compositions of the corresponding mutations through the generating objects:
\begin{equation*}
\LL_{\langle E_1, \ldots , E_m\rangle} = \LL_{E_1} \circ \ldots \circ \LL_{E_m},
\qquad
\RR_{\langle E_1, \ldots , E_m\rangle} = \RR_{E_m} \circ \ldots \circ \RR_{E_1}.
\end{equation*}
\begin{proposition}[\cite{3}]
\label{proposition:mutations}
The functors of left and right mutations through an exceptional collection
induce mutually inverse equivalences of the left and the right orthogonals to the collection:
\begin{equation*}
\xymatrix@1@C=7em{
{}^\perp\langle E_1, \ldots , E_m\rangle \ar@<.5ex>[r]^{\LL_{\langle E_1, \ldots , E_m\rangle}} &
\langle E_1, \ldots , E_m\rangle^\perp \ar@<.5ex>[l]^{\RR_{\langle E_1, \ldots , E_m\rangle}}
}
\end{equation*}
Mutation of a \textup(full\textup) exceptional collection is a \textup(full\textup) exceptional collection.
\end{proposition}

Let $X$ be a smooth projective variety over a field $\Bbbk$.
We denote by $\mathrm{D}^{b}(X)$ the bounded derived category of coherent sheaves on $X$.
The following result is useful when dealing with exceptional collections on $X$.

\begin{proposition}
\label{proposition:long-mutations}
If $(E_1, E_2, \ldots, E_{m-1}, E_m)$ is an exceptional collection in $\mathrm{D}^{b}(X)$ then
\begin{equation*}
(E_2, \ldots, E_{m-1}, E_m, E_1 \otimes \omega_X^{-1})
\qquad\text{and}\qquad
(E_m \otimes \omega_X, E_1, E_2, \ldots, E_{m-1})
\end{equation*}
are exceptional collections too.
If one of these collections is full then so are the others.
\end{proposition}
\begin{proof}
The first part follows easily from Serre duality.
The second part is~\cite[Theorem 4.1]{3}.
\end{proof}

The next definition is also quite useful.
\begin{definition}[\cite{8,2,KS}]
\label{def:lefschetz}
Let $\mathcal{O}_{X}(1)$ be an ample line bundle on $X$.
\begin{enumerate}
    \item A \textbf{Lefschetz collection} in $\mathrm{D}^{b}(X)$ with respect to
$\cO_{X}(1)$ is an exceptional collection of objects of $\mathrm{D}^{b}(X)$ which has a block structure
\begin{equation*}
\underbrace{E_{1}, E_{2},\ldots ,E_{\vartheta_{0}}}_{\mbox{block $1$}},
\underbrace{E_{1}(1), E_{2}(1),\ldots,E_{\vartheta_{1}}(1)}_{\mbox{block $2$}},\ldots,
\underbrace{E_{1}(i-1), E_{2}(  i-1),\ldots,E_{\vartheta_{i-1}}( i-1)}_{\mbox{block $i$}},
\end{equation*}
where $\vartheta= (\vartheta_0 \ge \vartheta_1 \ge \ldots \ge \vartheta_{i-1} > 0)$ is a nonincreasing sequence of positive integers that is called the \textbf{support partition} of the Lefschetz collection.
\item The \textbf{rectangular part} of the Lefschetz collection is the subcollection
    $$\underbrace{E_{1}, E_{2},\ldots ,E_{\vartheta_{i-1}}}_{\mbox{block $1$}},
\underbrace{E_{1}(1), E_{2}(1),\ldots,E_{\vartheta_{i-1}}(1)}_{\mbox{block $2$}},\ldots,
\underbrace{E_{1}(i-1), E_{2}(  i-1),\ldots,E_{\vartheta_{i-1}}( i-1)}_{\mbox{block $i$}}.$$
\item The subcategory of $\mathrm{D}^{b}(X)$ orthogonal to the rectangular part of a given Lefschetz
collection is called its \textbf{residual category}:
\begin{multline*}
    \mathsf{Res}\coloneqq \langle E_{1}, E_{2},\ldots,E_{\vartheta_{i-1}},E_{1}(1), E_{2}(1),\ldots,E_{\vartheta_{i-1}}(1),\ldots,E_{1}(i-1), E_{2}(i-1),\ldots,E_{\vartheta_{i-1}}(i-1) \rangle^{\perp}.
\end{multline*}
\end{enumerate}
\end{definition}
We should point out that being Lefschetz is not a property of an exceptional collection, but rather a structure expressed as the block decomposition.

The next example will be used for the proof of Theorem \ref{Theorem:introduction:main}. Let $W_6$ be a symplectic vector space of dimension $6$ and $\mathrm{IGr}(3,W_6)$ be the isotropic Grassmannian of three-dimensional subspaces in $W_6$ (these are varieties $\CG_f$).
\begin{example}[\cite{S01}, Theorem 2.3]
The bounded derived category of a smooth hyperplane section of~$\mathrm{IGr}(3,W_6)$ has the following full Lefschetz exceptional collection
\begin{equation}  \label{exceptional_collection_hyperplane}
    \underbrace{\cO,\cU_{3}^{\vee}}_{\mbox{block $1$}},\underbrace{\cO(1),\cU_{3}^{\vee}(1)}_{\mbox{block $2$}},\underbrace{\cO(2),\cU_{3}^{\vee}(2)}_{\mbox{block $3$}}.
\end{equation}
\end{example}

\subsection{The algebraic group $\rG_2$} \label{section:G2}
Recall some basic facts about the simple algebraic group $\rG_2$. We use \cite[Lecture 22]{FH} as a convenient reference about it.

The algebraic group $\rG_2$ is a simple simply connected group of rank 2
and dimension 14. We denote by $\mathfrak{g}_2$ its Lie algebra. The root system of $\mathfrak{g}_2$ has the form:
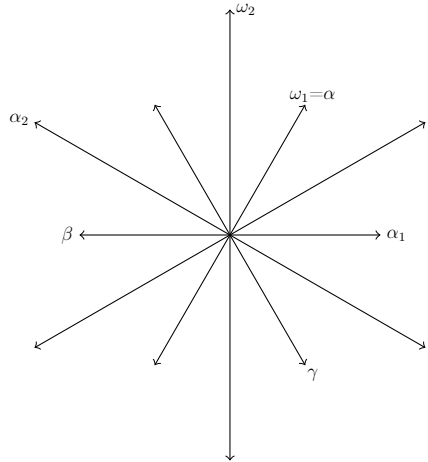
\begin{center}
\begin{equation}  \label{root_g2}
\begin{tikzpicture} 
    \foreach\ang in {60,120,...,360}{
     \draw[->] (0,0) -- (\ang:2cm);
    }
    \foreach\ang in {30,90,...,330}{
     \draw[->] (0,0) -- (\ang:3cm);
    }
    \node[anchor=west,scale=0.6] at (2,0) {$\alpha_1$};
     \node[anchor=north,scale=0.6] at (-2.8,1.7) {$\alpha_2$};
     \node[anchor=north,scale=0.6] at (1.1,2) {$\omega_1$=$\alpha$};
     \node[anchor=east,scale=0.6] at (-2,0) {$\beta$};
     \node[anchor=north,scale=0.6] at (1.1,-1.7) {$\gamma$};
      \node[anchor=west,scale=0.6] at (0,3) {$\omega_2$};
      \end{tikzpicture}
      \end{equation}
  \end{center}
where $\alpha_1$ and $\alpha_2$ are the simple roots, $\omega_1=2\alpha_1+\alpha_2$ and $\omega_2=3\alpha_1+2\alpha_2$ are the fundamental weights. We will denote the six short roots of $\rG_2$ as $\pm\alpha, \pm\beta , \pm\gamma$  with $\alpha+\beta+\gamma=0$, in order to emphazise the natural $\mathbf{S}_3$-symmetry.

Let us consider first the fundamental representation of $\rG_2$ with highest weight $\omega_1$. Its weights are the six short roots $ \{\pm \alpha, \pm \beta, \pm \gamma\}$ and zero, all with multiplicity one, so this fundamental representation has dimension 7, we will denote it by $V$. Consider the weight decomposition
\begin{equation} \label{weight_decomposition}
V=\Bbbk e_0\oplus \Bbbk e_{\alpha}\oplus \Bbbk e_{-\alpha}\oplus \Bbbk e_{\beta}\oplus \Bbbk e_{-\beta}\oplus 
 \Bbbk e_{\gamma}\oplus \Bbbk e_{-\gamma},
\end{equation}
or, dually 
\begin{equation} \label{weight_decomposition_dual}
    V^{\vee}=\Bbbk e^{\vee}_0\oplus \Bbbk e^{\vee}_{\alpha}\oplus \Bbbk e^{\vee}_{-\alpha}\oplus \Bbbk e^{\vee}_{\beta}\oplus \Bbbk e^{\vee}_{-\beta}\oplus \Bbbk e^{\vee}_{\gamma}\oplus \Bbbk e^{\vee}_{-\gamma}.
\end{equation}
It is a well known fact, that the action of $\rG_2$ on $V$ preserves a skew-symmetric 4-form $$\lambda\in \Lambda^4V^{\vee},$$
its dual (in the sense of Section \ref{notations}) 3-vector $$\lambda^{\vee}\in \Lambda^3V,$$
and a nondegenerate quadratic form $$q\in S^2V^{\vee}.$$
In particular, $\rG_2\subset \mathrm{SO}(V,q)$.
In fact, $\rG_2$ can be defined as the stabilizer in $\mathrm{GL}(V)$ of a general 4-form. 

All this forms are canonical up to rescaling. We choose a basis \eqref{weight_decomposition_dual} in a such a way that the 4-form~$\lambda$ and the quadratic form~$q$ have the following expressions
\begin{equation} \label{lambda}
    \lambda=2e_{0}^{\vee}\wedge e_{\alpha}^{\vee}\wedge e_{\beta}^{\vee}\wedge e_{\gamma}^{\vee}-2e_{0}^{\vee}\wedge e_{-\alpha}^{\vee}\wedge e_{-\beta}^{\vee}\wedge e_{-\gamma}^{\vee}+e_{\beta}^{\vee}\wedge e_{-\beta}^{\vee}\wedge e_{\gamma}^{\vee}\wedge e_{-\gamma}^{\vee}+e_{\alpha}^{\vee}\wedge e_{-\alpha}^{\vee}\wedge e_{\gamma}^{\vee}\wedge e_{-\gamma}^{\vee}+e_{\alpha}^{\vee}\wedge e_{-\alpha}^{\vee}\wedge e_{\beta}^{\vee}\wedge e_{-\beta}^{\vee},
\end{equation}
\begin{equation} \label{q:equation}
    q=e_{0}^{\vee}e_{0}^{\vee}-e_{\alpha}^{\vee} e_{-\alpha}^{\vee}-e_{\beta}^{\vee} e_{-\beta}^{\vee}-e_{\gamma}^{\vee} e_{-\gamma}^{\vee},
\end{equation}
Note that $q$ defines the canonical identification $$q\colon V\simeq V^{\vee}.$$
In particular, $q$ induces an isomorphism $\Lambda^i V\simeq \Lambda^i V^{\vee}$.  Let us denote by~$\nu\in \Lambda^3V^{\vee}$ a skew-symmetric 3-form that is preserved by~$\rG_2$
$$\nu\coloneqq q(\lambda^{\vee})\in\Lambda^3V^{\vee}.$$
In the chosen basis the 3-form $\nu$ has the following expression:
\begin{equation}
    \nu=e_{0}^{\vee} \wedge e_{\alpha}^{\vee} \wedge e_{-\alpha}^{\vee}+e_{0}^{\vee} \wedge e_{\beta}^{\vee} \wedge e_{-\beta}^{\vee}+e_{0}^{\vee} \wedge e_{\gamma}^{\vee} \wedge e_{-\gamma}^{\vee}+e_{\alpha}^{\vee} \wedge e_{\beta}^{\vee} \wedge e_{\gamma}^{\vee}+e_{-\alpha}^{\vee} \wedge e_{-\beta}^{\vee} \wedge e_{-\gamma}^{\vee}.
\end{equation}

The second fundamental representation of $\rG_2$ that corresponds to the weight $\omega_2$ is the adjoint representation. We will denote it by $V_{14}$. 

We have the following direct sum decompositions of $\rG_2$-representations, see \cite[Table 5]{OV}:
\begin{gather} \label{V14}
 \Lambda^2 V= V\oplus V_{14};\\
\label{decomposition1}
 S^2V=V_{27}\oplus \Bbbk;\\ \label{decomposition2}
 \Lambda^3V = V_{27} \oplus \Bbbk \oplus V = S^2V \oplus V ,
\end{gather}
where we denote by $V_{n}$ the irreducible $n$-dimensional representation of $\rG_2$ when this notation is unambiguous. The 1-dimensional representation $\Bbbk$ in \eqref{decomposition1} corresponds to $q^{-1}\in S^2V$ and in \eqref{decomposition2} to $\lambda^{\vee}\in \Lambda^3 V$.
Note also that the embedding $ V\simeq V^{\vee}\hookrightarrow \Lambda^2V$ in \eqref{V14} is given by the 3-vector $\lambda^{\vee}$.

Let us recall that the two minimal compact homogeneous spaces of $\rG_2$ can be described as the highest weight vector orbits in the projectivizations of the fundamental representations, they are unique closed orbits for the action.

In the projectivization $\PP(V)$ the action of $\rG_2$ has only two orbits: the closed one is a smooth 5-dimensional quadric with equation $q$, we denote it
$$\qQ\subset \PP(V).$$
We denote the minimal compact homogeneous space for the adjoint representation by
$$\Gad \subset \PP(\mathfrak{g}_2) =\PP(V_{14}).$$
The homogeneous variety $\Gad$ is called the \textbf{adjoint variety} of $\rG_2$. 

Let us recall a more geometric description of $\Gad$.  
By \eqref{V14} the adjoint representation $V_{14}$ is contained in $\Lambda^2V$, so that $\Gad\subset \PP(\Lambda^2V)$. Since  $\Gr(2,V)\subset \PP(\Lambda^2 V)$ is closed and $\rG_2$-invariant in $ \PP(\Lambda^2 V)$, it follows that $\Gad\subset \Gr(2,V)\cap \PP(V_{14})\subset \PP(\Lambda^2 V)$. Denote by $\cO(H_2)$ the restriction from $\Gr(2,V)$ to $\Gad$ of the Pl{\"u}cker line bundle. In what follows we will use the following well-known properties of $\Gad$. 
\begin{lemma}[\cite{LM}, Theorem 3.1] \label{adjoint_properties}
The adjoint variety $\Gad\subset \Gr(2,V)$ is the zero locus of the global section $\nu \in \mathrm{H}^0(\Gr(2,V),\cU_{2}^{\perp}(1))$, and $\Gad$ is a linear section of $\Gr(2,V)$, that is
\begin{equation*}
   \Gad=\Gr(2,V)\cap \PP(V_{14}).
\end{equation*} 
We have $\Gad\subset \mathrm{OGr}_{q}(2,V)$, the Grassmannian of $q$-isotropic planes in $V$. The Picard group of $\Gad$ is generated by $\cO(H_2)$, i.e. $\mathrm{Pic}(\Gad)=\mathbb{Z}\cdot \cO(H_2)$. The canonical line bundle of $\Gad$ is $\omega_{\Gad}\simeq \cO(-3H_2)$.
\end{lemma}
Denote by $\cB$ a Borel subgroup of $\rG_2$ and by $\cP_1$, $\cP_2$ the parabolic subgroups that contain $\cB$ and such that
\begin{equation*}
    \qQ\simeq \rG_2/\cP_2; \qquad \Gad \simeq \rG_2/\cP_1.
\end{equation*}
Then we have the following diagram 
\begin{equation} \label{g2}
\xymatrix{
& \rG_2/\cB\ar[dl]_{\PP^1} \ar[dr]^{\PP^1}\\ 
\qQ\simeq \rG_2/\cP_2 &&  \rG_2/\cP_1\simeq \Gad.
}
\end{equation}
\subsection{Conics on Grassmannians} \label{classes_of_conics}
Let $W$ be a vector space. Recall, see \cite{IM}, that conics in $\Gr(2,W)$ can be divided into three different
classes, according to the type of their linear span in~$\PP(\Lambda^2W)$:
\begin{itemize}
    \item \textbf{$\tau$-conics} are conics spanning a plane in $\PP(\Lambda^2 W)$ which is not contained in $\Gr(2,W)$;
    \item \textbf{$\sigma$-conics} are conics spanning a plane in $\Gr(2,W)$ that corresponds to a point of $\mathrm{Fl}(1,4;W)$;
    \item \textbf{$\rho$-conics} are conics spanning a plane in $\Gr(2,W)$ that corresponds to a point of $\Gr(3,W)$.
\end{itemize}
\section{Quadric bundles constructions.} \label{section:quadric_constructions}
Let $\hQ \to S$ be a \textbf{quadric bundle} over a scheme~$S$, that is a proper morphism which can be presented as a composition $\hQ \hookrightarrow \PP_{S}(\cF) \to S$, where $\PP_{S}(\cF) \to S$ is the projectivization of a vector bundle $\cF$ and~$\hQ \to \PP_{S}(\cF)$ is a divisorial embedding of relative degree~2 over~$S$. A quadric bundle is determined by a quadratic form $S^2 \cF \to \cL$ with values in a line bundle~$\cL$, or, equivalently, by a self-dual morphism
$$\cF\stackrel{f}{\rightarrow} \cF^{\vee}\otimes \cL.$$
With a quadric bundle one associates the coherent sheaf
$$\mathrm{Coker}(\cF\stackrel{f}{\rightarrow} \cF^{\vee}\otimes \cL)$$
on $S$, which we call its \textbf{cokernel sheaf}. Sometimes we will also denote the cokernel sheaf of $\bQ$ by~$\mathrm{Coker}(\bQ)$.
\subsection{Two general constructions} The constructions that we describe in the following two propositions are crucial for our proof of Theorem \ref{Theorem:introduction:main}.
\begin{proposition} \label{lemma:quadric_bundles_cokernel}
There is a bijection between the set of isomorphism classes of triples
\begin{equation*}
  \left\{ (f_1, f_2, g) \mid f_1 \colon \cF_1\to \cF^{\vee}_1\otimes \cL,\quad f_2\colon \cF_2\to \cF^{\vee}_2\otimes \cL \quad \mbox{and} \quad g\colon \mathrm{Coker}(f_1)\xrightarrow{\sim} \mathrm{Coker}(f_2) \right\},
\end{equation*}
where $f_1$ and $f_2$ are self-dual, and the set of pairs
\begin{equation*}
  \left\{(f, \epsilon)\mid f\colon \cF\simeq \cF^{\vee}\otimes \cL, \quad \epsilon\colon 0\to \cF_1 \to \cF\to \cF^{\vee}_2\otimes \cL\to 0 \right\},
\end{equation*}
where $f$ is self-dual, which maps $(f_1, f_2, g)$ to $(f, \epsilon)$.
\end{proposition}
\begin{proof}
Suppose that we have a pair $(f, \epsilon)$. Denote by $i_1\colon \cF_1 \to \cF$ the embedding.
Consider the dual exact sequence of $\epsilon$ (twisted by $\cL$)
\begin{equation} \label{exact_sequence_F_dual}
  0\to \cF_2 \xrightarrow{i_2} \cF^{\vee}\otimes \cL \xrightarrow{i_1^{\vee}} \cF^{\vee}_1\otimes \cL\to 0.  
\end{equation}
We define self-dual maps $f_1$ and $f_2$ as compositions
\begin{gather*}
f_1\colon  \cF_1 \stackrel{i_1}{\hookrightarrow} \cF \stackrel{f}{\rightarrow} \cF^{\vee}\otimes \cL  \stackrel{i_1^{\vee}}{\twoheadrightarrow} \cF_1^{\vee}\otimes \cL;\\
f_2\colon  \cF_2 \stackrel{i_2}{\hookrightarrow} \cF^{\vee}\otimes \cL \stackrel{f^{-1}}{\rightarrow} \cF \stackrel{i_2^{\vee}}{\twoheadrightarrow} \cF_2^{\vee}\otimes \cL.
\end{gather*}
Let us prove that we have an isomorphism of cokernel sheaves $g\colon \mathrm{Coker}(f_1)\xrightarrow{\sim} \mathrm{Coker}(f_2)$.
Indeed, we have~$\mathrm{Coker}(f\circ i_1)\simeq \cF_2^{\vee}\otimes \cL$ and the cokernel map is $\cF^{\vee}\otimes \cL\stackrel{f^{-1}}{\rightarrow}\cF \xrightarrow{i_2^{\vee}} \cF_2^{\vee}\otimes \cL$. Thus, we obtain
\begin{multline*}
\mathrm{Coker}(f_1) \simeq \mathrm{Coker}(\cF_1\stackrel{f\circ i_1}{\hookrightarrow} \cF^{\vee}\otimes \cL \xrightarrow{i_1^{\vee}} \cF_1^{\vee}\otimes \cL) \simeq \\
\simeq \mathrm{Coker}(\cF_1\oplus \cF_2 \xrightarrow{(f\circ i_1, i_2)}\cF^{\vee}\otimes \cL)\simeq \\
\simeq \mathrm{Coker}(\cF_2\stackrel{i_2}{\hookrightarrow} \cF^{\vee}\otimes \cL \stackrel{f^{-1}}{\rightarrow}\cF \xrightarrow{i_2^{\vee}} \cF_2^{\vee}\otimes \cL)\simeq \mathrm{Coker}(f_2).
\end{multline*}

Suppose now that we have a triple $(f_1, f_2, g)$. We have two exact sequences
\begin{gather} \label{es:f_1}
  0\to  \cF_1\stackrel{f_1}{\longrightarrow} \cF^{\vee}_1\otimes \cL \stackrel{g_1}{\longrightarrow} \mathrm{Coker}(f_1) \to 0\\ \label{es:f_2}
 0\to \cF_2\stackrel{f_2}{\longrightarrow} \cF^{\vee}_2\otimes \cL \stackrel{g_2}{\longrightarrow} \mathrm{Coker}(f_1) \to 0.
\end{gather}
We define $\cF$ from the exact sequence
\begin{equation} \label{ec:f}
    0\to \cF \to \cF^{\vee}_1\otimes \cL \oplus \cF^{\vee}_2\otimes \cL \stackrel{(g_1,g_2)}{\longrightarrow} \mathrm{Coker}(f_1) \to 0.
\end{equation}
From the exact sequences \eqref{es:f_1} and \eqref{ec:f} we obtain the exact sequence
\begin{equation*}
    0\to \cF_1 \to \cF \to \cF^{\vee}_2\otimes \cL \to 0,
\end{equation*}
so that $\cF$ is an extension of $\cF^{\vee}_2\otimes \cL$ by $\cF_1$. Using the exact sequence that is the direct sum of \eqref{es:f_1} and~\eqref{es:f_2} and the exact sequence \eqref{ec:f} we obtain the following exact sequence for $\cF$
\begin{equation} \label{ec:F_dual}
 0\to \cF_1\oplus \cF_2 \to \cF \to \mathrm{Coker}(f_1)\to 0.
\end{equation}
Using the exact sequences \eqref{ec:f} and \eqref{ec:F_dual} we conclude that $\cF$ is self-dual
\begin{equation*}
    f\colon \cF\xrightarrow{\sim} \cF^{\vee}\otimes \cL.
\end{equation*}
The constructions are mutually inverse, so we obtain the statement of the proposition. 
\end{proof}
Let $i\colon D\hookrightarrow S$ be a Cartier divisor on a scheme $S$, let $\xi\colon \cO\to \cO(D)$ be a section whose zero locus is $D$ and let $\xi^{\vee}\colon \cO(-D)\to \cO$ be the dual map to $\xi$.
\begin{proposition} \label{quadric_construction}
Let $f\colon \cF\otimes \cL^{\vee}\to \cF^{\vee}$ be a quadric bundle on $S$. Suppose that the scheme support of the cokernel sheaf of $f$ is $D$, i.e. $\mathrm{Coker}(f)\simeq i_*\cC(f)$ for a coherent sheaf $\cC(f)$ on $D$. Then there exists a quadric bundle on $S$
$$f'\colon \cF^{\vee}\otimes \cO(-D) \to \cF\otimes \cL^{\vee},$$
such that the composition $\cF^{\vee}\otimes \cO(-D)\xrightarrow{f'} \cF\otimes \cL^{\vee}\xrightarrow{f}\cF^{\vee}$ is equal to $id\otimes\xi^{\vee}$, and the composition $\cF\otimes \cL^{\vee}\xrightarrow{f}\cF^{\vee} \xrightarrow{f'} \cF\otimes \cL^{\vee}\otimes \cO(D)$ is equal to $id\otimes \xi$. Moreover,~$\mathrm{Coker}(f')\simeq i_*\cC(f')$ and $\cC(f')\simeq \mathrm{Im}(f\restr{D})$,~$\cC(f)\simeq \mathrm{Im}(f'\restr{D})\otimes \cO_{D}(D)$, and $\cC(f')$ and $\cC(f)$ admit self-dual isomorhisms on~$D$:
\begin{align*}
  & \cC(f') \simeq \cC(f')^{\vee} \otimes i^*\cL^{\vee}\\
  & \cC(f) \simeq \cC(f)^{\vee} \otimes i^*\cO(D) \otimes i^*\cL^{\vee}.
\end{align*}
In particular, the sheaves $\cC(f)$ and $\cC(f')$ are reflexive.
\end{proposition}
\begin{proof}
Consider the following diagram 
\begin{equation} \label{diagram_cF}
\xymatrix@R=1pc{
0\ar[r]& \cF^{\vee}\otimes \cO(-D)\ar[r]^(.7){id\otimes\xi^{\vee}} \ar@{-->}[d]_{f'} & \cF^{\vee}\ar@{=}[d] \ar[r]& i_*i^*\cF^{\vee} \ar[d] \ar[r] & 0 \\
0\ar[r]& \cF\otimes \cL^{\vee}\ar[r]^(.55){f} & \cF^{\vee} \ar[r] & i_*\cC(f) \ar[r] & 0,
}   
\end{equation}
where the right vertical arrow comes from adjunction. We define a self-dual map
$$f'\colon \cF^{\vee}\otimes \cO(-D)\to \cF\otimes \cL^{\vee}$$
from \eqref{diagram_cF}. Using the snake lemma for \eqref{diagram_cF} we obtain that $\mathrm{Coker}(f')=i_*\cC(f')$, where~$\cC(f')=\mathrm{Ker}(i^*\cF^{\vee}\to \cC(f))$. Restricting the bottom row to $D$ we obtain $\cC(f')\simeq \mathrm{Im}(f\restr{D})$, so that $\tcC(f') \simeq \tcC(f')^{\vee} \otimes i^*\cL^{\vee}$ is self-dual.

Note that the construction described above is invertible, i.e. if we apply it to $f'$ we obtain the self-dual morphism $f$ twisted by $\cO(-D)$. In particular, we deduce the isomorphisms~$\cC(f)\simeq \mathrm{Im}(f'\restr{D})\otimes \cO_{D}(D)$ and $\cC(f) \simeq \cC(f)^{\vee} \otimes i^*\cO(D) \otimes i^*\cL^{\vee}$. So we get the statement. 
\end{proof}
\subsection{Quadrics containing the Veronese surface}
Recall the well-known geometric construction that we will use later. Consider a 3-dimensional space~$U_3$ and the Veronese embedding $\nu_2\colon \PP(U_3)\hookrightarrow \PP(S^2U_3)$. A quadratic polynomial on the projective space $\PP(S^2U_3)$ will restrict to a quartic polynomial on the Veronese surface~$\nu_2(\PP(U_3))$, so that the kernel of the evaluation map 
$$S^2(S^2U_3^{\vee})\to S^4U_3^{\vee}$$
is the vector space of quadratic polynomials in $\PP(S^2U_3)$ vanishing on the Veronese surface $\nu_2(\PP(U_3))$. A decomposition (see for example \cite[Section 13.3]{FH})
\begin{equation*}
  S^2(S^2 U^{\vee}_3)=S^4 U_3^{\vee}\oplus S^2U_3 
\end{equation*}
shows that $\mathrm{Ker}(S^2(S^2U_3^{\vee})\to S^4U_3^{\vee})\simeq S^2U_3$, so we obtain that the space of quadrics in~$\PP(S^2 U_3)$ containing the Veronese surface $\nu_2(\PP(U_3))$ is isomorphic to~$\PP(S^2 U_3)$, the space of conics in $\PP(U^{\vee}_3)$. For an equation~$f\in S^2 U_3$ we will denote by $C_f\subset \PP(U^{\vee}_3)$ and $Q_f\subset \PP(S^2 U_3)$ the corresponding conic and quadric.
\begin{proposition} \cite[Section 13.4]{FH} \label{proposition_veronese}
Under the identification described above a smooth conic $C_f\subset \PP(U^{\vee}_3)$ correspond to smooth quadric $Q_f\subset \PP(S^2 U_3)$ and a conic $C_{f'}\subset \PP(U^{\vee}_3)$ of rank 2 corresponds to a quadric~$Q_{f'}\subset \PP(S^2 U_3)$ of rank 4. Moreover, for $f'= u\otimes v +v\otimes u $ the vertex of the quadric cone $Q_{f'}$ is equal to $\langle u\otimes u,v\otimes v \rangle\subset  \PP(S^2 U_3)$.
\end{proposition}
We also discuss a family version of the above construction. Consider a vector bundle $\cF_3$ of rank 3 on a scheme $S$ and the Veronese embedding $\nu_2\colon \PP_{S}(\cF_3)\hookrightarrow \PP_{S}(S^2\cF_3)$. We have
$$\mathrm{Ker}(S^2(S^2\cF_3^{\vee})\otimes \cL \to S^4\cF_3^{\vee}\otimes\cL)\simeq S^2\cF_3\otimes \mathrm{det}(\cF^{\vee})^{\otimes 2}\otimes \cL.$$
\begin{corollary} \label{proposition_veronese_family}
There is a bijection between the set of conic bundles in $\PP_{S}(\cF_3^{\vee})$
$$ g \colon \cF_3^{\vee}\otimes \mathrm{det}(\cF)^{\otimes 2}\otimes \cL^{\vee}\to \cF_3 $$
and the set of quadric bundles in $\PP_{S}(S^2\cF_3)$
$$\tg\colon S^2\cF_3\to  S^2\cF_3^{\vee}\otimes \cL$$
containing $\nu_2(\PP_{S}(\cF_3))\subset \PP_{S}(S^2\cF_3)$. The set-theoretic supports of cokernel sheaves $\mathrm{Coker}(g)$ and~$\mathrm{Coker}(\tg)$ coincide. Moreover, over the locus where the fibers of $g$ are reducible conics the fibers of~$\tg$ have rank 4. 
\end{corollary}
Denote by $S'\subset S$ the locus where $g$ has rank 2 and by $\cF_2\coloneqq \mathrm{Im}(g)\restr{S'}$ the 2-dimensional image of~$g$ over $S'$. Over $S'$ the conic bundle $g$ induces a self-dual isomorphism
$$ \cF^{\vee}_2\otimes  \mathrm{det}(\cF_3)^{\otimes 2}\otimes \cL^{\vee}\to \cF_2,$$
so that over $S'$ we have a decomposition $S^2\cF_2 =  (\mathrm{det}(\cF_3)^{\otimes 2}\otimes \cL^{\vee}) \oplus (S^2\cF_2/\mathrm{det}(\cF_3)^{\otimes 2}\otimes \cL^{\vee})$.
\begin{corollary} \label{kernel_veronese}
The kernel of $\tg\restr{S'}$ is isomorphic to $S^2\cF_2/\mathrm{det}(\cF_3)^{\otimes 2}\otimes \cL^{\vee}$.
\end{corollary}
\subsection{Quadrics containing the Segre variety}
Consider now a 4-dimensional space~$U_4$ and the Segre embedding $\sigma\colon \PP(U_3)\times \PP(U_4) \hookrightarrow \PP(U_3\otimes U_4)$. The kernel of the evaluation map 
$$S^2(U_3\otimes U_4)^{\vee}\to S^2U_3^{\vee}\otimes S^2U_4^{\vee}$$
is the vector space of quadratic polynomials in $\PP(U_3\otimes U_4)$ vanishing on the Segre variety $\PP(U_3)\times \PP(U_4)$. A decomposition \cite[Corollary 2.3.3]{5}
\begin{equation}
  S^2(U_3\otimes U_4)^{\vee}= S^2U_3^{\vee}\otimes S^2U_4^{\vee}\oplus \Lambda^2U_3^{\vee}\otimes \Lambda^2U_4^{\vee}
\end{equation}
shows that $\mathrm{Ker}(S^2(U_3\otimes U_4)^{\vee}\to S^2U_3^{\vee}\otimes S^2U_4^{\vee})\simeq \Lambda^2U_3^{\vee}\otimes \Lambda^2U_4^{\vee}$, so the space of quadrics in~$\PP(U_3\otimes U_4)$ containing the Segre variety $\PP(U_3)\times \PP(U_4)$ is isomorphic to~$\PP(\Lambda^2U_3^{\vee}\otimes \Lambda^2U_4^{\vee})$. We will denote by~$Q_{\varsigma}\subset \PP(U_3\otimes U_4)$ the quadric corresponding to $\varsigma\in \Lambda^2U_3^{\vee}\otimes \Lambda^2U_4^{\vee}$.

Note that an element $\varsigma\in \Lambda^2U_3^{\vee}\otimes \Lambda^2U_4^{\vee}$ determines a self-dual map
\begin{equation*}
  \Lambda^2U_3\xrightarrow{\varsigma} \Lambda^2U_4^{\vee} \simeq \Lambda^2U_4 \xrightarrow{\varsigma}  \Lambda^2U^{\vee}_3.
\end{equation*}
We will denote by $C_{\varsigma}\subset \PP(\Lambda^2U_3)$ the corresponding conic. 

\begin{proposition} \label{proposition_quadrics_segre}
Suppose $\varsigma\in \Lambda^2U_3^{\vee}\otimes \Lambda^2U_4^{\vee}$ determines an embedding $\Lambda^2U_3\to \Lambda^2U_4^{\vee}$ such that~$\varsigma(\PP(\Lambda^2U_3))$ is not contained in $\Gr(2,U_4)$. Then
\begin{itemize}
    \item if $C_{\varsigma}$ is smooth then $Q_{\varsigma}$ is smooth
    \item if $C_{\varsigma}$ is reducible then $Q_{\varsigma}$ has rank 10.
\end{itemize}
\end{proposition}
\begin{proof}
Recall, see \cite[Lemma 3.9]{CC}, that the Hilbert scheme of conics $\mathrm{Hilb}^{1+2t}(\Gr(2,U^{\vee}_4))$ on the Grassmannian $\Gr(2,U^{\vee}_4)$ is ismorphic to
$$\mathrm{Bl}_{\PP(U_4) \sqcup \PP(U_4^{\vee})}\Gr(3, \Lambda^2U_4^{\vee}),$$
where  $\PP(U_4^{\vee})$ corresponds to $\sigma$-conics and $\PP(U_4)$ corresponds to $\rho$-conics, see Subsection \ref{classes_of_conics} for the notation. 
Note that $\Gr(3, \Lambda^2U_4^{\vee})$ is isomorphic to
the quotient of the open dense locus in $\mathrm{Hom}(\Lambda^2U_3, \Lambda^2U_4^{\vee})$ parametrizing embeddings by the action of $\mathrm{GL}(U_3)$. So we obtain that under the suppositions of the proposition $\varsigma$ corresponds to a $\tau$-conic on $\Gr(2,U^{\vee}_4)$. We conclude that smooth $\tau$-conics on $\Gr(2,U^{\vee}_4)$ are identified with the open dense orbit of the action of $\mathrm{GL}(U_3)\times \mathrm{GL}(U_4)$ on $\Lambda^2U^{\vee}_3\otimes \Lambda^2U_4^{\vee}$ and reducible $\tau$-conics on $\Gr(2,U^{\vee}_4)$ are identified with the divisorial orbit of the action of $\mathrm{GL}(U_3)\times \mathrm{GL}(U_4)$ on $\Lambda^2U^{\vee}_3\otimes \Lambda^2U_4^{\vee}$.

To prove the statement it is enough to show that for a reducible $\tau$-conic $C_{\varsigma'}$ the rank of $Q_{\varsigma'}$ is equal to 10. Choose bases $(x_1,x_2,x_3)$ and $(y_1,y_2,y_3,y_4)$ of $U_3^{\vee}$ and $U_4^{\vee}$ respectively, so that $z_{m,n}=x_{m}\otimes y_{n}$ is a basis of $U_3^{\vee}\otimes U_4^{\vee}$. The dual bases of $U_3$ and $U_4$ will be denoted by $(u_1,u_2,u_3)$ and $(v_1,v_2,v_3,v_4)$ respectively. By \cite[Proposition 3.2.3]{5} the embedding $\Lambda^2U_3^{\vee}\otimes \Lambda^2U_4^{\vee}\hookrightarrow  S^2(U_3\otimes U_4)^{\vee}$ sends $$(x_{m_1}\wedge x_{m_2})\otimes (y_{n_1}\wedge y_{n_2}) \to z_{m_1n_1}z_{m_1n_2}- z_{m_1n_2}z_{m_2n_1} =\mathrm{det}\begin{pmatrix}
z_{m_1n_1} & z_{m_1n_2}\\
z_{m_2n_1}& z_{m_1n_2}\\
\end{pmatrix}.$$
So we compute that the element
$$\varsigma'=(x_2\wedge x_3)\otimes (y_1\wedge y_2)-2(x_1\wedge x_2)\otimes (y_2\wedge y_3)-2(x_1\wedge x_3)\otimes (y_1\wedge y_4) \in \Lambda^2U_3^{\vee}\otimes \Lambda^2U_4^{\vee},$$
that determines reducible $\tau$-conic $C_{\varsigma'}\subset \Gr(2,U^{\vee}_4)$,
corresponds to the quadric $Q_{\varsigma'}$ given by the equation 
$$z_{21}z_{32}-z_{22}z_{31}-2z_{11}z_{34}+2z_{14}z_{31}-2z_{12}z_{23}+2z_{13}z_{22}\in S^2(U_3\otimes U_4)^{\vee},$$
so that $Q_{\varsigma'}$ has 2-dimensional kernel isomorphic to $\langle u_2\otimes v_4, u_3\otimes v_3 \rangle$.
\end{proof}
We also discuss a family version of the above construction. Consider vector bundles $\cF_3$ and $\cF_4$ of rank 3 and 4 respectively on a scheme $S$ and the Segre embedding $\sigma\colon \PP_{S}(\cF_3)\times \PP_{S}(\cF_4) \hookrightarrow \PP_{S}(\cF_3\otimes\cF_4)$. We have
$$\mathrm{Ker}(S^2(\cF_3\otimes \cF_4)^{\vee}\otimes \cL \to S^2\cF_3^{\vee}\otimes S^2\cF_4^{\vee}\otimes \cL)\simeq \Lambda^2\cF_3^{\vee}\otimes \Lambda^2\cF_4^{\vee}\otimes \cL.$$
\begin{corollary} \label{proposition:segre_family}
The projective space 
$\PP(\mathrm{Hom}(\Lambda^2\cF_3,\Lambda^2\cF_4^{\vee}\otimes \cL))$ parametrizes quadric bundles in~$\PP_{S}(\cF_3\otimes\cF_4)$
$$\cF_3\otimes\cF_4\otimes \cL^{\vee} \to  \cF^{\vee}_3\otimes\cF^{\vee}_4,$$
containing $\sigma(\PP_{S}(\cF_3)\times \PP_{S}(\cF_4)) \subset \PP_{S}(\cF_3\otimes\cF_4)$.
\end{corollary}
The quadric bundle corresponding to an element $\varsigma \in \mathrm{Hom}(\Lambda^2\cF_3,\Lambda^2\cF_4^{\vee}\otimes \cL)$ will be denoted by~$\bQ_{\varsigma}\subset \PP(\cF_3\otimes\cF_4)$.

Suppose for simplicity that $\mathrm{det}(\cF_3)=\mathrm{det}(\cF_4)=\cL$, then $\varsigma$ determines a conic bundle 
\begin{equation*}
     q_{\varsigma}\colon \Lambda^2\cF_3\xrightarrow{\varsigma} \Lambda^2\cF_4^{\vee}\otimes \cL\simeq \Lambda^2\cF_4\xrightarrow{\varsigma} \Lambda^2\cF_3^{\vee}\otimes \cL \simeq \cF_3,
\end{equation*}
that we will denote by $\bC_{\varsigma}\subset \PP_{S}(\Lambda^2\cF_3)$. Denote by $S'\subset S$ the locus where $q_{\varsigma}$ has rank 2.
\begin{corollary} \label{corollary:Segre}
Suppose $\varsigma \in \mathrm{Hom}(\Lambda^2\cF_3,\Lambda^2\cF_4)$ determines an embedding of vector bundles, such that the projective bundle $\varsigma(\PP(\Lambda^2\cF_3))$ is not contained in $\Gr(2,\cF_4)$ fiberwise. Then the set-theoretic supports of the cokernel sheaves $\mathrm{Coker}(\bC_{\varsigma})$ and $\mathrm{Coker}(\bQ_{\varsigma})$ coincide. Moreover, over the locus $S'$ the map $\cF_3\otimes\cF_4\otimes \cL^{\vee} \to \cF^{\vee}_3\otimes\cF^{\vee}_4$ determining $\bQ_{\varsigma}$ has 2-dimensional kernel.
\end{corollary}
\subsection{Auxiliary lemma}
The following lemma deals with any morphism of vector bundles of equal rank, not necessary self-dual. But since we will use the lemma only for self-dual morphisms we place it in this section. 

Given a morphism $\phi\colon \cF \to \cG$ of vector bundles on a scheme $S$ we denote by~$D_k(\phi) \subset S$
its $k$-th degeneration scheme, i.e. the subscheme of $S$ whose ideal is locally generated by all
$(r + 1 - k) \times (r + 1 - k) $ minors of a matrix of $\phi$ for $r = \text{min}\{\text{rank}(\cF), \text{rank}(\cG)\}$. This
is a closed subscheme in $S$, and~$D_{k+1}(\phi) \subset D_{k}(\phi)$.

The following observation is quite basic. Unfortunately,
we were not able to find a reference for it, so we sketch a short proof.
\begin{lemma} \label{lemma:degeneracy_locus}
Let $\phi\colon \cF \to \cG$ be a morphism of vector bundles of equal rank $n$ on a smooth connected scheme $S$. Let $D$ be a prime Cartier divisor on $S$.
Suppose that the first degeneracy locus of $\phi$ is $D_1(\phi)=2D$ and~$D\subset D_2(\phi)$, i.e. the rank of $\phi$ on $D$ is less than or equal to $n-2$. Then the scheme-theoretic support of $\mathrm{Coker}(\phi)$ is~$D$. Moreover, at the generic point of $D$ the sheaf $\mathrm{Coker}(\phi)$ has rank~2.
\end{lemma}
\begin{proof}
The question is local, so we may work over the generic point of $D$, that is $\mathrm{Spec}\ A$, where $A$ is a discrete valuation ring. In this setting~$\cF$ and~$\cG$ correspond to free modules $A^{\oplus n}$, so that $\phi$ is a matrix of size $n\times n$ with entries in $A$, and~$D$ corresponds to the maximal ideal $\cm\subset A$. Then the assumption reads as~$\cm^{2} = A\cdot \mathrm{det}(\phi)$ and $\mathrm{rank}(\tilde{\phi})=n-2$, where 
$$\tilde{\phi}\colon (A/\cm)^{\oplus n}\to (A/\cm)^{\oplus n}$$
is the map induced by $\phi$.

From the assumption $\cm^{2} = A\cdot \mathrm{det}(\phi)$ we deduce that we have two possibilities for $\mathrm{Coker}(\phi)$: either~$\mathrm{Coker}(\phi)\simeq A/\cm^2$ or $\mathrm{Coker}(\phi)\simeq A/\cm\oplus A/\cm$. From the assumption $\mathrm{rank}(\tilde{\phi})=n-2$ we get that $\mathrm{Coker}(\phi)/\cm \mathrm{Coker}(\phi)$ has rank 2. So we conclude that cokernel of $\phi$ should be isomorphic to $A/\cm\oplus A/\cm$ and hence it is annihilated by $\cm$. Since $D$ is a Cartier divisor we get the statement.
\end{proof}

\section{$\rG_2$-action} \label{section:G2action}
Recall from Subsection \ref{section:G2} that $\rG_2$ is the stabilizer of the 4-form $\lambda\in \Lambda^4V^{\vee}$ in $\mathrm{GL}(V)$, so $\rG_2$ acts on the Cayley Grassmannian~$\CG\subset \Gr(3,V)$. To prove the fullness of the exceptional collection~\eqref{collection} we will need an explicit description of orbits of this action that was given by L. Manivel. 

Recall the notation from Subsection \ref{section:G2}: the weight decomposition \eqref{weight_decomposition} of $V$, the 5-dimensional quadric $\qQ$ and the adjoint variety $\Gad\subset \Gr(2,V)$.
\begin{proposition}\cite[Proposition 3.1]{M} \label{proposition:orbits}
The action of $\rG_2$ on $\CG$ has three orbits:
\begin{enumerate}
    \item the open dense 8-dimensional orbit $\mathsf{O}_0=\rG_2\cdot P_0$, where $$[P_0]:=\langle e_0,e_{\gamma},e_{-\gamma} \rangle\in \CG;$$ this orbit parametrizes 3-dimensional subspaces in $V$ annihilated by $\lambda$ on which the restriction of $q$ is non-degenerate; for $[U_3]\in \mathsf{O}_0$ there is no $[U_2]\in \Gad$ such that $U_2\subset U_3$;
    \item the locally closed 7-dimensional orbit $\mathsf{O}_1=\rG_2\cdot P_1$, where $$[P_1]:=\langle e_{0},e_{\beta},e_{-\gamma} \rangle\in \CG;$$ this orbit parametrizes 3-dimensional subspaces annihilated by $\lambda$ on which the restriction of~$q$ has rank one; for $[U_3]\in \mathsf{O}_1$ if $U_2= \mathrm{Ker}(q|_{U_3})$ then $[U_2]\in \Gad$, and this is the only such subspace;
    \item the closed 5-dimensional orbit $\mathsf{O}_2=\rG_2\cdot P_2$, where $$[P_2]:=\langle e_{\alpha},e_{\beta},e_{-\gamma} \rangle\in \CG;$$ this orbit parametrizes $q$-isotropic 3-dimensional subspaces annihilated by $\lambda$; in a 3-dimensional subspace $U_3$ on $\mathsf{O}_2$ there is a unique line $U_1\subset U_3$ such that each 2-subspace of $U_3$ containing~$U_1$ lies on $\Gad$; the correspondence~$[U_3]\to [U_1]$ defines an isomorphism $\mathsf{O}_2\simeq \qQ$. 
\end{enumerate}
\end{proposition}
By \cite[Corollary 3.4]{M} and decomposition \eqref{decomposition1} the linear span of $\CG$ is $S^2V=V_{27}\oplus \Bbbk\subset \Lambda^3 V$. In fact, we have the following lemma.
\begin{lemma} \label{lemma:CGlinear}
The Cayley Grassmannian $\CG$ is the linear section of $\Gr(3,V)$:
\begin{equation} \label{CG_is_linear_section}
    \CG=\Gr(3,V)\cap \PP(V_{27}\oplus \Bbbk)\subset \PP(\Lambda^3V).
\end{equation}
\end{lemma}
\begin{proof}
The Cayley Grassamnnian is the zero locus of a regular section~$\lambda\in \mathrm{H}^0(\Gr(3,V),\cU^{\perp}(1))$. Therefore, on $\Gr(3,V)$ we have the following Koszul resolution of the structure sheaf~$\varrho_*\cO_{\CG}$:
\begin{equation}\label{resolution}
    0\to \cO(-3)\to \Lambda^3\cQ(-3)\to \Lambda^2\cQ(-2)\to \cQ(-1)\to \cO\to \varrho_*\cO_{\CG}\to 0,
\end{equation} 
where $\varrho\colon \CG \hookrightarrow \Gr(3,V)$ is the embedding. From the above resolution we immediately obtain that the ideal sheaf $\cI_{\CG}(1)=\mathrm{Ker}(\cO(1)\to \varrho_*\varrho^*\cO(1))$ is a quotient of the globally generated vector bundle~$\cQ$, hence it is globally generated, and therefore, $\CG$ is a linear section of $\CG$. 
\end{proof}
By \eqref{decomposition2} there exists a unique $\rG_2$-invariant hyperplane section of $\CG$, that corresponds to the canonical three-form $\nu\in \mathrm{H}^0(\CG, \cO(1))$. The invariant hyperplane section should be the union of some $\rG_2$-orbits so we get that the orbit closure $\bar{\mathsf{O}}_1=\mathsf{O}_1\cup \mathsf{O}_2$ of $\mathsf{O}_1$ is the $\rG_2$-equivariant hyperplane section of~$\CG$. Let us denote by
\begin{equation*} \label{i_embedding}
   i\colon \bar{\mathsf{O}}_1\hookrightarrow \CG
\end{equation*}
the embedding of $\bar{\mathsf{O}}_1$ in $\CG$.

By Lemma \ref{adjoint_properties} we have $\Gad\subset \mathrm{OGr}_{q}(2,V)$, so on $\Gad$ the tautological bundle $\cU_2$ via $q$ is a subbundle of $\cU^{\perp}_2$, that is, $\cU_2 \subset \cU_2^{\perp}$, and $q$ induces a nondegenerate quadratic form on $\cU^{\perp}_2/\cU_2$.
\begin{proposition} \cite[Section 3.1]{M} \label{proposition:resolution_of_Hnu}
The $\rG_2$-equivariant hyperplane section $\bar{\mathsf{O}}_1=\mathsf{O}_1\cup \mathsf{O}_2$ is singular along~$\mathsf{O}_2$, with nodal singularities. It admits the following equivariant resolution of singularities
\begin{equation} \label{diagram:resolution}
\xymatrix{
& \mathrm{Bl}_{\mathsf{O}_2}\bar{\mathsf{O}}_1  \ar[dl]_{\pi_1} \ar@{=}[r] &  \PP_{\Gad}(\cU^{\perp}_2/\cU_2)  \ar[dr]^{\pi_2}\\ 
\bar{\mathsf{O}}_1  &&& \Gad,
}
\end{equation}
where $\cU_2$ is the restriction of the tautological vector bundle on $\Gr(2,V)$ to $\Gad$.  If~$E \subset \PP_{\Gad}(\cU^{\perp}_2/\cU_2)$ is the exceptional divisor of $\pi_1$ then the induced projection $E\to \Gad$ is a $\PP^1$-bundle and it is embedded into the $\PP^2$-bundle $\PP_{\Gad}(\cU^{\perp}_2/\cU_2)$ as the family of smooth conics with the equation given by the restriction of~$q$, i.e. $E\simeq \Gad \times_{\mathrm{OGr}_q(2,V)}\mathrm{OFl}_q(2,3;V)$.
\end{proposition}
Let us introduce some notation.
We will denote by $\cO(H_2)$ and $\cO(H_3)$ the Pl{\"u}cker line bundles on $\Gr(2,V)$ and $\Gr(3,V)$ restricted to~$\Gad$ and $\bar{\mathsf{O}}_1$ respectively. Abusing the notation the restriction~$\cO(H_3)\restr{\mathsf{O}_1}$ will be also denoted by~$\cO(H_3)$. The relative hyperplane bundle of $\PP_{\Gad}(\cU^{\perp}_2/\cU_2)$ will be denoted by~$\cO_{\PP_{\Gad}(\cU^{\perp}_2/\cU_2)}(1)$. The restriction of the map $\pi_2$ to $\mathsf{O}_1\simeq \mathrm{Bl}_{\mathsf{O}_2}\bar{\mathsf{O}}_1\setminus E$ will be denoted by
\begin{equation} \label{tpi}
\tpi\colon \mathsf{O}_1\to\Gad.
\end{equation}

Using Proposition \ref{proposition:resolution_of_Hnu} we can prove the following statement. 
\begin{lemma} \label{lemma:isomorphism}
On $\Gad$ we have $S^2\cU_2(H_2) \simeq \cU^{\perp}_2/\cU_2$. In particular
\begin{equation*}
 E\simeq \PP_{\Gad}(\cU_2)\subset \PP_{\Gad}(S^2\cU_2)\simeq \PP_{\Gad}(\cU^{\perp}_2/\cU_2),   
\end{equation*}
where the embedding is the Veronese embedding. 
\end{lemma}
\begin{proof}
Consider the diagram
\begin{equation} \label{diagram:OGr}
\xymatrix{
& \mathrm{OFl}_{q}(2,3;V) \ar[dl]_{pr_2} \ar[dr]^{pr_3} &\\ 
\mathrm{OGr}_{q}(2,V) && \mathrm{OGr}_{q}(3,V),
}
\end{equation}
where we denote by $\mathrm{OFl}_{q}(2,3;V)\subset \mathrm{Fl}(2,3;V)$ isotropic partial flag variety. Denote by $\cU_k$ and $\cO(H_k)$ the restrictions of the tautological vector bundle and the Pl{\"u}cker line bundle on~$\Gr(k,V)$ to $\mathrm{OGr}_{q}(k,V)$ respectively.

On the one hand, $\mathrm{OFl}_{q}(2,3;V)\subset \PP_{\mathrm{OGr}_{q}(2,V)}(\cU^{\perp}_2/\cU_2)$ is a smooth conic bundle over $\mathrm{OGr}_{q}(2,V)$. On the other hand, the line bundle $pr_3^*\cO(H_3)$ restricts to each conic fiber of $pr_2$ as $\cO_{\PP^1}(1)$ and~$pr_{2*}pr_3^*\cO(H_3)\simeq \cS$ is the spinor bundle of $\mathrm{OGr}_q(2,V)$. Therefore, $\mathrm{OFl}_{q}(2,3;V)\simeq \PP_{\mathrm{OGr}_q(2,V)}(\cS)$ and $S^2\cS^{\vee}(H_2)\simeq \cU^{\perp}_2/\cU_2$.

Finally, $\cS^{\vee}\restr{\Gad}\simeq \cU_2$  by \cite[Lemma 8.3]{K06}, so we conclude
$$ S^2\cU_2(H_2)\simeq S^2\cS^{\vee}(H_2)\restr{\Gad} \simeq  \cU^{\perp}_2/\cU_2.$$
The last statement follows from the above isomorphism and Proposition \ref{proposition:resolution_of_Hnu}. 
\end{proof}
\begin{lemma} \label{lemma:computation_of_O(E)}
In the notation of Proposition \ref{proposition:resolution_of_Hnu} we have $\cO_{\PP_{\Gad}(\cU^{\perp}_2/\cU_2)}(1)\simeq \pi_1^*\cO(H_3)\otimes \pi_2^*\cO(-H_2)$ and~$\cO(E)\simeq \pi_1^*\cO(2H_3)\otimes \pi_2^*\cO(-2H_2)$.
\end{lemma}
\begin{proof}
Note that the Picard group of $\PP_{\Gad}(\cU^{\perp}_2/\cU_2)$ is torsion-free. Indeed, $\PP_{\Gad}(\cU^{\perp}_2/\cU_2)$ is a $\PP^2$-bundle over $\Gad$ and the Picard group of $\Gad$ is torsion-free by Lemma \ref{adjoint_properties}. On the one hand, the canonical line bundle of the projectivization $\PP_{\Gad}(\cU^{\perp}_2/\cU_2)$ is isomorphic to
$$\cO_{\PP_{\Gad}(\cU^{\perp}_2/\cU_2)}(-3)\otimes \pi_2^*( \omega_{\Gad}\otimes \mathrm{det}(\cU^{\perp}_2/\cU_2)^{\vee})\simeq \cO_{\PP_{\Gad}(\cU^{\perp}_2/\cU_2)}(-3)\otimes \pi_2^*\cO(-3H_2).$$
On the other hand, since $\bar{\mathsf{O}}_1$ has nodal singularities along the codimension 2 orbit $\mathsf{O}_2$, the canonical line bundle of $\mathrm{Bl}_{\mathsf{O}_2}\bar{\mathsf{O}}_1$ is isomorphic to $\pi_1^*\omega_{\bar{\mathsf{O}}_1}\simeq \pi_1^*\cO(-3H_3)$. Since~$\mathrm{Bl}_{\mathsf{O}_2}\bar{\mathsf{O}}_1\simeq \PP_{\Gad}(\cU^{\perp}_2/\cU_2)$ we conclude that $$\cO_{\PP_{\Gad}(\cU^{\perp}_2/\cU_2)}(1)\simeq \pi_1^*\cO(H_3)\otimes \pi_2^*\cO(-H_2).$$

The conic bundle $E \subset \PP_{\Gad}(\cU^{\perp}_2/\cU_2)$ is given by the quadratic form $q$ hence $E$ is the zero locus of a section of the line bundle $\cO_{\PP_{\Gad}(\cU^{\perp}_2/\cU_2)}(2)$. So we get that~$\cO(E)\simeq \pi_1^*\cO(2H_3)\otimes \pi_2^*\cO(-2H_2)$.
\end{proof}
The tautological exact sequence on $\PP_{\Gad}(\cU^{\perp}_2/\cU_2)\simeq \PP_{\Gad}(S^2\cU_2(H_2))$ has the following form
\begin{equation} \label{tautological_exact_sequence_on_projectivization_1}
0\to \pi_1^*\cO(-H_3)\otimes \pi_2^*\cO(H_2)\to \pi^{*}_2(\cU^{\perp}_2/\cU_2)\to \pi^{*}_2(\cU^{\perp}_2/\cU_2)/(\pi_1^*\cO(-H_3)\otimes \pi_2^*\cO(H_2))\to 0
\end{equation}
or, equivalently, 
\begin{equation} \label{tautological_exact_sequence_on_projectivization}
0\to \pi_1^*\cO(-H_3)\otimes \pi_2^*\cO(H_2)\to \pi^{*}_2(S^2\cU_2(H_2) )\to \pi^{*}_2(S^2\cU_2(H_2))/(\pi_1^*\cO(-H_3)\otimes \pi_2^*\cO(H_2))\to 0.
\end{equation}
We will denote by 
\begin{equation} \label{quotient_bundle}
 \cQ_{\PP_{\scaleto{\Gad}{6.5pt}}(S^2\cU_2)}\coloneqq \pi^{*}_2(\cU^{\perp}_2/\cU_2)/(\pi_1^*\cO(-H_3)\otimes \pi_2^*\cO(H_2))\simeq \pi^{*}_2(S^2\cU_2(H_2))/(\pi_1^*\cO(-H_3)\otimes \pi_2^*\cO(H_2))
 \end{equation}
the tautological quotient bundle on $\PP_{\Gad}(\cU^{\perp}_2/\cU_2)\simeq \PP_{\Gad}(S^2\cU_2(H_2))$.
\begin{corollary} \label{corollary_unramified}
There exists a unique nontrivial unramified double cover of $\mathsf{O}_1$:
$$h\colon \widetilde{\mathsf{O}}_1\to \mathsf{O}_1.$$
Moreover, we have $\widetilde{\mathsf{O}}_1= \mathrm{Spec}_{\mathsf{O}_1}(\cO\oplus (\cO(H_3)\otimes \tpi^* \cO(-H_2)))$.
\end{corollary}
\begin{proof}
Since $\mathrm{Bl}_{\mathsf{O}_2}\bar{\mathsf{O}}_1$ is a $\PP^2$-bundle over $\Gad$, we have $\mathrm{Pic}(\mathrm{Bl}_{\mathsf{O}_2}\bar{\mathsf{O}}_1)=\mathbb{Z}\cdot \pi_2^* \cO(H_2)\oplus \mathbb{Z}\cdot \pi_1^*\cO(H_3)$. Therefore, by Proposition \ref{proposition:resolution_of_Hnu} and Lemma \ref{lemma:computation_of_O(E)}
\begin{equation} \label{picard_on_O1}
    \mathrm{Pic}(\mathsf{O}_1)=\mathrm{Pic}(\mathrm{Bl}_{\mathsf{O}_2}\bar{\mathsf{O}}_1\setminus E)=(\mathbb{Z}\cdot \pi_2^*\cO(H_2)\oplus \mathbb{Z}\cdot \pi_1^* \cO(H_3))/\mathbb{Z}\cdot (\pi_1^*\cO(2H_3)-\pi_2^*\cO(2H_2))\simeq \mathbb{Z}\oplus \mathbb{Z}/2\mathbb{Z}
\end{equation}
and the only non-trivial 2-torsion element is $\cO(H_3)\otimes \tpi^* \cO(-H_2)$. So we get the statement. 
\end{proof}
\begin{corollary} \label{corollary:double_cover}
The double cover $h$ is isomorphic to $$\PP(\cU_2)\times_{\Gad} \PP(\cU_2)\to \PP_{\Gad}(S^2\cU_2)\simeq \PP_{\Gad}(\cU^{\perp}_2/\cU_2)$$
base changed to $\mathsf{O}_1$.
\end{corollary}
\begin{proof}
Immediately follows from the uniqueness of $h$.
\end{proof}
Denote $\cO_{\PP(\cU_2)\times_{\Gad} \PP(\cU_2)}(-1,-1)$ by $\cO_{\PP(\cU_2)\times \PP(\cU_2)}(-1,-1)$.
\begin{corollary} \label{corollary:pushforward}
On $\mathsf{O}_1$ we have $$h_*\cO_{\PP(\cU_2)\times \PP(\cU_2)}(-1,-1)\simeq \cO(-H_3)\oplus \pi_2^*\cO(-H_2).$$ Moreover, the fiber of $h_*\cO_{\PP(\cU_2)\times \PP(\cU_2)}(-1,-1)\subset \pi_2^*\cU_2 \otimes \pi_2^*\cU_2$ over $[P_1]=\langle e_{0},e_{\beta},e_{-\gamma} \rangle$ is isomorphic to~$\langle e_{\beta}\otimes e_{-\gamma}, e_{-\gamma}\otimes e_{\beta} \rangle$.
\end{corollary}
\begin{proof}
We have $$h_*\cO_{\PP(\cU_2)\times \PP(\cU_2)}(-1,-1)\simeq h_*h^*\cO(-H_3)\simeq \cO(-H_3)\oplus \pi_2^*\cO(-H_2),$$
so the first claim follows. All the remaining statements follows immediately from Corollary \ref{corollary:double_cover}.
\end{proof}
\begin{corollary} \label{corollary:twist_of_quotient}
We have $(\cQ_{\PP_{\scaleto{\Gad}{6.5pt}}(S^2\cU_2)}\restr{\mathsf{O}_1})^{\vee}\simeq \cQ_{\PP_{\scaleto{\Gad}{6.5pt}}(S^2\cU_2)}\restr{\mathsf{O}_1}\simeq \cQ_{\PP_{\scaleto{\Gad}{6.5pt}}(S^2\cU_2)}\restr{\mathsf{O}_1}\otimes \tpi^*\cO(-H_2)\otimes \cO(H_3)$.
\end{corollary}
\begin{proof}
By Proposition \ref{lemma:quadric_bundles_cokernel} the pair $(q,\epsilon)$, where $\epsilon$ is the tautological exact sequence \eqref{tautological_exact_sequence_on_projectivization_1} and~$q\colon \pi^{*}_2(\cU^{\perp}_2/\cU_2)\simeq  \pi^{*}_2(\cU^{\perp}_2/\cU_2)^{\vee}$, induces a self-dual isomorphism $(\cQ_{\PP_{\scaleto{\Gad}{6.5pt}}(S^2\cU_2)}\restr{\mathsf{O}_1})^{\vee}\simeq \cQ_{\PP_{\scaleto{\Gad}{6.5pt}}(S^2\cU_2)}\restr{\mathsf{O}_1}$ on $\mathsf{O}_1$, that determines the double cover $h$. Denote by $i_h\colon \widetilde{\mathsf{O}}_1\hookrightarrow \PP(\cQ_{\PP_{\scaleto{\Gad}{6.5pt}}(S^2\cU_2)}\restr{\mathsf{O}_1})$ the corresponding embedding.

Denote by $\cO_{\PP(\cQ)}(1)$ the tautological line bundle on $\PP(\cQ_{\PP_{\scaleto{\Gad}{6.5pt}}(S^2\cU_2)}\restr{\mathsf{O}_1})$. We have
\begin{multline*}
 \cQ_{\PP_{\scaleto{\Gad}{6.5pt}}(S^2\cU_2)}\restr{\mathsf{O}_1}\simeq h_*i_{h}^*\cO_{\PP(\cQ)}(1)\simeq  h_*(i_{h}^*\cO_{\PP(\cQ)}(1)\otimes \cO_{\widetilde{\mathsf{O}}_1})\simeq\\
 \simeq h_*(i_{h}^*\cO_{\PP(\cQ)}(1)\otimes h^*(\tpi^*\cO(-H_2)\otimes \cO(H_3)))\simeq \cQ_{\PP_{\scaleto{\Gad}{6.5pt}}(S^2\cU_2)}\restr{\mathsf{O}_1}\otimes \tpi^*\cO(-H_2)\otimes \cO(H_3),   
\end{multline*}
so we get the statement. 
\end{proof}
\section{Self-dualities} \label{section:selfdualities}
The main goal of this section is to prove the self-dualities \eqref{self-dualities}, that are crucial for the proof of Theorem \ref{Theorem:introduction:main}. We prove \eqref{self-dualities} using certain $\rG_2$-equivariant quadric bundles on $\CG$ and Proposition~\ref{lemma:quadric_bundles_cokernel}. In Subsection \ref{subsection:quadric_bundles} we describe quadric bundles that can be constructed from $q\colon V\otimes \cO\stackrel{\sim}{\longrightarrow} V^{\vee}\otimes \cO$, in Subsection \ref{subsection:first} we prove the isomorphism~$\cE_{10}(1)\simeq \cE_{10}^{\vee}$, in Subsection \ref{subsection:wedge} we describe quadric bundles that can be constructed from $\wedge\colon \Lambda^2\cU_3^{\perp}\stackrel{\sim}{\longrightarrow} \Lambda^2\cQ_3(-H_3)$, and in Subsection \ref{subsection:second} we prove the isomorphism~$\cE_{16}(-1)\simeq \cE_{16}^{\vee}$.

We will continue to use the notation of Section \ref{section:G2action}. To avoid confusion in this subsection the vector bundles $\cU$, $\cU^{\perp}$ and $\cQ$ on $\CG$ will be denoted by $\cU_3$ and $\cU_3^{\perp}$ and $\cQ_3$ respectively, and the Pl{\"u}cker line bundle $\cO(1)$ and its restriction to $\bar{\mathsf{O}}_1$ or $\mathsf{O}_1$ will be denoted by $\cO(H_3)$.
\subsection{Quadric bundles induced by the quadratic form} \label{subsection:quadric_bundles}
In this subsection we describe quadric bundles that can be constructed using Proposition \ref{lemma:quadric_bundles_cokernel} and Proposition \ref{quadric_construction} from the self-dual isomorphism~$q\colon V\otimes \cO\stackrel{\sim}{\longrightarrow} V^{\vee}\otimes \cO$.
\begin{lemma} \label{lemma:self_dual_1}
A quadric bundle $q\colon V\otimes \cO\stackrel{\sim}{\longrightarrow} V^{\vee}\otimes \cO$ and the tautological subbundle $\cU_3\hookrightarrow V\otimes \cO$ induce a pair of quadric bundles with isomorphic cokernel sheaves:
\begin{align*}
  & 0\to \cU_3 \xrightarrow{q_{\scaleto{\cU}{3.5pt}}} \cU_3^{\vee} \to i_*\cC(q_{\scaleto{\cU}{3.5pt}})\to 0 \\
  & 0\to \cU_3^{\perp}\xrightarrow{q_{\scaleto{\cU^{\perp}}{5pt}}}\cQ_3 \to i_*\cC(q_{\scaleto{\cU}{3.5pt}})\to 0,
\end{align*}
where $\cC(q_{\scaleto{\cU}{3.5pt}})\simeq \cC(q_{\scaleto{\cU}{3.5pt}})^{\vee}\otimes \cO(H_3)$ is a self-dual sheaf of rank 2 on $\bar{\mathsf{O}}_1$. Moreover, we have $\cC(q_{\scaleto{\cU}{3.5pt}})\restr{\mathsf{O}_1}\simeq \tpi^*\cU^{\vee}_2$ and on $\mathsf{O}_1$ there exits an embedding $\tpi^*\cU_2\hookrightarrow \cU_3^{\perp}$, such that $\mathrm{Im}(q_{\cU^{\perp}})\restr{\mathsf{O}_1}\simeq \cU_3^{\perp}/ \tpi^*\cU_2$.
\end{lemma}
\begin{proof}
Applying Proposition \ref{lemma:quadric_bundles_cokernel} to the pair $(q,\epsilon)$, where $\epsilon$ is the tautological exact sequence \eqref{tautological_exact_sequence} on~$\CG$ we get the first statement. Since $\mathrm{det}(\cU^{\vee})\otimes \mathrm{det}(\cU^{\vee})=\cO(2H_3)$ we obtain by Lemma \ref{lemma:degeneracy_locus} and Proposition~\ref{proposition:orbits} that the scheme support of the cokernel sheaf $\mathrm{Coker}(q_{\scaleto{\cU}{3.5pt}})$ is $\bar{\mathsf{O}}_1$ and that~$\cC(q_{\scaleto{\cU}{3.5pt}})\restr{\mathsf{O}_1}\simeq \tpi^*\cU^{\vee}_2$. In particular, on $\mathsf{O}_1$ we get
$$\mathrm{Im}(q_{\cU^{\perp}})\simeq \cU_3^{\perp}/\mathrm{Ker}(q_{\cU^{\perp}})\simeq \cU_3^{\perp}/ \tpi^*\cU_2.$$
Using Proposition \ref{quadric_construction} we conclude $\cC(q_{\scaleto{\cU}{3.5pt}})\simeq \cC(q_{\scaleto{\cU}{3.5pt}})^{\vee}\otimes \cO(H_3)$.
\end{proof}
Recall the notation \eqref{quotient_bundle}.
\begin{lemma} \label{lemma_f2}
On $\CG$ we have a $\rG_2$-equivariant quadric bundle
\begin{equation*}
q_{\scaleto{\cQ_3}{6.5pt}}\colon \cQ_3(-1)\to \cU_3^{\perp} 
\end{equation*}
with the cokernel sheaf isomorphic to~$i_*\cC(q_{\scaleto{\cQ_3}{6.5pt}})$, where $\cC(q_{\scaleto{\cQ_3}{6.5pt}})$ is a reflexive sheaf of rank 2 on~$\bar{\mathsf{O}}_1$. Moreover,~$\cC(q_{\scaleto{\cQ_3}{6.5pt}})\restr{\mathsf{O}_1}\simeq \cQ_{\PP_{\scaleto{\Gad}{6.5pt}}(S^2\cU_2)}\restr{\mathsf{O}_1}$.
\end{lemma}
\begin{proof}
Applying Proposition \ref{quadric_construction} to the quadric bundle $q_{\scaleto{\cU^{\perp}}{6.5pt}}$ from Lemma \ref{lemma:self_dual_1} we obtain the required $\rG_2$-equivariant quadric bundle $q_{\scaleto{\cQ_3}{6.5pt}}$. By Proposition \ref{quadric_construction} and Lemma \ref{lemma:self_dual_1} we obtain
$$\cC(q_{\scaleto{\cQ_3}{6.5pt}})\restr{\mathsf{O}_1}\simeq \mathrm{Im}(q_{\cU^{\perp}})\restr{\mathsf{O}_1}\simeq \cU_3^{\perp}/ \tpi^*\cU_2.$$
Since on $\mathsf{O}_1$ we have $\cO(-H_3)\otimes \tpi^*\cO(H_2) \simeq \cU_3/\tpi^*\cU_2$ the restriction of the exact sequence \eqref{tautological_exact_sequence_on_projectivization_1} to $\mathsf{O}_1$ has the following form:
\begin{equation} \label{tautological_restriction}
    0\to \cU_3/\tpi^*\cU_2 \to \tpi^*\cU_2^{\perp}/ \tpi^*\cU_2 \to \tpi^*\cU^{\perp}_2/\cU_3 \to 0,
\end{equation}
so we obtain $\cQ_{\PP_{\scaleto{\Gad}{6.5pt}}(S^2\cU_2)}\restr{\mathsf{O}_1}\simeq \tpi^*\cU^{\perp}_2/\cU_3$. By Corollary \ref{corollary:twist_of_quotient} we have $$\cQ_{\PP_{\scaleto{\Gad}{6.5pt}}(S^2\cU_2)}\restr{\mathsf{O}_1}\simeq (\cQ_{\PP_{\scaleto{\Gad}{6.5pt}}(S^2\cU_2)}\restr{\mathsf{O}_1})^{\vee}\simeq (\tpi^*\cU^{\perp}_2/\cU_3)^{\vee}.$$
The exact sequence dual to \eqref{tautological_restriction}
\begin{equation} \label{U_3U_2}
   0\to \cU_3^{\perp}/ \tpi^*\cU_2\to \tpi^*\cU_2^{\perp}/ \tpi^*\cU_2 \to (\cU_3/\tpi^*\cU_2)^{\vee}\to 0
\end{equation}
shows $(\tpi^*\cU^{\perp}_2/\cU_3)^{\vee}\simeq \cU_3^{\perp}/\tpi^*\cU_2$, so we deduce 
$$\cC(q_{\scaleto{\cQ_3}{6.5pt}})\simeq \cU_3^{\perp}/\tpi^*\cU_2\simeq (\tpi^*\cU^{\perp}_2/\cU_3)^{\vee}\simeq \cQ_{\PP_{\scaleto{\Gad}{6.5pt}}(S^2\cU_2)}\restr{\mathsf{O}_1}$$
\end{proof}
\begin{lemma} \label{lemma:map_l}
On $\CG$ there exists a nonzero self-dual $\rG_2$-equivariant map
\begin{equation} \label{definition_l}
     q_{\scaleto{\Lambda^2\cU_3}{6.5pt}}\colon \Lambda^2\cU_3\to \cU_3
\end{equation} 
with the cokernel sheaf isomorphic to $i_*\cC(q_{\scaleto{\Lambda^2\cU_3}{6.5pt}})$, where $\cC(q_{\scaleto{\Lambda^2\cU_3}{6.5pt}})$ is a reflexive sheaf of rank~1 on $\bar{\mathsf{O}}_1$. Moreover, we have~$\cC(q_{\scaleto{\Lambda^2\cU_3}{6.5pt}})\restr{\mathsf{O}_1}\simeq  \cO(-H_3)\otimes\tpi^*\cO(H_2)$.
\end{lemma}
\begin{proof}
Applying Proposition \ref{quadric_construction} to the quadric bundle $q_{\cU}$ from Lemma \ref{lemma:self_dual_1} we obtain a $\rG_2$-equivariant self-dual map~$\cU_3^{\vee}(-1)\to \cU_3$ with the required cokernel sheaf. Using the isomorphism \eqref{preliminaries_isomorphisms} we get the required map. 

By Proposition \ref{quadric_construction} the map $ q_{\Lambda^2\cU_3}$ can be described as
\begin{equation} \label{definition_l_1}
     q_{\scaleto{\Lambda^2\cU_3}{6.5pt}}(u_1\wedge u_2) = q^{-1}(\nu\conv u_1\wedge u_2), 
\end{equation}
because $\nu$ is the equation of $\bar{\mathsf{O}}_1$.
\end{proof}
Now let us show that the map $q_{\Lambda^2\cU_3}$ defines a Lie algebra structure on $\cU$. 
\begin{lemma} \label{lemma:lie_sructure}
Subspaces of $V$ parametrized by $\CG$ are Lie algebras with respect to $q_{\Lambda^2\cU_3}$:
\begin{enumerate}
\item the open dense orbit $\mathsf{O}_0$ parametrizes three-dimensional semisimple Lie algebras that are isomorphic to~$\mathfrak{sl}_2$;
\item the locally closed 7-dimensional orbit $\mathsf{O}_1$ parametrizes three-dimensional solvable Lie algebras that are isomorphic to the subalgebra of $\mathfrak{gl}_3$ formed by matrices with zero second and third row;
\item the closed 5-dimensional orbit $\mathsf{O}_2$ parametrizes three-dimensional nilpotent Lie algebras that are isomorphic to the strictly upper triangular matrices $\mathfrak{n}_3\subset \mathfrak{sl}_3$.
\end{enumerate}
\end{lemma}
\begin{proof}
Using the explicit formula \eqref{definition_l_1} we compute $q_{\Lambda^2\cU_3}$ at points~$P_i$ from Proposition \ref{proposition:orbits}.

For $P_0:=\langle e_0,e_{\gamma},e_{-\gamma} \rangle$ we have 
\begin{equation} \label{lie_p0}
     q_{\scaleto{\Lambda^2\cU_3}{6.5pt}}(e_0\wedge e_{\gamma})=-2e_{\gamma}, \quad
     q_{\scaleto{\Lambda^2\cU_3}{6.5pt}}(e_0\wedge e_{-\gamma})=2e_{-\gamma}, \quad
     q_{\scaleto{\Lambda^2\cU_3}{6.5pt}}(e_{\gamma}\wedge e_{-\gamma})=e_{0}.
\end{equation}
For $P_1:=\langle e_{0},e_{\beta},e_{-\gamma} \rangle$ we have 
\begin{equation} \label{lie_p1}
     q_{\scaleto{\Lambda^2\cU_3}{6.5pt}}(e_0\wedge e_{\beta})=-2e_{\beta}, \quad
     q_{\scaleto{\Lambda^2\cU_3}{6.5pt}}(e_0\wedge e_{-\gamma})=2e_{-\gamma}, \quad
     q_{\scaleto{\Lambda^2\cU_3}{6.5pt}}(e_{\beta}\wedge e_{-\gamma})=0.
\end{equation} 
For $P_2:=\langle e_{\alpha},e_{\beta},e_{-\gamma} \rangle$ we have 
\begin{equation} \label{lie_p2}
     q_{\scaleto{\Lambda^2\cU_3}{6.5pt}}(e_{\alpha}\wedge e_{\beta})=-2e_{-\gamma}, \quad
     q_{\scaleto{\Lambda^2\cU_3}{6.5pt}}(e_{\alpha}\wedge e_{-\gamma})=0, \quad
     q_{\scaleto{\Lambda^2\cU_3}{6.5pt}}(e_{\beta}\wedge e_{-\gamma})=0.
\end{equation}
The statement of the lemma follows from these computations.
\end{proof}
\begin{remark} \label{remark:cayley}
By \cite[Proposition 3.2]{M} the Cayley Grassmannian $\CG$ can be described as a closed subvariety in~$\Gr(3,\mathrm{Im}(\mathbb{O}))$, that parametrizes the imaginary parts of 4-dimensional associative subalgebras of~$\mathbb{O}$, i.e. 3-dimensional subspaces of $V= \mathrm{Im}(\mathbb{O})$ closed with respect to the skew-symmetric bracket~$\br \coloneqq q^{-1}(\nu\conv (x\wedge y))= \mathrm{Im}(x\times y)$, where~$\times\colon \mathbb{O}\times \mathbb{O}\to \mathbb{O}$ is the octonionic product. The lemma above shows that $\br$ defines a Lie algebra structure on 3-dimensional subspaces of~$\mathrm{Im}(\mathbb{O})$ parametrized by~$\CG$. In fact, it is easy to check that the Jacobiator of $\br$ on $V$ is equal to~$q^{-1}(\lambda\conv(-))\colon \Lambda^3V \to V$, hence vanishes on $U_3\subset V$ precisely when $[U_3]\in \CG$.
\end{remark}
Denote by $\bQ_{\Lambda^2\cU_3}$ the $\rG_2$-equivariant conic bundle that is defined by $q_{\scaleto{\Lambda^2\cU_3}{6.5pt}}$.
\begin{corollary} \label{lemma_l}
Over~$\mathsf{O}_0$ the fibers of $\bQ_{\Lambda^2\cU_3}$ are smooth conics, over~$\mathsf{O}_1$ the fibers of $\bQ_{\Lambda^2\cU_3}$ are reducible conics and over $\mathsf{O}_2$ the fibers of~$\bQ_{\Lambda^2\cU_3}$ are double lines. Moreover, over $\mathsf{O}_1$ the map $q_{\scaleto{\Lambda^2\cU_3}{6.5pt}}$ induces a self-dual isomorphism
$$\tpi^*\cU_2^{\vee}(-H_3)\to \tpi^*\cU_2,$$
that is determined by the section $\cO(-H_3)\otimes \tpi^*\cO(H_2)\to \tpi^{*}S^2\cU_2(H_2)$ corresponding to the restriction to~$\mathsf{O}_1$ of the embedding from the tautological exact sequence \eqref{tautological_exact_sequence_on_projectivization}.
\end{corollary}
\begin{proof}
Using the computations \eqref{lie_p0}, \eqref{lie_p1} and \eqref{lie_p2} we deduce the first statement.

Applying Proposition \ref{quadric_construction} and Lemma \ref{lemma:self_dual_1} to the quadric bundles $q_{\scaleto{\cU}{4pt}}$ and $q_{\scaleto{\Lambda^2\cU_3}{6.5pt}}$ we obtain a self-dual isomorphism $\tpi^*\cU_2^{\vee}(-H_3)\to \tpi^*\cU_2$ on $\mathsf{O}_1$. The fact that this isomorphism is determined by the restriction of the embedding from the exact sequence \eqref{tautological_exact_sequence_on_projectivization} follows from the direct computation \eqref{lie_p1}, that shows that~$q_{\scaleto{\Lambda^2\cU_3}{6.5pt}}\in S^2\cU_2(H_3)$ at the point $P_1$ is equal to~$-2 e_{\beta}\otimes e_{-\gamma}- 2 e_{-\gamma}\otimes e_{\beta}$.
\end{proof}
\begin{corollary} \label{corollary_l_twisted}
The cokernel sheaf of the self-dual morphism $q_{\scaleto{\Lambda^2\cU_3}{6.5pt}}$ twisted by $\cO(H_3)$
$$q_{\scaleto{\cU^{\vee}_3}{6.5pt}}\colon \cU_3^{\vee}\to \Lambda^2\cU_3^{\vee}$$
is isomorphic to $i_*\cC(q_{\scaleto{\cU^{\vee}_3}{6.5pt}})$, where $\cC(q_{\scaleto{\cU^{\vee}_3}{6.5pt}})$ is a reflexive sheaf of rank~1 on $\bar{\mathsf{O}}_1$. Moreover, we have~$\cC(q_{\scaleto{\cU^{\vee}_3}{6.5pt}})\restr{\mathsf{O}_1}\simeq  \tpi^*\cO(H_2)$.
\end{corollary}
\begin{proof}
This is an immediate consequence of Lemma \ref{lemma:map_l} and isomorphism \eqref{preliminaries_isomorphisms}.
\end{proof}
\subsection{First self-duality} \label{subsection:first}
The aim of this subsection is to prove the isomorphism $\cE_{10}(H_3)\simeq \cE_{10}^{\vee}$. Recall the notation \eqref{quotient_bundle}.
\begin{lemma} \label{lemma_f1}
On $\CG$ we have a $\rG_2$-equivariant self-dual morphism 
\begin{equation*}
   q_{\scaleto{S^2\cU_3}{6.5pt}}\colon S^2\cU_3 \to S^2\cU_3^{\vee}(-H_3)
\end{equation*}
that determines a $\rG_2$-equivariant quadric bundle in $\bQ_{\scaleto{S^2\cU_3}{6.5pt}}\hookrightarrow\PP_{\CG}(S^2\cU_3)\to \CG$. The cokernel sheaf of~$q_{\scaleto{S^2\cU_3}{6.5pt}}$
is isomorphic to $i_*\cC(q_{\scaleto{S^2\cU_3}{6.5pt}})$, where $\cC(q_{\scaleto{S^2\cU_3}{6.5pt}})$ is a reflexive sheaf of rank 2 on $\bar{\mathsf{O}}_1$. Moreover,~$\cC(q_{\scaleto{S^2\cU_3}{6.5pt}})\restr{\mathsf{O}_1}\simeq \cQ_{\PP_{\scaleto{\Gad}{6.5pt}}(S^2\cU_2)}\restr{\mathsf{O}_1}$.
\end{lemma}
\begin{proof}
Applying Corollary \ref{proposition_veronese_family} to the conic bundle $q_{\scaleto{\Lambda^2\cU_3}{6.5pt}}$ from Lemma \ref{lemma:map_l}, so that in the notation of Corollary \ref{proposition_veronese_family} we have $\cF_3=\cU_3$ and $\cL=\cO(-H_3)$, we obtain a $\rG_2$-equivariant self-dual morphism $q_{\scaleto{S^2\cU_3}{6.5pt}}$ that determines a quadric bundle $\bQ_{\scaleto{S^2\cU_3}{6.5pt}}\hookrightarrow\PP_{\CG}(S^2\cU_3)$. By Corollary \ref{proposition_veronese_family} and Lemma~\ref{lemma_l} over $\mathsf{O}_0$ the fibers of $\bQ_{\scaleto{S^2\cU_3}{6.5pt}}$ are smooth quadrics and over~$ \mathsf{O}_1$ the fibers of $\bQ_{\scaleto{S^2\cU_3}{6.5pt}}$ are quadrics of rank 4. 

We have $\mathrm{det}(S^2\cU_3^{\vee}(-H_3))\otimes \mathrm{det}(S^2\cU_3^{\vee})=\cO(2H_3)$
hence by $\rG_2$-equivariancy the degeneracy locus of $q_{\scaleto{S^2\cU_3}{6.5pt}}$ is $\bar{\mathsf{O}}_1$ with multiplicity 2. We have already shown that the fibers of $\bQ_{\scaleto{S^2\cU_3}{6.5pt}}$ over $\mathsf{O}_1$ are conics of rank 4 so we obtain using Lemma \ref{lemma:degeneracy_locus} that the scheme-theoretic support of the cokernel sheaf is~$\bar{\mathsf{O}}_1$, i.e.~$\mathrm{Coker}(q_{\scaleto{S^2\cU_3}{6.5pt}})\simeq i_*\cC(q_{\scaleto{S^2\cU_3}{6.5pt}})$. By Proposition \ref{quadric_construction} the sheaf $\cC(q_{\scaleto{S^2\cU_3}{6.5pt}})$ is self-dual.

Now let us show $\cC(q_{\scaleto{S^2\cU_3}{6.5pt}})\restr{\mathsf{O}_1}\simeq \cQ_{\PP_{\scaleto{\Gad}{6.5pt}}(S^2\cU_2)}\restr{\mathsf{O}_1}$. By Corollary \ref{kernel_veronese} and Corollary~\ref{lemma_l} the kernel of~$q_{\scaleto{S^2\cU_3}{6.5pt}}\restr{\mathsf{O}_1}$ is isomorphic to~$S^2\cU_2/\cO(-H_3)\simeq \cQ_{\PP_{\scaleto{\Gad}{6.5pt}}(S^2\cU_2)}\restr{\mathsf{O}_1}\otimes \tpi^*\cO(-H_2)$. Thus, using Corollary \ref{corollary:twist_of_quotient} we obtain
$$\cC(q_{\scaleto{S^2\cU_3}{6.5pt}})\restr{\mathsf{O}_1}\simeq \cQ_{\PP_{\scaleto{\Gad}{6.5pt}}(S^2\cU_2)}\restr{\mathsf{O}_1}\otimes \tpi^*\cO(-H_2)\otimes \cO(H_3)\simeq \cQ_{\PP_{\scaleto{\Gad}{6.5pt}}(S^2\cU_2)}\restr{\mathsf{O}_1}$$
and the statement follows.
\end{proof}
Recall the quadric bundle $q_{\scaleto{\cQ_3}{6.5pt}}$ from Lemma \ref{lemma_f2}.
\begin{lemma} \label{lemma:cokernels}
The cokernel sheaves $i_*\cC(q_{\scaleto{S^2\cU_3}{6.5pt}})$ and $i_*\cC(q_{\scaleto{\cQ_3}{6.5pt}})$ are isomorphic.
\end{lemma}
\begin{proof}
Since the sheaves $\cC(q_{\scaleto{S^2\cU_3}{6.5pt}})$ and $\cC(q_{\scaleto{\cQ_3}{6.5pt}})$ are reflexive, $\bar{\mathsf{O}}_1$ is normal and $\mathrm{codim}(\bar{\mathsf{O}}_1\setminus \mathsf{O}_1)=2$ it is enough to show that $\cC(q_{\scaleto{S^2\cU_3}{6.5pt}})\restr{\mathsf{O}_1}\simeq \cC(q_{\scaleto{\cQ_3}{6.5pt}})\restr{\mathsf{O}_1}$. By Lemma \ref{lemma_f1} and Lemma \ref{lemma_f2} we have
\begin{equation*}
 \cC(q_{\scaleto{\cQ_3}{6.5pt}})\restr{\mathsf{O}_1}\simeq \cQ_{\PP_{\scaleto{\Gad}{6.5pt}}(S^2\cU_2)}\restr{\mathsf{O}_1}\simeq \cC(q_{\scaleto{S^2\cU_3}{6.5pt}})\restr{\mathsf{O}_1},
\end{equation*}
so we get the statement.
\end{proof}
\begin{corollary} \label{proposition_E_E}
There exists an extension
\begin{equation} \label{extension_E}
    0 \to S^2\cU_3 \to \cE_{10} \to \cU_3^{\perp}\to 0
\end{equation}
that is self-dual $\cE_{10}(H_3)\simeq \cE_{10}^{\vee}$.
\end{corollary}
\begin{proof}
Using Proposition \ref{lemma:quadric_bundles_cokernel} for quadric bundles $q_{\scaleto{S^2\cU_3}{6.5pt}}$ and $q_{\scaleto{\cQ_3}{6.5pt}}$ and Lemma \ref{lemma:cokernels} we immediately get the statement. 
\end{proof}

\subsection{Quadric bundles induced by wedge product} \label{subsection:wedge}
In this subsection we describe quadric bundles that can be constructed using Proposition \ref{lemma:quadric_bundles_cokernel} from the self-dual isomorphism \eqref{preliminaries_isomorphisms2}:
$$\wedge\colon \Lambda^2\cU_3^{\perp}\stackrel{\sim}{\longrightarrow} \Lambda^2\cQ_3(-H_3).$$
First we need to describe the embedding $\Lambda^2\cU_3\hookrightarrow\Lambda^2\cU_3^{\perp}$ that we have already discussed in the Introduction.
\begin{lemma} \label{lemma:embedding}
On $\CG$ we have the embedding of vector bundles 
$$i_{\lambda}\colon \Lambda^2\cU_3\hookrightarrow \Lambda^2\cU_3^{\perp}\stackrel{\sim}{\longrightarrow} \Lambda^2\cQ_3(-H_3),$$
so that on $\CG$ there exist the following mutually dual exact sequences
\begin{align} \label{exact_sequence:R_Lambda^2Q}
  &  0\to \cR \to \Lambda^2\cQ_3 \xrightarrow{i^{\vee}_{\lambda}} \Lambda^2\cU_3^{\vee} \to 0 \\
 \label{exact_sequence:R*(2)}
  &   0\to\Lambda^2\cU_3\xrightarrow{i_{\lambda}} \Lambda^2\cU_3^{\perp}\to \cR^{\vee}\to 0.
\end{align}
\end{lemma}
\begin{proof}
This embedding is given by the 4-form $\lambda$. It can be described as follows 
\begin{equation}\label{i_explicit_formula}
   i_{\lambda}\colon u_1\wedge u_2 \to \lambda\conv\langle u_1\wedge u_2 \rangle.
\end{equation}
Indeed, $i_{\lambda}(u_1\wedge u_2)(u_3)=\lambda(u_1\wedge u_2\wedge u_3, -)=0$ so that we obtain $i_{\lambda}(\Lambda^2\cU_3)\subset \Lambda^2\cU^{\perp}_3$.

Note that $i_{\lambda}$ is a $\rG_2$-equivariant map, so to prove that $i_{\lambda}$ is an embedding of vector bundles it is enough to show that $i_{\lambda}$ is an embedding at the point $P_2$ of the closed orbit. Recall that $P_2:=\langle e_{\alpha},e_{\beta},e_{-\gamma} \rangle$, so using \eqref{i_explicit_formula} and \eqref{lambda} we compute
\begin{equation*}
    i_{\lambda}(e_{\alpha}\wedge e_{\beta})=2e^{\vee}_{0}\wedge e^{\vee}_{\gamma}+e^{\vee}_{-\alpha}\wedge e^{\vee}_{-\beta}, \quad
    i_{\lambda}(e_{\alpha}\wedge e_{-\gamma})=e^{\vee}_{-\alpha}\wedge e^{\vee}_{\gamma}, \quad
    i_{\lambda}(e_{\beta}\wedge e_{-\gamma})=e^{\vee}_{-\beta}\wedge e^{\vee}_{\gamma}.
\end{equation*}
The right hand sides are obviously linearly-independent hence $i_{\lambda}$ is injective.
\end{proof}
In what follows we will use an explicit computation of $i_{\lambda}\colon \Lambda^2\cU_3\hookrightarrow \Lambda^2\cQ_3(-H_3)$ at $P_0$ and $P_1$: 
\begin{align} \label{i_P0}
    & i_{\lambda}(e_0\wedge e_{\gamma})=-2e_{-\alpha}\wedge e_{-\beta}, \quad
    i_{\lambda}(e_0\wedge e_{-\gamma})=2e_{\alpha}\wedge e_{\beta}, \quad
    i_{\lambda}(e_{\gamma}\wedge e_{-\gamma})=e_{\alpha}\wedge e_{-\alpha}+e_{\beta}\wedge e_{-\beta};\\
    \label{i_P1}
    & i_{\lambda}(e_0\wedge e_{\beta})=-2e_{-\alpha}\wedge e_{-\beta}, \quad
    i_{\lambda}(e_0\wedge e_{-\gamma})=-2e_{\alpha}\wedge e_{\gamma}, \quad
    i_{\lambda}(e_{\beta}\wedge e_{-\gamma})=e_{\alpha}\wedge e_{-\alpha}.
\end{align}
\begin{lemma} \label{lemma:quadrics_wedge}
A quadric bundle $\wedge\colon \Lambda^2\cU_3^{\perp}\stackrel{\sim}{\longrightarrow} \Lambda^2\cQ_3(-H_3)$ and the embedding $i_{\lambda}\colon \Lambda^2\cU_3\hookrightarrow \Lambda^2\cU_3^{\perp}$ induce a pair of $\rG_2$-equivariant conic bundles with isomorphic cokernel sheaves:
\begin{align*}
  & 0\to \Lambda^2\cU_3 \xrightarrow{q_{\scaleto{\Lambda^2\cU_3}{6.5pt}}} \cU_3 \to i_*\cC(q_{\scaleto{\Lambda^2\cU_3}{6.5pt}})\to 0 \\
  & 0\to \cR(-H_3)\xrightarrow{q_{\scaleto{\cR}{3.5pt}}} \cR^{\vee} \to i_*\cC(q_{\scaleto{\Lambda^2\cU_3}{6.5pt}})\to 0.
\end{align*}
For the corresponding conics $\bQ_{\Lambda^2\cU_3}\subset \PP_{\CG}(\Lambda^2\cU)$ and $\bQ_{\cR}\subset \PP_{\CG}(\cR(-1))$ we have 
\begin{equation} \label{conic_bundles_intersections}
 \bQ_{\Lambda^2\cU_3}\simeq \PP_{\CG}(\Lambda^2\cU_3)\cap \Gr_{\CG}(2,\cU_3^{\perp}), \qquad \bQ_{\cR}\simeq \PP_{\CG}(\cR(-1))\cap \Gr_{\CG}(2,\cU_3^{\perp}).
\end{equation}
\end{lemma}
\begin{proof}
Applying Proposition \ref{lemma:quadric_bundles_cokernel} to the pair $(\wedge,\epsilon)$, where $\epsilon$ is the exact sequence \eqref{exact_sequence:R*(2)} we immediately obtain a pair of $\rG_2$-equivariant conic bundles with isomorphic cokernel sheaves
\begin{equation*}
\Lambda^2\cU_3 \xrightarrow{q'_{\scaleto{\Lambda^2\cU_3}{6.5pt}}}\cU_3  \quad \mbox{and} \quad \cR(-1)\stackrel{q_{\cR}}{\longrightarrow} \cR^{\vee}.
\end{equation*}
Note that the isomorphism $\wedge$ defines the quadric bundle~$\Gr_{\CG}(2,\cU_3^{\perp})$, so that we immediately get \eqref{conic_bundles_intersections} for the obtained conic bundles. Thus, all we need to show is that the conic bundle $q'_{\scaleto{\Lambda^2\cU_3}{6.5pt}}$ coincides with~$q_{\scaleto{\Lambda^2\cU_3}{6.5pt}}$.

Using \eqref{conic_bundles_intersections} and the explicit computation \eqref{i_P0} we see that $q_{\cR}$ is nonzero, thus, $q'_{\scaleto{\Lambda^2\cU_3}{6.5pt}}$ is also nonzero. By Lemma~\ref{main_computations} we have
$$\mathrm{Hom}(\Lambda^2\cU_3, \cU_3)=\mathrm{H}^0(\CG,\cU_3^{\vee})\oplus \mathrm{H}^0(\CG,S^2\cU_3(H_3))\simeq V^{\vee}\oplus \Bbbk,$$
hence we obtain that $q_{\scaleto{\Lambda^2\cU_3}{6.5pt}}$ is the unique non-zero $\rG_2$-equivariant map between $\Lambda^2\cU_3$ and $\cU_3$ and we deduce the statement.
\end{proof}
\subsection{Second self-duality} \label{subsection:second}
The aim of this subsection is to prove the isomorphism  $\cE_{16}\simeq \cE^{\vee}_{16}(H_3)$.
\begin{lemma} \label{lemma_f3}
On $\CG$ we have a $\rG_2$-equivariant quadric bundle
\begin{equation*}
    q_{i_{\lambda}}\colon  \cU_3^{\perp}\otimes \Lambda^2\cU_3^{\vee}\to \cQ_3\otimes \cU_3^{\vee}
\end{equation*}
with the cokernel sheaf isomorphic to $i^*\cC(q_{i_{\lambda}})$, where $\cC(q_{i_{\lambda}})$ is a reflexive sheaf of rank 2 on $\bar{\mathsf{O}}_1$. Moreover, we have~$\cC(q_{i_{\lambda}})\restr{\mathsf{O}_1}\simeq  \cO(H_3)\oplus \tpi^*\cO(H_2)$.
\end{lemma}
\begin{proof}
Applying Proposition \ref{proposition:segre_family} to the vector bundles $\cF_3=\cU_3$, $\cF_4=\cU_3^{\perp}$, $\cL=\cO(-H_3)$ and the map~$i_{\lambda}\in \mathrm{Hom}(\Lambda^2\cU_3, \Lambda^2\cQ_3(-H_3))$ we get $\rG_2$-equivariant quadric bundle 
$$\cU_3^{\perp}\otimes \cU_3(H_3)\to \cQ_3\otimes \cU_3^{\vee}.$$
Using the isomorphism \eqref{preliminaries_isomorphisms} we obtain the required quadric bundle $q_{i_{\lambda}}$. By Lemma \ref{lemma:quadrics_wedge} the corresponding conic bundle is $q_{\scaleto{\Lambda^2\cU_3}{6.5pt}}$.

Note that $\mathrm{det}(\cU_3)=\mathrm{det}(\cU_3^{\perp})=\cO(H_3)$. By Corollary \ref{lemma_l} and Lemma \ref{lemma:quadrics_wedge} the projective bundle~$i_{\lambda}(\PP(\Lambda^2\cU_3))\subset \PP(\Lambda^2\cU^{\perp}_3)$ is not contaibed in $\Gr(2,\cU^{\perp}_3)$ fiberwise. Thus, we can apply Corollary~\ref{corollary:Segre} and get that the set-theoretic support of the cokernel sheaf of $q_{i_{\lambda}}$ is $\bar{\mathsf{O}}_1$ and, moreover, the kernel~$\mathrm{Ker}(q_{i_{\lambda}}\restr{\mathsf{O}_1})$ is 2-dimensional. Since~$\mathrm{det}(\cU^{\perp}_3\otimes \Lambda^2\cU^{\vee}_3)^{\vee}\otimes \mathrm{det}(Q_3\otimes \cU_3^{\vee})=\cO(2H_3)$ we obtain that the degeneracy locus of $q_{i_{\lambda}}$ is $\bar{\mathsf{O}}_1$ with multiplicity 2. By Lemma \ref{lemma:degeneracy_locus} we conclude that the scheme-theoretic support of the cokernel sheaf is $\bar{\mathsf{O}}_1$, i.e. it is isomorphic to $i^*\cC(q_{i_{\lambda}})$.

By Proposition \ref{quadric_construction} we obtain that~$\cC(q_{i_{\lambda}})$ is a reflexive sheaf of rank 2.
Now let us compute the kernel of $q_{i_{\lambda}}$ twisted by $\cO(-H_3)$ at the point $P_1$. Using the explicit computations in Proposition \ref{proposition_quadrics_segre} and \eqref{i_P1}  we obtain
\begin{equation} \label{kernel_f_i}
  \mathrm{Ker}(f_{i_{\lambda}})\restr{[P_1]}=\langle e^{\vee}_{\gamma}\otimes e_{\beta}, e^{\vee}_{-\beta}\otimes e_{-\gamma} \rangle \subset (\cU_3^{\perp}\otimes \cU_3) \restr{[P_1]}
\end{equation}
(in the notation of Corollary \ref{corollary:Segre} we let $(x_1,x_2,x_3)=(e^{\vee}_{0}, e^{\vee}_{\beta}, e^{\vee}_{-\gamma})$ and~$(y_1,y_2,y_3, y_4)=(e_{\alpha}, e_{-\alpha}, e_{-\beta}, e_{\gamma})$).

By Lemma \ref{lemma:self_dual_1} over $\mathsf{O}_1$ there exits an embedding
$\tpi^*\cU_2 \otimes \tpi^*\cU_2 \hookrightarrow \cU_3^{\perp} \otimes \cU_3$. Thus, by Corollary~\ref{corollary:pushforward} over $\mathsf{O}_1$ we obtain an embedding $\cO(-H_3)\oplus \pi_2^*\cO(-H_2)\hookrightarrow \cU_3^{\perp} \otimes \cU_3$, which can be explicitly described at the point $P_1$ as follows
\begin{align*}
& e_{\beta} \otimes e_{\gamma} \to -\frac{1}{2} e^{\vee}_{-\beta}\otimes e_{\gamma}  \\
& e_{-\gamma} \otimes e_{\beta} \to -\frac{1}{2} e^{\vee}_{\gamma} \otimes e_{\beta}.
\end{align*}
Using \eqref{kernel_f_i} we conclude 
\begin{equation*}
\cC(q_{i_{\lambda}})\restr{\mathsf{O}_1}\simeq \cO(H_3) \oplus \tpi^*\cO(-H_2)\otimes \cO(2H_3)\simeq \cO(H_3) \oplus \tpi^*\cO(H_2),
\end{equation*}
where the last isomorphism follows from \eqref{picard_on_O1}.
\end{proof}
Recall that $\nu\colon \cO\to \cO(H_3)$ is the $\rG_2$-equivariant map given by the 3-form $\nu$. 
\begin{lemma}
The self-dual map
\begin{equation*}
      l\oplus \nu \colon \cU_3^{\vee}\oplus \cO \to \Lambda^2\cU_3^{\vee}\oplus \cO(H_3)
\end{equation*}
defines a $\rG_2$-equivariant quadric bundle on~$\CG$. Its cokernel sheaf is isomorphic to $i_*\cC(q_{i_{\lambda}})$.
\end{lemma}
\begin{proof}
This statement is a consequence of Lemma \ref{lemma_f3} and Corollary \ref{corollary_l_twisted}. 
\end{proof}
\begin{corollary} \label{corollary_extension}
There exists an extension
\begin{equation} \label{extension_1}
    0 \to \cU_3^{\perp}\otimes \Lambda^2\cU_3^{\vee} \to \cE_{16} \to \Lambda^2\cU_3^{\vee} \oplus \cO(H_3)\to 0
\end{equation}
that is self-dual $\cE_{16}\simeq \cE^{\vee}_{16}(H_3)$.
\end{corollary}
\begin{proof}
Using Proposition \ref{lemma:quadric_bundles_cokernel} for quadric bundles $q_{i_{\lambda}}$ and $l\oplus \nu$ we deduce the statement.
\end{proof}
\section{Fullness} \label{section:fullness}
The proof of the exceptionality of the collection \eqref{collection} is purely computational, so we moved it to Appendix \ref{section:computations}. In this section we prove the fullness of \eqref{collection}.
\subsection{Adding some objects}
Recall notation \eqref{mathfrak_E}. It will be convenient to start with proving fullness of a slightly different collection
\begin{equation} \label{ec}
   \langle \cU, \mathfrak{E}, \cR, \Sigma^{2,1}\cU^{\vee},\mathfrak{E}(1), \cR(1),\mathfrak{E}(2), \cO(3), \cU^{\vee}(3)\rangle.
\end{equation}
obtained from the collection \eqref{collection} by removing its last object $\Lambda^2\cU^{\vee}(3)$ and adding instead $\cU\simeq \Lambda^2\cU^{\vee}(-1)$ on the left. By Proposition \ref{proposition:long-mutations} exceptionality of \eqref{collection} implies exceptionality of \eqref{ec}.
\begin{lemma} \label{lemma:fullness}
The exceptional collection \eqref{collection} is full if and only if the collection \eqref{ec} is full.
\end{lemma}
\begin{proof}
The statement follows from Proposition \ref{proposition:long-mutations}.
\end{proof}
Let us denote by 
\begin{equation} \label{definition:D}
    \cD\coloneqq  \langle \cU, \mathfrak{E}, \cR, \Sigma^{2,1}\cU^{\vee},\mathfrak{E}(1), \cR(1),\mathfrak{E}(2), \cO(3), \cU^{\vee}(3)\rangle
\end{equation}
the subcategory of $\Db(\CG)$ generated by \eqref{ec}. 
\begin{lemma} \label{lemma:Lambda^2Q(i)}
We have $\cU^{\perp}(1), \cQ, \Lambda^2\cQ, \Lambda^2\cQ(1) \in \cD$.
\end{lemma}
\begin{proof}
From the exact sequence \eqref{tautological_exact_sequence} we obtain that $\cQ\in \cD$. From the exact sequence \eqref{tautological_exact_sequence_1} twisted by~$\cO(1)$ we get that $\cU^{\perp}(1)\in \cD$.
The statement for $\Lambda^2\cQ$ and $\Lambda^2\cQ(1)$ follows from the exact sequence \eqref{exact_sequence:R_Lambda^2Q} twisted by~$\cO$ and $\cO(1)$.
\end{proof}
\begin{proposition} \label{proposition:S^2}
We have $S^2\cU^{\vee}(n)\in \cD$ for $n=0,1,2$.
\end{proposition}
\begin{proof}
Consider the (twisted) Koszul complex and its dual
\begin{align} \label{Koszul}
& 0\to S^2\cU(n)\to V\otimes \Lambda^2\cU^{\vee}(n-1)\to  \Lambda^2V^{\vee}\otimes \cO(n)\to  \Lambda^2\cQ(n)\to 0\\ \label{Koszul_dual}
& 0\to \Lambda^2\cQ(n-1)\to \Lambda^2V^{\vee}\otimes \cO(n)\to V^{\vee}\otimes \cU^{\vee}(n)\to S^2\cU^{\vee}(n)\to 0,   
\end{align}
here we use isomorphisms $\Lambda^2\cU^{\vee}(n-1)\simeq \cU(n)$ and $\Lambda^2\cQ(n-1)\simeq \Lambda^2\cU^{\perp}(n)$. 

Using the exact sequence \eqref{Koszul_dual} and Lemma \ref{lemma:Lambda^2Q(i)} we immediately obtain that $S^2\cU^{\vee}(n)\in \cD$ for~$n=1,2$.

Using the exact sequence \eqref{Koszul} for $n=1$ and Lemma \ref{lemma:Lambda^2Q(i)} we get that $S^2\cU(1)\in \cD$. From the exact sequence \eqref{extension_E} twisted by $\cO(1)$ and Lemma \ref{lemma:Lambda^2Q(i)} we obtain that $\cE_{10}(1)\in \cD$. Thus, by Corollary \ref{proposition_E_E} we have $\cE_{10}^{\vee}\in\cD$. Using the dual to \eqref{extension_E} exact sequence
\begin{equation} \label{exact_sequence:E_dual}
    0\to \cQ \to \cE_{10}^{\vee}\to S^2\cU^{\vee} \to 0
\end{equation}
and Lemma \ref{lemma:Lambda^2Q(i)} we conclude that $S^2\cU^{\vee}\in \cD$.
\end{proof}
\begin{proposition} \label{proposition:21}
We have $\Sigma^{2,1}\cU^{\vee}\in \langle \Sigma^{2,1}\cU^{\vee}(-1),\mathfrak{E},\cO(1)\rangle$. 
\end{proposition}
\begin{proof}
Consider the following mutually dual exact sequences restricted from $\Gr(3,V)$:
\begin{align} \label{K_definition}
  & 0\to \cK\to V^{\vee} \otimes \Lambda^2\cU^{\vee}\stackrel{ev}{\longrightarrow} \Sigma^{2,1}\cU^{\vee} \to  0, \\ \label{K_dual}
  & 0\to \Sigma^{2,1}\cU \to V\otimes \Lambda^2\cU\to  \cK^{\vee} \to 0,
\end{align}
where $ev$ is the evaluation map and~$\cK\coloneqq \mathrm{Ker}(V^{\vee}\otimes\Lambda^2\cU^{\vee}\stackrel{ev}{\longrightarrow} \Sigma^{2,1}\cU^{\vee})$. Applying the snake lemma to the diagram
\begin{equation*} 
\xymatrix@R=1pc{
0\ar[r]& \cK\ar[r] \ar[d] & V^{\vee} \otimes \Lambda^2\cU^{\vee}\ar[d] \ar[r]& \Sigma^{2,1}\cU^{\vee} \ar@{=}[d] \ar[r] & 0 \\
0\ar[r]& \cO(1) \ar[r]&\cU^{\vee} \otimes \Lambda^2\cU^{\vee} \ar[r] & \Sigma^{2,1}\cU^{\vee} \ar[r] & 0,
}   
\end{equation*}
we get the following mutually dual exact sequences
\begin{align}\label{exact_sequence_bK'}
    & 0\to \cU^{\perp}\otimes \Lambda^2\cU^{\vee} \to \cK\to \cO(1)\to 0,\\ \label{exact_sequence_bK'_dual}
    &  0\to \cO(-1) \to \cK^{\vee}\to \cQ\otimes \Lambda^2\cU \to 0.
\end{align}
The exact sequence \eqref{K_dual} twisted by $\cO(1)$ shows that  $\cK^{\vee}(1)\in \langle \Sigma^{2,1}\cU^{\vee}(-1), \cU^{\vee}\rangle$. The exact sequence~\eqref{exact_sequence_bK'_dual} twisted by $\cO(1)$ shows that $\cQ\otimes \cU^{\vee}\in \langle \cO, \cK^{\vee}(1)\rangle\subset \langle \Sigma^{2,1}\cU^{\vee}(-1), \cO, \cU^{\vee}\rangle$. The dual exact sequence to~\eqref{extension_1} twisted by $\cO(1)$
\begin{equation*}
    0 \to \cU^{\vee} \oplus \cO  \to \cE_{16}^{\vee}(1) \to \cQ\otimes\cU^{\vee}\to 0
\end{equation*}
shows that $\cE_{16}^{\vee}(1)\in \langle \Sigma^{2,1}\cU^{\vee}(-1), \cO, \cU^{\vee}\rangle$. Applying Corollary~\ref{corollary_extension} we obtain~$\cE_{16} \in \langle \Sigma^{2,1}\cU^{\vee}(-1), \cO, \cU^{\vee}\rangle$. Using the exact sequence \eqref{extension_1} we deduce~$\cU^{\perp}\otimes \Lambda^2\cU^{\vee}\in \langle \cE_{16}, \Lambda^2\cU^{\vee}, \cO(1)\rangle\subset \langle \Sigma^{2,1}\cU^{\vee}(-1), \mathfrak{E}, \cO(1)\rangle$. The exact sequence \eqref{exact_sequence_bK'} shows that $\cK\in \langle \Sigma^{2,1}\cU^{\vee}(-1), \mathfrak{E}, \cO(1)\rangle$. Finally, using the exact sequence \eqref{K_definition} we deduce $\Sigma^{2,1}\cU^{\vee}\in \langle \Sigma^{2,1}\cU^{\vee}(-1),\mathfrak{E},\cO(1)\rangle$.
\end{proof}
\begin{corollary} \label{proposition:Sigma^{2,1}}
We have $\Sigma^{2,1}\cU^{\vee}(n)\in \cD$ for $n=1,2$.
\end{corollary}
\begin{proof}
This directly follows from Proposition \ref{proposition:21}.
\end{proof}
\subsection{Fullness}
In this section we prove the fullness of the collection \eqref{ec} and consequently of the collection \eqref{collection}.
Let us briefly sketch the idea of the proof, that comes from \cite{2}. We consider the family of subvarieties~$\CG_f \stackrel{i_f}{\hookrightarrow}\CG$ defined as zero loci of sufficiently general global sections $f\in \mathrm{H}^0(\CG,\cU^{\vee})$. We show that $\CG_f$ is isomorphic to a smooth hyperplane section of the isotropic Grassmannian~$\IGr(3,6)$ and prove that every point of $\CG$ lies on such a smooth zero locus. Using Koszul resolution of $i_{f*}\cO_{\CG_{f}}$ and the exceptional collection \eqref{exceptional_collection_hyperplane} on $\CG_f$ we check that the pullback to~$\CG_f$ of an object $F\in \cD^{\perp}$ is zero. It follows that~$F=0$, so we conclude that $\cD^{\perp}=0$, hence $\cD=\Db(\CG)$.

To prove the fullness of \eqref{ec} we need some preparations. By Proposition \ref{main_computations} we have~$\mathrm{H}^0(\CG,\cU^{\vee})=V^{\vee}$, so $\mathrm{H}^0(\CG,\cU^{\vee})$ is one of the fundamental representation of $\rG_2$. Recall that the action of $\rG_2$ on $\PP(V) \xrightarrow[q]{\sim}\PP(V^{\vee})$ has two orbits: the 5-dimensional quadric $\qQ^{\vee}=q(\qQ)$ and its open dense complement $\PP(V^{\vee})\setminus \qQ^{\vee}$, see Subsection~\ref{section:G2}. We will denote by $H_f\subset V$ the 6-dimensional vector subspace of $V$ that corresponds to~$[f]\in \PP(V^{\vee})$.
\begin{lemma} \label{lemma:zerolocus}
If $[f]\in \PP(V^{\vee})\setminus \qQ^{\vee}$ then
$$\CG_f\coloneqq \Gr(3,H_{f})\cap \CG \simeq \IGr(3,H_f^{\vee}) \cap H$$ and the right hand side is a smooth hyperplane section of $\IGr(3,H_f^{\vee})$. In particular, we have the following resolution of the structure
sheaf $i_{f*}\cO_{\CG_{f}}$ on $\CG$:
\begin{equation} \label{krM}
    0\to \cO(-1)\to \Lambda^2 \cU \to \cU \to \cO_{\CG}\to i_{f*}\cO_{\CG_{f}}\to 0.
\end{equation}
where $i_{f}\colon \CG_{f}\to \CG$ is the embedding. Moreover, the Cayley Grassmannian $\CG$ is covered by the subvarieties $\CG_{f}$ for~$[f]\in \PP(V^{\vee})\setminus \qQ^{\vee}$.
\end{lemma}
\begin{proof}
The intersection~$\Gr(3,H_{f})\cap \CG$ is isomorphic to the zero locus of~$f\in \mathrm{H}^0(\Gr(3,V), \cU^{\vee})$ on~$\Gr(3,V)$, hence $\CG_f$ is the zero locus of $f\in \mathrm{H}^0(\CG,\cU^{\vee})$. 

The vector bundle $\cU^{\vee}$ is globally generated, hence by Bertini's theorem the zero locus of a general global section is smooth of codimension 3. The subset of $\PP(\mathrm{H}^0(\CG, \cU^{\vee}))$ of $f$ such that $\CG_f$ is smooth is open and $\rG_2$-invariant hence this subset contains the open orbit $\PP(V^{\vee})\setminus \qQ^{\vee}$. In particular, we obtain that for $[f]\notin \qQ^{\vee}$ the zero locus $\CG_f$ is smooth of codimension 3.

We denote by $\cU_{\Gr(3,H_{f})}$ and $\cU^{\perp}_{\Gr(3,H_{f})}$ the tautological vector bundle and the dual to the quotient bundle on $\Gr(3,H_{f})$, respectively. The Pl{\"u}cker line bundle on $\Gr(3,H_{f})$ will be denoted by~$\cO_{\Gr(3,H_{f})}(1)$.
The restriction of $\cU^{\perp}(1)$ to~$\Gr(3,H_{f})$ is isomorphic to
$$\cU^{\perp}(1)\restr{\Gr(3,H_{f})}\simeq \cU^{\perp}_{\Gr(3,H_{f})}(1)\oplus \cO_{\Gr(3,H_{f})}(1),$$
hence we get that $\CG_f$ is isomorhic to a hyperplane section of the zero locus of~$\lambda|_{H_f}\in \mathrm{H}^0(\Gr(3,H_{f}),\cU_{\Gr(3,H_{f})}^{\perp}(1))\simeq \Lambda^4 H_f^{\vee}$ on $\Gr(3,H_{f})$. Under the canonical isomorphism~$\Gr(3,H_f)\simeq \Gr(3, H_f^{\vee})$ the vector bundle $\cU^{\perp}_{\Gr(3,H_{f})}(1)$ maps to the vector bundle~$\cU_{\Gr(3,H_{f}^{\vee})}(1)\simeq \Lambda^2\cU^{\vee}_{\Gr(3,H_{f}^{\vee})}$, hence the zero locus of the 4-form~$\lambda|_{H_f}\in \mathrm{H}^0(\Gr(3,H_{f}),\cU_{\Gr(3,H_{f})}^{\perp}(1))$ on~$\Gr(3,H_f)$ is isomorphic to the zero locus of a dual 2-vector~$(\lambda|_{H_f})^{\vee}\in \mathrm{H}^0(\Gr(3,H_{f}^{\vee}),\Lambda^2\cU_{\Gr(3,H_{f}^{\vee})}^{\vee})\simeq \Lambda^2H_f$ on~$\Gr(3,H_f^{\vee})$. Hence $\CG_{f}$ is isomorphic to a hyperplane section of the isotropic Grassmannian~$\IGr_{(\lambda|_{H_f})^{\vee}}(3,H_f^{\vee})$ and we get the first statement.

We have already shown that $\CG_f$ is smooth of codimension 3. Let us prove that this implies nondegenerateness of~$(\lambda|_{H_f})^{\vee}$. Suppose the rank of $(\lambda|_{H_f})^{\vee}$ is equal to 4, that is $(\lambda|_{H_f})^{\vee}$ has 2-dimensional kernel $U_2\subset H_f^{\vee}$. Then the isotropic Grassmannian~$\IGr_{(\lambda|_{H_f})^{\vee}}(3,H_f^{\vee})$ contains singular $\PP^3=\PP(\Lambda^2 U_2\wedge (H_f^{\vee}/U_2))$. Thus, a hyperplane section of $\IGr_{(\lambda|_{H_f})^{\vee}}(3,H_f^{\vee})$ can not be smooth. If the rank of $(\lambda|_{H_f})^{\vee}$ is equal to 2 then the dimension of~$\IGr_{(\lambda|_{H_f})^{\vee}}(3,H_f^{\vee})$ is
bigger than 6, so its hyperplane section can not have codimension 3 in $\CG$. We conclude that $(\lambda|_{H_f})^{\vee}$ is nondegenerate.

Any section $[f]\in \PP(V^{\vee})\setminus \qQ^{\vee}$ is regular so the sheaf $i_{f*}\cO_{\CG_{f}}$ admits a Koszul resolution which takes the form~\eqref{krM}.

Let us prove the last statement. Take a point $[U_3]$ of $\CG$. Note that
$$[U_3]\in \CG_f \Leftrightarrow U_3 \subset H_f \Leftrightarrow [f]\in \PP(U_3^{\perp}).$$
Suppose that~$[U_3]$ does not lie on any zero locus $\CG_{f}$ with~$[f]\in \PP(V^{\vee})\setminus \qQ^{\vee}$. Then~$\PP^3\simeq \PP(U_3^{\perp})\subset \qQ^{\vee}$ and we get a contradiction because the smooth 5-dimensional quadric $\qQ^{\vee}$ does not contain~$\PP^3$. 
\end{proof}
Now we can prove the main result of this paper. 
\begin{theorem} \label{theorem:main}
The collection \eqref{ec} is full.
\end{theorem}
\begin{proof}
Recall notation \eqref{definition:D}. Assume that collection \eqref{ec} is not full. Then by \cite[Lemma 3.1 and Theorem 3.2(a)]{3} there exists a nonzero object $F\in\cD^{\perp}$. 

We show that $i_{f}^*F=0$ for any global section $[f]\in \PP(V^{\vee})\setminus \qQ^{\vee}$.
Indeed, by Lemma \ref{lemma:zerolocus} the zero locus~$\CG_f$ is isomorphic to a smooth hyperplane section of the isotropic Grassmannian. Recall the Lefschetz decomposition \eqref{exceptional_collection_hyperplane} of $\Db(\CG_f)$:
\begin{equation} \label{ec_1}
    \Db(\CG_f)= \langle \cO_{\CG_f},\cU_{\CG_f}^{\vee},\cO_{\CG_f}(1),\cU_{\CG_f}^{\vee}(1), \cO_{\CG_f}(2),\cU_{\CG_f}^{\vee}(2) \rangle,
\end{equation}
where to avoid confusion we denote by $\cU_{\CG_f}$ and $\cO_{\CG_f}(1)$ the restrictions to $\CG_f$ of the tautological and the Pl{\"u}cker line bundles respectively.

First, tensoring the resolution \eqref{krM} by $\cO(-k)\otimes F$ we obtain the complex
\begin{equation} \label{complex_F_1}
  0\to \cO(-k-1)\otimes F\to \Lambda^2 \cU(-k)\otimes F \to \cU(-k)\otimes F \to \cO(-k)\otimes F\to
 i_{f *}(\cO_{\CG_f}(-k) \otimes i_{f}^{*}F)\to 0,  
\end{equation}
where we use an isomorphism $(\cO(-k)\otimes F)\otimes i_{f *} \mathcal{O}_{\CG_{f}}\simeq i_{f *}(i_{f}^{*}( \cO(-k) \otimes F)) \simeq i_{f *}(\cO_{\CG_f}(-k) \otimes i_{f}^{*}F).$
Since~$\Lambda^j\cU^{\vee}(k)\in \cD$ for $0\le j\le 3,\ 0\le k \le 2$ and $F\in \cD^{\perp}$ we get
$$\mathrm{H}^{\bullet}(\CG,\Lambda^j\cU(-k)\otimes F)\simeq \mathrm{Ext}^{\bullet}(\Lambda^j\cU^{\vee}(k),F)=0,$$
so that the cohomology on $\CG$ of all terms of the complex \eqref{complex_F_1} except for the last one vanishes. Thus the cohomology of the last term also vanishes, and we obtain \begin{equation} \label{rhom_1}
\mathrm{H}^{\bullet}(\CG, i_{f *}(\cO_{\CG_f}(-k) \otimes i_{f}^{*}F))= \mathrm{H}^{\bullet}(\CG_{f}, \cO_{\CG_f}(-k) \otimes i_{f}^{*}F)=\mathrm{Ext}^{\bullet}_{\CG_{f}}(\cO_{\CG_f}(k), i_{f}^{*} F)=0. \end{equation}

Next, tensoring the resolution \eqref{krM} by $\cU(-k)\otimes F$, where $0\le k \le 2$, we obtain the complex
\begin{multline} \label{complex:F_2}
   0\to \cU(-k-1)\otimes F\to \Lambda^2 \cU \otimes \cU(-k) \otimes F \to \\
   \to \cU\otimes \cU(-k)\otimes F \to \cU(-k)\otimes F\to i_{f *}(\cU_{CG_f}(-k) \otimes i_{f}^{*}F)\to 0,
\end{multline}
where we use an isomorphism $(\cU(-k)\otimes F)\otimes i_{f *} \mathcal{O}_{\CG_{f}}\simeq i_{f *}(i_{f}^{*}( \cU(-k) \otimes F)) \simeq i_{f *}(\cU_{\CG_{f}}(-k) \otimes i_{f}^{*}F)$.
By Proposition~\ref{proposition:S^2} and Corollary~\ref{proposition:Sigma^{2,1}} we have $\cU^{\vee}\otimes \cU^{\vee}(k)\simeq S^2\cU^{\vee}(k)\oplus \Lambda^2\cU^{\vee}(k)\in \cD$ and $\Lambda^2\cU^{\vee}\otimes \cU^{\vee}(k)\simeq \Sigma^{2,1}\cU^{\vee}(k)\oplus \cO(k+1)\in \cD$ for $0\le k \le 2$, hence for~$F\in \cD^{\perp}$ we obtain
$$\mathrm{H}^{\bullet}(\CG, \Lambda^j\cU\otimes \cU(-k) \otimes F)\simeq \mathrm{Ext}^{\bullet}(\Lambda^j\cU^{\vee}\otimes \cU^{\vee}(k),F)=0$$
and the cohomology on $\CG$ of all terms of the complex \eqref{complex:F_2} except for the last one vanishes.
We deduce that the cohomology of the last term also vanishes, and we obtain
\begin{equation} \label{rhom_2}
\mathrm{H}^{\bullet}(\CG, i_{f *}(\cU_{\CG}(-k) \otimes i_{f}^{*}F)= \mathrm{H}^{\bullet}(\CG_{f}, \cU_{\CG_f}(-k) \otimes i_{f}^{*}F) = \mathrm{Ext}^{\bullet}_{\CG_{f}}(\cU_{\CG_f}^{\vee}(k), i_{f}^{*} F)=0.    
\end{equation}
Using the full exceptional collection $\eqref{ec_1}$ and the equalities \eqref{rhom_1} and \eqref{rhom_2} we deduce~$i_{f}^*F\in \mathrm{D}^b(\CG_f)^{\perp}=0$.

Now, using the last statement of Lemma \ref{lemma:zerolocus} and \cite[Lemma 3.29]{H} we deduce that $F=0$. Thus we have proved that the collection \eqref{ec} is indeed full.
\end{proof}
We can also prove the theorem stated in the Introduction. 
\begin{proof}[Proof of Theorem \ref{Theorem:introduction:main}]
The statement immediately follows from Lemma \ref{lemma:fullness} and Theorem \ref{theorem:main}.
\end{proof}
\appendix \section{Geometric constructions} \label{section:geometric_construction}
In this appendix we prove some geometric results related to the Cayley Grassmannian~$\CG$: we show that $\CG$ is isomorphic to the Hilbert scheme of conics on $\Gad$ and describe the Hilbert scheme of lines on $\CG$. 

Recall that we denote by $\cO(H_2)$ and~$\cU_2$ the Pl{\"u}cker line bundle and the tautological bundle on $\Gr(2,V)$ respectively. The quotient bundle and its dual are denoted by $\cQ_2$ and $\cU_2^{\perp}$.
\subsection{The Cayley Grassmannian as the Hilbert scheme}
\label{section:Hilbert} In this section we prove that the Cayley Grassmannian admits another description: it is the Hilbert scheme of conics on the adjoint variety $\Gad$.

So far we considered the adjoint Grassmannian $\Gad\subset \Gr(2,V)$ associated with the 3-form~$\nu=q(\lambda^{\vee})$. Now it is more convenient to consider its image $\Gad(\lambda^{\vee})\coloneqq q(\Gad)\subset \Gr(2,V^{\vee})$ in the dual Grassmannian, as the notation suggests, it is the adjoint variety associated with~$\lambda^{\vee}$.
\begin{proposition} \label{proposition:CGHilb}
The Cayley Grassmannian $\CG\subset \Gr(3,V)\simeq \Gr(4,V^{\vee})$ is isomorphic to the Hilbert scheme of conics on the adjoint variety $\Gad(\lambda^{\vee})$. The universal family of conics is the conic bundle $\bQ_{\cR}$ defined in \eqref{conic_bundles_intersections}, so that we have the following diagram
\begin{equation} \label{universal_family}
\xymatrix@C-=0.5cm{
&& \bQ_{\cR} \ar[dl]_{\psi_1} \ar[dr]^{\psi_2} &&\\
\Gr(2, V^{\vee})& \Gad(\lambda^{\vee}) \ar@{_{(}->}[l] && \CG\ar@{^{(}->}[r]& \Gr(4, V^{\vee}).
}
\end{equation}
\end{proposition}
\begin{proof}
Recall that the canonical isomorphism $\Gr(3,V)\simeq \Gr(4,V^{\vee})$ is defined by $ (U_3\subset V) \mapsto (U_3^{\perp}\subset V^{\vee})$. It is easy to see that this isomorphism takes $\CG$ to the subvariety of $\Gr(4,V)$ that parametrizes all 4-dimensional subspaces of $V^{\vee}$ isotropic with respect to the 3-form~$\lambda^{\vee}\in \Lambda^3 V$.

Twisting \eqref{exact_sequence:R_Lambda^2Q} by $\cO(-1)$ and using \eqref{preliminaries_isomorphisms2} we get the following exact sequence:
\begin{equation} \label{exact_sequence_R(-1)}
    0\to \cR(-1) \to \Lambda^2\cU^{\perp} \stackrel{\lambda^{\vee}}{\longrightarrow} \cU \to 0,
\end{equation}
where the map $\Lambda^2\cU^{\perp} \stackrel{\lambda^{\vee}}{\longrightarrow} \cU$ is given by the convolution with the 3-vector~$\lambda^{\vee}\in \Lambda^3V$. We have $\Gr_{\CG}(2,\cU^{\perp})= \CG\times_{\Gr(4,V^{\vee})} \mathrm{Fl}(2,4,V^{\vee})$ hence there is a canonical projection~$\PP_{\CG}(\cR(-1))\cap \Gr_{\CG}(2,\cU^{\perp})\to \Gr(2, V^{\vee})$ so that using \eqref{conic_bundles_intersections} we get $\psi_1\colon \bQ_{\cR}\to \Gr(2, V^{\vee})$. The exact sequence \eqref{exact_sequence_R(-1)} shows that the intersection~$\PP_{\CG}(\cR(-1))\cap \Gr_{\CG}(2,\cU^{\perp})\subset \PP_{\CG}(\Lambda^2\cU^{\perp})$ consists of those 2-dimensional subspaces in~$\cU^{\perp}\subset V^{\vee}\otimes \cO$ that are annihilated by $\lambda^{\vee}$. Since $\Gad(\lambda^{\vee})$ is the locus of 2-dimensional subspaces of $V^{\vee}$ annihilated by $\lambda^{\vee}$ we deduce that the image of $\psi_1$ lies in $\Gad(\lambda^{\vee})$, so we obtain the diagram \eqref{universal_family}. Moreover, every fiber of $\psi_2$ becomes a conic on $\Gad(\lambda^{\vee})$.

From the universal property of the Hilbert scheme of conics $\mathrm{Hilb}^{1+2t}(\Gad(\lambda^{\vee}))$ we get a unique morphism 
\begin{equation*}
  m\colon \CG \to  \mathrm{Hilb}^{1+2t}(\Gad(\lambda^{\vee})),
\end{equation*}
such that $\bQ_{\cR}$ is isomorphic to the pullback of the universal family of conics on $\mathrm{Hilb}^{1+2t}(\Gad(\lambda^{\vee}))$.

Let us prove that $\bQ_{\cR}$ is the universal family of conics on $\Gad(\lambda^{\vee})$. Recall that the Hilbert scheme of planes in $\Gr(2, V^{\vee})$ has two connected components: one is $\mathrm{Fl}(1,4; V^{\vee})$ (it parameterizes planes of the form $\PP(U_4/U_1)$, where $[U_1\subset U_4]\in \mathrm{Fl}(1,4; V^{\vee})$) and the other one is $\Gr(3,V^{\vee})$ (it parameterizes planes of the form $\PP(U^{\vee}_3)$, where $[U_3]\in \Gr(3,V^{\vee})$). Recall that conics in $\Gr(2,V^{\vee})$ can be divided into three different
classes, according to the type of their linear span in~$\PP(\Lambda^2V^{\vee})$: $\tau$-conics, $\sigma$-conics and $\rho$-conics, see Subsection \ref{classes_of_conics}.

First let us show that there are no conics of types $\sigma$ or $\rho$ on $\Gad(\lambda^{\vee})$. Indeed, by Lemma \ref{adjoint_properties} the adjoint variety~$\Gad(\lambda^{\vee})$ is a linear section of $\Gr(2,V^{\vee})$. Hence we obtain that any conic of type $\sigma$ or $\rho$ on~$\Gad(\lambda^{\vee})$ would span a plane on $\Gad(\lambda^{\vee})$, but there are no planes on $\Gad(\lambda^{\vee})$ by \cite[Lemma 3.1]{KR}.

On the other hand, for a conic that belongs to type $\tau$ there exits a unique 4-dimensional subspace~$U_4\subset V^{\vee}$ such that~$C\subset \Gr(2,U_4)$. Therefore, $\mathrm{Hilb}^{1+2t}(\Gad(\lambda^{\vee}))\subset \Gr(4, V^{\vee})$, and by definition of $\bQ_{\cR}$ the composition $\CG\xrightarrow{m} \mathrm{Hilb}^{1+2t}(\Gad(\lambda^{\vee}))\to \Gr(4, V^{\vee})$ coincides with the canonical embedding $\CG\subset \Gr(4, V^{\vee})$; in particular, $m$ is a closed embedding. Let us show the surjectivity of $m$, i.e. let us prove that the subspace $U_4\subset V^{\vee}$ corresponding to a $\tau$-conic $C\subset \Gad(\lambda^{\vee})$ lies on the Cayley Grassmannian~$\CG$. Indeed, $C$ lies in $$\Gr(2, U_4)\cap \Gad(\lambda^{\vee})= \Gr(2, U_4)\cap (\Gr(2, V^{\vee})\cap \PP^{13})=\Gr(2, U_4)\cap \PP^{13},$$ a linear section of $\Gr(2, U_4)$ that is contained in $\Gad(\lambda^{\vee})$. Since $\Gad(\lambda^{\vee})$ contains no planes and no quadric surfaces, such a linear section must be a conic. Therefore $C=\Gr(2, U_4)\cap \Gad(\lambda^{\vee})$ and~$C$ is the zero locus of
\begin{equation*}
   \cU_2^{\perp}(1)|_{\Gr(2,U_4)}\simeq \cU^{\perp}_{\Gr(2,U_4)}(1)\oplus \cO(1)^{\oplus 3}
\end{equation*}
on $\Gr(2, U_4)$. If $U_4\notin \CG$ the zero locus of~$\cU^{\perp}_{\Gr(2,U_4)}(1)$ is isomorphic to $\Gr(2,3)$ and $C$ should have type~$\rho$, that is the contradiction. Hence we get that $[C]\in \mathrm{Hilb}^{1+2t}(\Gad(\lambda^{\vee}))$ belongs to the image of $m$. 

We deduce that $$\mathrm{Hilb}^{1+2t}(\Gad(\lambda^{\vee}))_{\mathrm{red}}\simeq \CG.$$ By \cite[Proposition 3.6]{CHK} the Hilbert scheme~$\mathrm{Hilb}^{1+2t}(\Gad(\lambda^{\vee}))$ is smooth, so we get the statement. 
\end{proof}

\subsection{The Hilbert scheme of lines on the Cayley Grassmannian}
In this subsection we describe the Hilbert scheme of lines on~$\CG$.

We need to describe first some special subvariety of~$\Gr(2,V)$. Consider the 4-form $\lambda \in \Lambda^4V^{\vee}$ as a global section of the vector bundle~$\Lambda^2\cU^{\perp}_{2}(H_2)$.
It gives a skew-symmetric morphism of vector bundles
\begin{equation*}
    \cQ_2\stackrel{\varphi_{\lambda}}{\longrightarrow} \cU^{\perp}_2(H_2).
\end{equation*}
Denote by $$\mathrm{LieGr}(2,V)\subset \Gr(2,V)$$ the locus where~$\varphi_{\lambda}$ has rank at most 2. We choose the notation due to the fact, see Lemma \ref{lemma:lie_is_lie} below, that the locus~$\mathrm{LieGr}(2,V)$ parametrizes two-dimensional Lie subalgebras in~$V$ with respect to $q_{\scaleto{\Lambda^2\cU_3}{6.5pt}}$. Note that since $\varphi_{\lambda}$ is skew-symmetric, its rank at every point of $\Gr(2,V)$ is even.
\begin{lemma} \label{lemma:D_k}
The map $\varphi_{\lambda}$ has rank 4 over a general point of~$\Gr(2,V)$ and has rank $2$ over $\rG_2$-invariant locus $\mathrm{LieGr}(2,V)\subset \Gr(2,V)$. The locus where the map $\varphi_{\lambda}$ has zero rank is empty. 
\end{lemma}
\begin{proof}
Let us give a precise description of the map~$\varphi_{\lambda}$. For a two-dimensional space~$U_2\subset V$ consider the 2-form $\lambda\conv U_2$ obtained by convolution of $\lambda$ with~$U_2$. The 2-form $\lambda\conv U_2$ is defined on the 5-dimensional quotient space $V/U_2$ and the map $\varphi_{\lambda}$ at the point~$U_2$ is the canonical convolution with the 2-form
\begin{equation*}
    V/U_2\stackrel{\lambda\conv U_2}{\longrightarrow} (V/U_2)^{\vee}=U^{\perp}_2.
\end{equation*}
So the rank of $\varphi_{\lambda}$ over a point $[U_2]\in \Gr(2,V)$ is equal to the rank of the 2-form $\lambda\conv U_2$ on $V/U_2$. Let us take the 2-dimensional subspace $U'_2=\langle e_{\alpha}, e_{-\alpha}\rangle\subset V$, then using \eqref{lambda} we have
\begin{equation*}
    \lambda\conv U'_2=e_{\gamma}^{\vee}\wedge e_{-\gamma}^{\vee}+ e_{\beta}^{\vee}\wedge e_{-\beta}^{\vee}
\end{equation*}
and we see that $\lambda\conv U'_2$ has the maximal rank. Hence the map $\varphi_{\lambda}$ has rank 4 over a general point of~$\Gr(2,V)$ and we get the first statement. 

Suppose that there exists a 2-dimensional space $U_2\subset V$ such that $\lambda\conv U_2=0$. Let us take any 3-dimensional space $U_3\subset V$ that contains $U_2$. Then $\lambda\conv U_3=0$ hence $U_3\in \CG$. Consider the map $i_{\lambda}$ from Lemma  \ref{lemma:embedding} at the point $U_3$:
\begin{equation*}
    \Lambda^2U_3 \stackrel{i_{\lambda}}{\longrightarrow} \Lambda^2U_3^{\perp}.
\end{equation*}
Since $\lambda\conv U_2=0$ the subspace $\Lambda^2U_2\subset \Lambda^2U_3$ lies in the kernel of $i_{\lambda}$ we get a contradiction with Lemma~\ref{lemma:embedding}. 

The locus $\mathrm{LieGr}(2,V)\subset \Gr(2,V)$ is obviously $\rG_2$-invariant.
\end{proof}
Let us denote by $e\colon \mathrm{LieGr}(2,V) \hookrightarrow \Gr(2,V)$ the embedding of $\mathrm{LieGr}(2,V)$ in $\Gr(2,V)$. From Lemma~\ref{lemma:D_k} we get the following exact sequence on $\mathrm{LieGr}(2,V)$:
\begin{equation} \label{exact_sequence_Lie}
   0\to \tcE \to e^*(\cQ_2)\stackrel{\varphi_{\lambda}}{\longrightarrow} e^*(\cU_2^{\perp}(H_2))\to \tcE^{\vee}(H_2)\to 0,
\end{equation}
where $\tcE$, the kernel of $\varphi_{\lambda}\restr{\mathrm{LieGr}(2,V)}$, by Lemma \ref{lemma:D_k} is a vector bundle of rank 3.
\begin{proposition} \label{corollary:birationality}
Consider the diagram
\begin{equation} \label{diagram:conic_bundle}
\xymatrix{
& \Gr_{\CG}(2,\cU)=\PP_{\CG}(\Lambda^2\cU)\ar[dl]_(0.55){\varphi_1} \ar[dr]^(0.55){\varphi_2}\\ 
\Gr(2,V) && \CG,
}
\end{equation}
where $\varphi_2$ is the projective bundle and $\varphi_1$ is given by $$\Gr_{\CG}(2,\cU)\hookrightarrow \Gr_{\CG}(2, V\otimes \cO) = \CG\times \Gr(2,V)\to \Gr(2,V).$$
The locus $\mathrm{LieGr}(2,V)$ is a smooth 7-dimensional $\rG_2$-invariant subvariety of $\Gr(2,V)$ and
the map $\varphi_1$ is the blow-up with center in $\mathrm{LieGr}(2,V)$. The exceptional divisor $\tE$ of $\varphi_1$ is isomorphic to $\PP_{\mathrm{LieGr}(2,V)}(\tcE)$. The projectivization~$\PP_{\mathrm{LieGr}(2,V)}(\tcE^{\vee})$ is isomorphic to the Hilbert scheme of lines $\mathrm{Hilb}^{1+t}(\CG)$ on the Cayley Grassmannian $\CG$. 
\end{proposition}
\begin{proof}
By Lemma \ref{lemma:CGlinear} the Cayley Grassmannian $\CG$ is a linear section of $\Gr(3,V)$, so for a point~$[U_2]\in \Gr(2,V)$ we obtain
$$\varphi_1^{-1}([U_2])=\PP(V_{27}\oplus \Bbbk)\cap \PP(\Lambda^2 U_2\wedge (V/U_2))\subset \PP(\Lambda^3 V),$$
i.e. the fibers of $\varphi_1$ are linear subspaces in the fibers of $\PP_{\Gr(2,V)}(\cQ_2)\to \Gr(2,V)$. By the definition of~$\CG$ a 3-dimensional subspace~$U_3$ that contains~$U_2$ lies on $\CG$ if and only if~$\lambda\conv U_3=0$ and this holds if and only if the line $U_3/U_2\subset V/U_2$ is contained in the kernel of $\lambda\conv U_2$. Hence we see that~$\varphi_1^{-1}([U_2])\subset \PP(V/U_2)$ is identified with the kernel of the 2-form~$\lambda\conv U_2\in \Lambda^2(V/U_2)^{\vee}$.

By Lemma \ref{lemma:D_k} for $[U_2]\in \Gr(2,V)\setminus \mathrm{LieGr}(2,V)$ the rank of the 2-form $\lambda\conv U_2$ on $V/U_2$ is equal to~4, i.e. $\lambda\conv U_2$ has 1-dimensional kernel, so we get that the scheme fiber of $\varphi_1$ over~$[U_2]\in \Gr(2,V)\setminus \mathrm{LieGr}(2,V)$ is one point.

Analogously, by Lemma \ref{lemma:D_k} for a point $[U'_2]\in\mathrm {LieGr}(2,V)$ the fiber of $\varphi_1$ is $\PP(\tcE_{U'_2})\simeq \PP^2$, so that over $\mathrm{LieGr}(2,V)$ the projection $\varphi_1$ induces the $\PP^2$-bundle $\PP_{\mathrm{LieGr}(2,V)}(\tcE)\to \mathrm{LieGr}(2,V)$.

Now let us prove that $\PP_{\mathrm{LieGr}(2,V)}(\tcE^{\vee})$ is isomorphic to the Hilbert scheme of lines $\mathrm{Hilb}^{1+t}(\CG)$. Indeed, the Hilbert scheme $\mathrm{Hilb}^{1+t}(\Gr(3,V))$ is isomorphic to $\mathrm{Fl(2,4;V)}$, i.e. it is isomorphic to the relative Hilbert scheme $\mathrm{Hilb}^{1+t}(\mathrm{Fl}(2,3,V)/\Gr(2,V))$ of lines contained in the fibers of the projection~$\mathrm{Fl}(2,3; V)\to \Gr(2,V)$. Hence we obtain $$\mathrm{Hilb}^{1+t}(\CG)\simeq \mathrm{Hilb}^{1+t}(\Gr_{\CG}(2,\cU)/\Gr(2,V))\simeq \Gr_{\mathrm{LieGr}(2,V)}(2,\tcE)\simeq \PP_{\mathrm{LieGr}(2,V)}(\tcE^{\vee}).$$

By \cite[Proposition 3.1 and Proposition 3.2]{BM} the Hilbert scheme $\mathrm{Hilb}^{1+t}(\CG)\simeq \PP_{\mathrm{LieGr}(2,V)}(\tcE^{\vee})$ is a smooth and irreducible 9-dimensional variety. Hence we get that $\mathrm{LieGr}(2,V)$ is a smooth and irreducible 7-dimensional subvariety of $\Gr(2,V)$ which is $\rG_2$-invariant by Lemma \ref{lemma:D_k}. Also we get that~$\PP_{\mathrm{LieGr}(2,V)}(\tcE)$ is smooth and irreducible. So the projection $\varphi_1$ is the blow-up of $\mathrm{LieGr}(2,V)$ by \cite[Theorem 1.1]{ESB}.
\end{proof}
In fact, the exceptional divisor $\tE$ is isomorphic to the conic bundle $\bQ_{\Lambda^2\cU_3}\subset \PP_{\CG}(\Lambda^2\cU)$ from Corollary \ref{lemma_l}.
\begin{lemma} \label{lemma:conic_bundle_Lie}
We have $\tE\simeq \bQ_{\Lambda^2\cU_3}$, so that the restriction of the projection~$\varphi_2$ in the diagram \eqref{diagram:conic_bundle} to the exceptional divisor $\tE\simeq \PP_{\mathrm{LieGr}(2,V)}(\tcE)$ is the conic fibration corresponding to $q_{\scaleto{\Lambda^2\cU_3}{6.5pt}}$. We have the following diagram for~$\bQ_{\Lambda^2\cU_3}$
\begin{equation}\label{diagram:E}
\vcenter{\xymatrix{
& \tE \simeq \bQ_{\Lambda^2\cU_3} \ar[dl]_{\varphi_1\restr{\tE}} \ar[dr]^{\varphi_2\restr{\tE}}\\
\mathrm{LieGr}(2,V) && \CG.
}}
\end{equation}
\end{lemma}
\begin{proof}
To avoid confusion the Pl{\"u}cker line bundle on $\CG$ will be denoted here by $\cO(H_3)$.  We will use the notation from the diagram \eqref{diagram:conic_bundle}. 

Let us compute the class of the divisor $\tE\subset \PP_{\CG}(\Lambda^2\cU)$.
The canonical line bundle $\omega_{\PP(\Lambda^2\cU)}$ of the projectivization $\PP_{\CG}(\Lambda^2\cU)$ is isomorphic to
$$\cO_{\PP_{\CG}(\Lambda^2\cU)}(-3)\otimes \varphi_2^*(\mathrm{det}(\Lambda^2\cU^{\vee})\otimes \omega_{\CG})\simeq \varphi_1^*\cO(-3H_2)\otimes \varphi_2^*\cO(-2H_3).$$
Using the above computation of $\omega_{\PP(\Lambda^2\cU)}$ and the blow-up formula we obtain
$$2\tE= \omega_{\scaleto{\PP_{\CG}(\Lambda^2\cU)}{6.5pt}}\otimes \varphi_1^*\omega^{-1}_{\scaleto{\Gr(2,V)}{6.5pt}}\simeq \varphi_1^*\cO(4H_2)\otimes \varphi_2^*\cO(-2H_3).$$
The Picard group of $\PP_{\CG}(\Lambda^2\cU)$ is torsion-free, indeed, $\PP_{\CG}(\Lambda^2\cU)$ is a $\PP^2$-bundle over $\CG$ and the Picard group of $\CG$ is torsion-free. Thus, we obtain $\cO(\tE)\simeq \varphi_1^*\cO(2H_2)\otimes \varphi_2^*\cO(-H_3).$
Since $\tE$ is the exceptional divisor we have $\mathrm{H}^0(\PP_{\CG}(\Lambda^2\cU),\varphi_1^*\cO(2H_2)\otimes \varphi_2^*\cO(-H_3))=\Bbbk$. Recall the conic bundle $\bQ_{\Lambda^2\cU_3} \subset \PP_{\CG}(\Lambda^2\cU)$ is the zero locus of~$\cO_{\PP(\Lambda^2\cU)}(2)\otimes \varphi_2^*\cO(-H_3)\simeq \varphi_1^*\cO(2H_2)\otimes \varphi_2^*\cO(-H_3)$, so we conclude that $\bQ_{\Lambda^2\cU_3}$ coincide with~$\tE$.
\end{proof}

Let us describe the vatiety $\Lie(2,V)$ more precisely.
\begin{lemma} \label{lemma:lie_is_lie}
The variety $\mathrm{LieGr}(2,V)$ parametrizes two-dimensional Lie subalgebras in~$V$ with respect to $q_{\scaleto{\Lambda^2\cU_3}{6.5pt}}$.
\end{lemma}
\begin{proof}
By Lemma \ref{lemma:lie_sructure} the Cayley Grassmannian $\CG$ parametrizes 3-dimensional Lie algebras in $V$ with respect to $q_{\scaleto{\Lambda^2\cU_3}{6.5pt}}$. Note that 2-dimensional Lie subalgebras in a fixed 3-dimensional Lie algebra~$U_3$ form the subscheme in~$\Gr(2,U_3)=\PP(\Lambda^2U_3)$ given by the equation
\begin{equation*}
   \{ u \wedge l(u)=0 \},
\end{equation*}
where $u\in \Lambda^2U_3$. This subscheme coincides with the conic given by the self-dual map $q_{\scaleto{\Lambda^2\cU_3}{6.5pt}}$. So the $\rG_2$-equivariant conic bundle in $\PP_{\CG}(\Lambda^2\cU)$ that parametrizes 2-dimensional Lie subalgebras coincides with~$\bQ_{\Lambda^2\cU_3}$. 

By Proposition \ref{corollary:birationality} the projection $\varphi_1$ is surjective, i.e. any 2-dimensional subspace is contained in a 3-dimensional subspace that belongs to $\CG$. In particular, any 2-dimensional Lie subalgebra in $V$ is contained in a 3-dimensional Lie subalgebra in $V$, so we deduce the statement.
\end{proof}
Let us now describe the orbits of the action of $\rG_2$ on~$\Lie(2,V)$. 
Recall that there are two isomorphism classes of 2-dimensional Lie algebras: abelian and non-abelian. 
\begin{lemma}
We have
\begin{enumerate}
    \item $\Gad\subset \Lie(2,V)$;
    \item the action of $\rG_2$ on $\Gad$ and its complement is transitive;
    \item $\Gad$ and its complement parametrize abelian and non-abelian Lie algebras, respectively.
\end{enumerate}  
\end{lemma}
\begin{proof}
Using \eqref{definition_l_1} we immediately get that the abelian Lie algebras with respect to $q_{\scaleto{\Lambda^2\cU_3}{6.5pt}}$ are parametrized by $\Gad$. In particular, $\Gad\subset \Lie(2,V)$. Thus, to prove the statement it is enough to show that $\rG_2$ acts transitively on 2-dimensional non-abelian Lie algebras.

Recall that by Lemma \ref{lemma:lie_sructure} the Cayley Grassmannian $\CG$ parametrizes 3-dimensional Lie algebras in $V$ with respect to $q_{\scaleto{\Lambda^2\cU_3}{6.5pt}}$. Consider a 2-dimensional non-abelian Lie algebra $U_2$. Let us show that there exists a 3-dimensional Lie algebra corresponding to a point of $\mathsf{O}_0$, that contains~$U_2$. By Proposition~\ref{corollary:birationality} the fiber of $\varphi_1$ over $[U_2]\in \Lie(2,V)$ is isomorphic to $\PP(\tcE\restr{[U_2]})\simeq \PP^2$. We need to prove that $$\varphi_2(\PP(\tcE\restr{[U_2]}))\simeq \PP(\Lambda^2 U_2\wedge (\tcE\restr{[U_2]}))\subset \CG$$ contains a 3-dimensional Lie algebra corresponding to a point of $\mathsf{O}_0$.

Using the description of the closed orbit~$\mathsf{O}_2$ from Lemma \ref{lemma:lie_sructure} we obtain that any 2-dimensional Lie subalgebra of a 3-dimensional Lie algebra corresponding to a point of the closed orbit $\mathsf{O}_2$ is abelian, so $\varphi_2(\PP(\tcE\restr{[U_2]}))$ consists of 3-dimensional Lie algebras corresponding to points of $\mathsf{O}_0$ or $\mathsf{O}_1$.

Suppose now that~$\varphi_2(\PP(\tcE\restr{[U_2]}))\subset \mathsf{O}_1$.  By Lemma \ref{lemma:computation_of_O(E)} we have $$\pi_1^*\cO(H_3) \simeq \cO_{\PP_{\Gad}(\cU^{\perp}_2/\cU_2)}(1)\otimes \pi_2^*\cO(H_2),$$ so the pushforward $\pi_{2*}(\varphi_2(\PP(\tcE\restr{[U_2]})))$ is a linearly embedded plane on $\Gad$, but there are no planes on $\Gad$ by \cite[Lemma 3.1]{KR}, a contradiction. We deduce that $\varphi_2(\PP(\tcE\restr{[U_2]}))$ contains a 3-dimensional Lie algebra corresponding to a point of $\mathsf{O}_0$.

We have shown that $U_2$ is a subalgebra of a 3-dimensional Lie algebra that belongs to the orbit~$\mathsf{O}_0$. By Proposition~\ref{proposition:orbits} we can assume that $U_2\subset P_0=\langle e_0,e_{\gamma},e_{-\gamma} \rangle$ and to prove the statement of the lemma it is enough to show that the stabilizer $\mathrm{Stab}_{P_0}$ of $P_0$ in $\rG_2$ acts transitively on the conic that parametrizes 2-dimensional Lie subalgebras in $P_0$, that is $\bQ_{\Lambda^2\cU_3}\restr{P_0}$ by Lemma \ref{lemma:conic_bundle_Lie}. By Proposition 2.4 in \cite{M} we have $$\mathrm{Stab}_{P_0}\simeq\mathrm{SL}_2\times \mathrm{SL}_2,$$
where one copy of $\mathrm{SL}_2$ corresponds to the subgroup in $\rG_2$ generated by the short root $\gamma$. It is easy to see that this $\mathrm{SL}_2$ acts transitively on $\bQ_{\Lambda^2\cU_3}\restr{P_0}$. Thus, we deduce the statement.  
\end{proof}
Now we have the following funny observation. 
\begin{proposition} \label{lemma:Lie(2,7)}
The variety $\mathrm{LieGr}(2,V)$ is isomorphic to $\OGr(2,V)$
\begin{equation*}
    s\colon \OGr(2,V)\simeq \mathrm{LieGr}(2,V).
\end{equation*}
\end{proposition}
\begin{remark}
The isomorphism $s$ does not agree with the embedding of $\OGr(2,V)$ to $\Gr(2,V)$: the pullback of the tautological bundle $\cU_2$ of $\Lie(2,V)\subset \Gr(2,V)$ under $s$ is the dual spinor bundle~$s^*\cU_2\simeq \cS^{\vee}$.
\end{remark}
\begin{proof}
Let us compute the stabilizer 
\begin{equation*}
    H:= \mathrm{Stab}_{\rG_2}P_{\Lie}
\end{equation*}
of the point $P_{\Lie}=\langle e_{0}, e_{\alpha} \rangle$.
Recall  the map~$\br\colon \Lambda^2V \to V$ that is given by $q^{-1}(\nu\conv (-\wedge -))$ and the notation in the diagram \eqref{g2}. The next proof is almost a repetition of the proof of Proposition 2.34 in \cite{P}.

First note that $H$ contains a maximal torus $T$ of $\rG_2$ (because the basis \eqref{weight_decomposition} is torus-invariant by definition).  Since $[e_{0}, e_{\alpha}]=-2e_{\alpha}$ the group $H$ stabilizes $[e_{\alpha}]\in \qQ$ (see its equation \eqref{q:equation}) so that~$H$ is contained in $\cP_2$, the parabolic subgroup. The torus $T$ acts on $\langle e_{\alpha}\rangle$ with weight a short root $\alpha$. Recall that we have the decomposition~\eqref{V14} of $\Lambda^2V$ and that the representation $V_{14}$ is the kernel of the map~$\br$.
Then, the only 1-dimensional subspace of $V_{14}$ where $T$ acts with weight $\alpha$ is the subspace 
\begin{equation*}
    L:= \langle e_0\wedge e_{\alpha}+e_{-\beta}\wedge e_{-\gamma}\rangle.
\end{equation*}
Let us prove that $H$ is the stabilizer of $L$. Suppose $\phi \in H$. We have already shown that $\phi(e_{\alpha})=\lambda e_{\alpha}$ for~$\lambda\in \Bbbk$. Note that $\phi(e_0)$ must be of the form $e_0+ a e_{\alpha}$ for some $a\in \Bbbk$ since $\phi\in \rG_2$ preserves the symmetric form $q$. So we have $\phi(e_0 \wedge e_{\alpha})=\lambda e_0\wedge e_{\alpha}$. We have $[e_{\alpha}, e_{-\beta}]=0$ and $[e_{\alpha}, e_{-\gamma}]=0$. Since~$\phi\in \rG_2$ preserves the map $\br$ we get that $\phi(e_{-\beta})\in \langle e_{\alpha}, e_{-\beta}, e_{-\gamma}\rangle$ and $\phi(e_{-\gamma})\in \langle e_{\alpha}, e_{-\beta}, e_{-\gamma}\rangle$, where $\langle e_{\alpha}, e_{-\beta}, e_{-\gamma}\rangle$ is the kernel of the linear map $[e_{\alpha},-\ ]$. We have~$[e_{0}, e_{-\beta}]=2e_{-\beta}$ and $[e_{0}, e_{-\gamma}]= 2e_{-\gamma}$ so we get that $\phi(e_{-\beta})\in \langle e_{-\beta}, e_{-\gamma}\rangle$ and $\phi(e_{-\gamma})\in \langle e_{-\beta}, e_{-\gamma}\rangle$. Since $[e_{-\beta}, e_{-\gamma}]=-2e_{\alpha}$ we obtain that $\phi(e_{-\beta}\wedge e_{-\gamma})= \lambda e_{-\beta}\wedge e_{-\gamma}$, so we conclude that $H$ is the stabilizer of $L$.

Then we have proved that $\rG_2/H$ is isomorphic to the homogeneous space of Case (iv)
of Lemma 2.31 in \cite{P}. By Lemma \ref{lemma:lie_is_lie} and Proposition 2.34 in \cite{P} we get that $\Lie(2,V)$ is isomorphic to~$\OGr(2,V)$.
\end{proof}
\begin{corollary}
The Hilbert scheme of lines $F_1(\CG)$ on the Cayley Grassmannian $\CG$ is isomorphic to the projectivization~$\PP_{\OGr(2,V)}(s^*\tcE^{\vee})$.
\end{corollary}
\begin{proof}
This immediately follows from Propostion \ref{corollary:birationality} and Proposition \ref{lemma:Lie(2,7)}.
\end{proof}
\begin{remark}
In fact, the restristion of the $\PP^2$-bundle $\tE\simeq \PP_{\mathrm{LieGr}(2,V)}(\tcE)$ to $\Gad\subset \mathrm{LieGr}(2,V)$ is isomorphic to~$ \PP_{\Gad}(\cU^{\perp}_2/\cU_2)$ so that constructions from Proposition \ref{proposition:resolution_of_Hnu} and Lemma \ref{lemma:conic_bundle_Lie} are related to each other. 
\end{remark}

\section{Cohomology computations} \label{section:computations}
This appendix contains all necessary cohomology computations for proving exceptionality.
\subsection{The Borel--Bott--Weil theorem}
\label{section:bbw}
The Borel--Bott--Weil theorem computes the cohomology of line bundles on the flag variety of a reductive algebraic group. It can also be used to compute the cohomology of irreducible equivariant vector bundles on Grassmannians. 

Let $W$ be a vector space of dimension $n$.
We will use the standard identification of the weight lattice of the group $\mathrm{GL}(W)$ with $\mathbb{Z}^{n}$
that takes the fundamental weight of the representation $\Lambda^{k}W^{\vee}$ to the vector $(1, 1,\ldots, 1, 0, 0,\ldots , 0) \in \mathbb{Z}^{n}$
(the first $k$ entries are $1$, and the last $n-k$ are $0$).
We denote by
\begin{equation*}
\rho= (n, n- 1,\ldots, 2, 1)
\end{equation*}
the sum of the fundamental weights of $\mathrm{GL}(W)$.

The cone of dominant weights of $\mathrm{GL}(W)$ gets identified with the set of non-increasing sequences
$a=(a_{1}, a_{2},\ldots ,a_{n})$ of integers, i.e. we have $a_{1}\ge a_{2}\ge \ldots \ge a_{n-1}\ge a_{n} $.
We denote by $\Sigma^{a}W^{\vee}=\Sigma^{a_{1},a_{2},\ldots,a_{n}}W^{\vee}$
the representation of $\mathrm{GL}(W)$ of highest weight $a$.

Similarly, given a vector bundle $E$ of rank $n$ on a scheme $X$, we consider the corresponding principal $\mathrm{GL}(n)$-bundle on $X$ and denote by $\Sigma^{a}E$ the vector bundle associated with the $\mathrm{GL}(n)$-representation of highest weight $a$.

The Weyl group of $\mathrm{GL}(W)$ is isomorphic to the permutation group $\mathbf{S}_{n}$
and the length function $\ell \colon \mathbf{S}_n \to \mathbb{Z}$ counts the number of inversions in a permutation.
Note that for every weight $a \in \mathbb{Z}^{n}$
there exists a permutation $\sigma \in \mathbf{S}_{n}$ such that $\sigma(a)$ is dominant, i.e., non-increasing.

The linear algebraic group $\mathrm{GL}(W)$ acts naturally on the Grassmannian $\mathrm{Gr}(k,W)$ of $k$-dimensional subspaces in $W$.
Every irreducible $\mathrm{GL}(W)$-equivariant vector bundle on $\mathrm{Gr}(k,W)$ is isomorphic to
\begin{equation*}
\Sigma^{b}\cU_k^{\vee}\otimes \Sigma^{c}\cU_k^{\perp}
\end{equation*}
for some dominant weights $b \in \mathbb{Z}^{k}$ and $c \in \mathbb{Z}^{n-k}$.
\begin{theorem} [\cite{Kap}] \label{theorem:bbw}
Let $b \in \mathbb{Z}^{k}$ and $c \in \mathbb{Z}^{n-k}$ be non-increasing sequences.
Let $a =(b,c) \in \mathbb{Z}^{n}$ be their concatenation.
Assume that all entries of $a+\rho$ are distinct.
Let $\sigma \in \mathbf{S}_n$ be the unique permutation such that $\sigma(a+\rho)$ is strictly decreasing. Then
\begin{equation}
H^{k}(\mathrm{Gr}(k,W), \Sigma^{b}\cU_k^{\vee}\otimes \Sigma^{c}\cU_k^{\perp})=\begin{cases}
\Sigma^{\sigma(a+\rho)-\rho}\,W^{\vee},&\text{if $k=\ell(\sigma)$;}\\
0,&\text{otherwise.}
\end{cases}
\end{equation}
If not all entries of $a+\rho$ are distinct then
\begin{equation*}
H^{\bullet}(\mathrm{Gr}(k,W), \Sigma^{b}\cU_k^{\vee}\otimes \Sigma^{c}\cU_k^{\perp} )=0.
\end{equation*}
\end{theorem}
Also we will use the Littlewood–Richardson rule, that provides a recipe to decompose a tensor product $\Sigma^{a}\cU \otimes \Sigma^{b}\cU^{\vee}$
into a direct sum of bundles of the form $\Sigma^{c}\cU^{\vee}$.
We refer to \cite{5} for the precise formulation of this rule. Let us just mention that we have the following property of this decomposition.
\begin{lemma} [\cite{5}] \label{lemma:lrr}
Let $a=(a_1, \ldots,a_k)\in \mathbb{Z}^{k}$ and $b=(b_1, \ldots, b_k)\in \mathbb{Z}^{k}$ be non-increasing sequences.
Then there is a direct sum decomposition
\begin{equation*}
\Sigma^{a_1, \ldots,a_k }\cU_k \otimes \Sigma^{b_1, \ldots, b_k}\cU_k^{\vee}=
\bigoplus \Sigma^{c_1,\ldots, c_k}\cU_k^{\vee},
\end{equation*}
where for each summand in the right hand side we have
\begin{equation*}
- a_{k+1-i} \le c_i \le b_i\quad\text{for all $1 \le i \le k$}\qquad\text{and}\qquad
\displaystyle\sum c_i = \sum b_i - \sum a_i.
\end{equation*}
\end{lemma}
We will frequently use the following natural identifications:
\begin{equation*}
\Sigma^{a_1,\ldots,a_k}\cU_k \simeq \Sigma^{-a_k,\ldots ,-a_1}\cU_k^{\vee},
\qquad
\Sigma^{a_1+t,\ldots,a_k+t}\cU_k^{\vee} \simeq \Sigma^{a_1,\ldots, a_k}\cU_k^\vee \otimes \cO(t),
\end{equation*}
where $\cO(t)=\Sigma^{t, t,\ldots,t}\cU_k^{\vee}$ is  the Pl{\"u}cker line bundle of $\Gr(k,W)$.
In what follows for a dominant weight of the form $(a_1,a_2,0)$ we will omit the last zero,
i.e., we will write just~$(a_1,a_2)$~and~$\Sigma^{a_1,a_2}\cU^{\vee}$ for such a weight.
Note also that $\Sigma^{a,0,0}\cU^\vee \simeq S^a\cU^\vee$ and $\Sigma^{1,1}\cU^\vee = \Lambda^2\cU^\vee$.
\subsection{Cohomology}
Let $V$ be a 7-dimensional vector space. Recall that $\cU\subset V\otimes \cO$ and $\cU^{\perp}\subset V^{\vee}\otimes \cO$ are the tautological vector bundles on $\Gr(3,V)$, of ranks 3 and 4 respectively, $\cQ=(\cU^{\perp})^{\vee}$ and $\cO(1)$ is the Pl{\"u}cker line bundle. 

The table below include only those vector bundles that we use for the proof of exceptionality and fullness of the collection \eqref{collection}. In the table we abbreviate $\mathrm{H}^{\bullet}(\Sigma^{c_1,c_2,c_3})\coloneqq \mathrm{H}^{\bullet}(\CG,\Sigma^{c_1,c_2,c_3}\cU^{\vee})$.
Each column in the table lists of vector bundles $\Sigma^{c_1,c_2,c_3}\cU^{\vee}$ with fixed $c_1$, except for the last column where we add $\Sigma^{3,-1,-1}\cU^{\vee}$. Each column is ordered by the values of $c_3$ and $c_2$. Each row in the table lists vector bundles $\Sigma^{c_1,c_2,c_3}\cU^{\vee}$ with fixed $c_2$ and $c_3$. The shaded cells correspond to vector bundles for which the cohomology on $\CG$ and $\Gr(3,V)$ differ.
\begin{proposition} \label{main_computations}
On $\CG$ we have
 \begin{center}
 \begin{tabular}{|c|c|c|c|c|} 
 \hline
& & $\mathrm{H}^{\bullet}(\Sigma^{0,-3,-4})=0$ & & \\
 \hline
$\mathrm{H}^{\bullet}(\Sigma^{-2,-3,-3})=0$& $\mathrm{H}^{\bullet}(\Sigma^{-1,-3,-3})=0$& $\mathrm{H}^{\bullet}(\Sigma^{0,-3,-3})=0$ & & \\
 \hline
 $\mathrm{H}^{\bullet}(\Sigma^{-2,-2,-4})=0$ & $\mathrm{H}^{\bullet}(\Sigma^{-1,-2,-4})=0$ & & & \\
 \hline
$\mathrm{H}^{\bullet}(\Sigma^{-2,-2,-3})=0$& $\mathrm{H}^{\bullet}(\Sigma^{-1,-2,-3})=0$& $\mathrm{H}^{\bullet}(\Sigma^{0,-2,-3})=0$ &&\\
 \hline
& $\mathrm{H}^{\bullet}(\Sigma^{-1,-2,-2})= 0$ & $\mathrm{H}^{\bullet}(\Sigma^{0,-2,-2})= 0$ && \\
\hline
& $\mathrm{H}^{\bullet}(\Sigma^{-1,-1,-5})=\Bbbk[-4]$ & & & \\
 \hline 
  & $\mathrm{H}^{\bullet}(\Sigma^{-1,-1,-4})=0$ & $\mathrm{H}^{\bullet}(\Sigma^{0,-1,-4})=0$ & & \\
\hline 
  & $\mathrm{H}^{\bullet}(\Sigma^{-1,-1,-3})=0$ & $\mathrm{H}^{\bullet}(\Sigma^{0,-1,-3})=0$ & $\mathrm{H}^{\bullet}(\Sigma^{1,-1,-3})=0$ & \\
 \hline
& $\mathrm{H}^{\bullet}(\Sigma^{-1,-1,-2})= 0$ &$\mathrm{H}^{\bullet}(\Sigma^{0,-1,-2})= 0$ & $\mathrm{H}^{\bullet}(\Sigma^{1,-1,-2})= 0$ & $\mathrm{H}^{\bullet}(\Sigma^{2,-1,-2})= 0$\\
 \hline
& $\mathrm{H}^{\bullet}(\cO(-1))= 0$& $\mathrm{H}^{\bullet}(\Sigma^{0,-1,-1})= 0$ & $\mathrm{H}^{\bullet}(\Sigma^{1,-1,-1})= 0$ & $\mathrm{H}^{\bullet}(\Sigma^{2,-1,-1})= 0$\\
&&&&
$\mathrm{H}^{\bullet}(\Sigma^{3,-1,-1})=0$\\
 \hline
  \hline
& & \cellcolor{lightgray}$\mathrm{H}^{\bullet}(\Sigma^{0,0,-3})=\Bbbk[-2]$  & & \\
\hline 
\hline
&& $\mathrm{H}^{\bullet}(\Sigma^{0,0,-2})= 0$ & $\mathrm{H}^{\bullet}(\Sigma^{1,0,-2})= 0$ & $\mathrm{H}^{\bullet}(\Sigma^{2,0,-2})= 0$ \\
 \hline
&& $\mathrm{H}^{\bullet}(\Sigma^{0,0,-1})= 0$ & $\mathrm{H}^{\bullet}(\Sigma^{1,0,-1})=0$ & $\mathrm{H}^{\bullet}(\Sigma^{2,0,-1})= 0$  \\
  \hline
&& $\mathrm{H}^{\bullet}(\cO)=\Bbbk$ & $\mathrm{H}^{\bullet}(\cU^{\vee})= V^{\vee}$ & $\mathrm{H}^{\bullet}(S^2\cU^{\vee})= S^2V^{\vee} $\\
 \hline
&& & $\mathrm{H}^{\bullet}(\Sigma^{1,1,-2})=0$ & $\mathrm{H}^{\bullet}(\Sigma^{2,1,-2})=0$ \\
 \hline
 \hline
&& & \cellcolor{lightgray}$\mathrm{H}^{\bullet}(\Sigma^{1,1,-1})= \Bbbk$ & \cellcolor{lightgray}$\mathrm{H}^{\bullet}(\Sigma^{2,1,-1})= V^{\vee}$ \\
 \hline
 \hline
&&  & $\mathrm{H}^{\bullet}(\Lambda^2\cU^{\vee})= \Lambda^2V^{\vee}$ & $\mathrm{H}^{\bullet}(\Sigma^{2,1})= \Sigma^{2,1}V^{\vee}$ \\
 \hline
\end{tabular}
\end{center}
\end{proposition}
\begin{proof}
Recall that the Cayley Grassmannian $\varrho\colon \CG\hookrightarrow \Gr(3,V)$ is the zero locus of a regular section~$\lambda\in \mathrm{H}^0(\Gr(3,V),\cU^{\perp}(1))$ and that on $\Gr(3,V)$ we have the Koszul resolution \eqref{resolution}. For all $\Sigma^{c_1,c_2,c_3}$ appearing in the table we need to compute $\mathrm{H}^{\bullet}(\CG,\Sigma^{c_1,c_2,c_3}\cU^{\vee})$. Note that we have $c_1\in [-2,3]$, $c_2 \in [-3,1]$, $c_3\in [-5,0]$ and $c_1\ge c_2 \ge c_3$.

Tensoring the resolution \eqref{resolution} by $\Sigma^{c_1,c_2,c_3}\cU^{\vee}$ we obtain a resolution of $\Sigma^{c_1,c_2,c_3}\cU^{\vee}\otimes \varrho_*\cO_{\CG}\simeq \varrho_*(\Sigma^{c_1,c_2,c_3}\cU^{\vee})$ on $\Gr(3,V)$. Hence to compute $\mathrm{H}^{\bullet}(\CG,\Sigma^{c_1,c_2,c_3}\cU^{\vee})$ we need to compute $\mathrm{H}^{\bullet}(\Gr(3,V),\Sigma^{c_1,c_2,c_3}\cU^{\vee}\otimes \Lambda^l \cQ(-l))$ for $l\in[0,4]$ and then apply the spectral sequence
\begin{equation} \label{spectral_sequence}
    \mathbb{E}_{1}^{-l,t}=\mathrm{H}^t(\Gr(3,V),\Sigma^{c_1,c_2,c_3}\cU^{\vee}\otimes \Lambda^l \cQ(-l))\Rightarrow  \mathrm{H}^{t-l}(\CG,\Sigma^{c_1,c_2,c_3}\cU^{\vee}).
\end{equation} 

Let us denote the weight associated to $\Sigma^{c_1,c_2,c_3}\cU^{\vee}\otimes \Lambda^l \cQ(-l)$ by $$(x_1,x_2,x_3,x_4,x_5,x_6,x_7)=(c_1-l,c_2-l,c_3-l,0,\underbrace{-1,\ldots,-1}_{l})+\rho,$$ where~$\rho=(7,6,5,4,3,2,1)$. Below we use Theorem~\ref{theorem:bbw} to compute cohomology on $\Gr(3,V)$.
\begin{itemize}
\item
If $l=0$ then $$(x_1,x_2,x_3,x_4,x_5,x_6,x_7)=(c_1+7,c_2+6,c_3+5,4,3,2,1)$$ 
\begin{itemize}
    \item if $c_3\in [-4,-1]$ then $x_3\in [1,4]$ and all cohomology groups vanish;
    \item if $c_3=0$ then~$\mathrm{H}^{\bullet}(\Gr(3,V),\Sigma^{c_1,c_2}\cU^{\vee})= \Sigma^{c_1,c_2}V^{\vee}$;
     \item for $\Sigma^{-1,-1,-5}$ we have $(x_1,x_2,x_3)=(6,5,0)$, so we obtain $$\mathrm{H}^{\bullet}(\Gr(3,V),\Sigma^{-1,-1,-5}\cU^{\vee})=\Bbbk[-4].$$
\end{itemize} 
\item
If $l=1$ then $$(x_1,x_2,x_3,x_4,x_5,x_6,x_7)=(c_1+6,c_2+5,c_3+4,4,3,2,0)$$ 
\begin{itemize}
    \item if $c_3\in \{-4,-2,-1,0\}$ then $x_3\in \{ 0,2,3,4\}$ and all cohomology groups vanish;
    \item  if $c_3=-3$ and $c_2\in [-3,-1]$ then $x_2\in [2,4]$ and all cohomology groups vanish;
    \item  for $\Sigma^{0,0,-3}$ we have $(x_1,x_2,x_3)=(6,5,1)$, so we obtain~$\mathrm{H}^{\bullet}(\Gr(3,V),\Sigma^{0,0,-3}\cU^{\vee})=\Bbbk[-3]$.
\end{itemize}
\item
If $l=2$ then $$(x_1,x_2,x_3,x_4,x_5,x_6,x_7)=(c_1+5,c_2+4,c_3+3,4,3,1,0)$$
\begin{itemize}
    \item if $c_3\in \{-3,-2, 0\}$ then $x_3\in \{0,1,3\}$ so all cohomology groups vanish;
    \item if $c_3=-4$:
    \begin{itemize}
        \item if $c_2 \in \{-3,-1\}$ then $x_2 \in \{1, 3\}$ and all cohomology groups vanish;
        \item if $c_2=-2$ then $c_1 \in [-2,-1]$ so that $x_1\in [3,4]$ and all cohomology groups vanish;
    \end{itemize}
    \item if $c_3=-1$ then $c_2\in [-1,1]$:
    \begin{itemize}
        \item if $c_2\in [-1,0]$ then $x_2\in [3,4]$ and all cohomology groups vanish;
        \item for $c_2=1$ we have $\mathrm{H}^{\bullet}(\Gr(3,V),\Sigma^{c_1,1,-1}\cU^{\vee})=\Lambda^{c_1-1}V[-2]$. 
    \end{itemize} 
\end{itemize}
\item
If $l=3$ then $$(x_1,x_2,x_3,x_4,x_5,x_6,x_7)=(c_1+4,c_2+3,c_3+2,4,2,1,0)$$
\begin{itemize}
    \item if $c_3\in [-2,0]$ then $x_3\in [0,2]$ so all cohomology groups vanish;
    \item if $c_3=-3$:
    \begin{itemize}
        \item if $c_2\in [-3,-1]$ then $x_2\in [0,2]$ so all cohomology groups vanish;
        \item if $c_2=0$ then $c_1=0$ so that $x_1=4$ and all cohomology groups vanish;
    \end{itemize}
    \item if $c_3=-4$ then $c_2\in [-3,-1]$ so that $x_2\in [0,2]$ and all cohomology groups vanish.
    \end{itemize} 
\item
If $l=4$ then $$(x_1,x_2,x_3,x_4,x_5,x_6,x_7)=(c_1+4,c_2+3,c_3+2,4,3,2,1)$$
\begin{itemize}
    \item if $c_2\in [-2,1]$ then $x_3\in [1,4]$ so all cohomology groups vanish;
    \item if $c_2=-3$ then $c_1\in [-2,0]$ so that $x_1\in [2,4]$ and all cohomology groups vanish.
\end{itemize} 
\end{itemize}
Using the spectral sequence \eqref{spectral_sequence} we obtain the statement of the lemma from the above computations.  
\end{proof}
\subsection{Exceptionality} Now we use these computations to prove exceptionality. 
We consider the partial ordering on dominant weights of $\mathrm{GL}_3$ defined by:
\begin{align*}
(c_1,c_2,c_3) &\preceq (b_1,b_2,b_3)  &&
\text{if $c_1 \le b_1$ and $c_2 \le b_2$ and $c_3 \le b_3$},
\\
(c_1,c_2,c_3) &\prec (b_1,b_2,b_3)  &&
\text{if $c \le b$ and $c \ne b$}.
\end{align*}

Recall notation \eqref{mathfrak_E}.
\begin{corollary} \label{exceptionality1}
The collection $\mathfrak{E},\mathfrak{E}(1),\mathfrak{E}(2),\mathfrak{E}(3)$ is exceptional.
\end{corollary}
\begin{proof}
We need to check that
$$\mathrm{Ext}^{\bullet}(\Sigma^{a}\cU^{\vee}(k),\Sigma^{b}\cU^{\vee})=
\mathrm{H}^{\bullet}(\CG,\Sigma^{a}\cU\otimes \Sigma^{b}\cU^{\vee}(-k))=0,$$
where $(0,0) \preceq b, a \preceq(1,1)$ and $0\le k\le 3$.
Using the Littlewood–Richardson rule we decompose
\begin{equation*}
\Sigma^{a}\cU\otimes \Sigma^{b}\cU^{\vee}(-k)= \bigoplus \Sigma^{c_1,c_2,c_3}\cU^{\vee}(-k),
\end{equation*}
where by Lemma~\ref{lemma:lrr} we have $c_1\in [0,1]$, $c_2 \in [-1,1]$, $c_3\in [-1,0]$. Hence we need to compute~$\mathrm{H}^{\bullet}(\CG,\Sigma^{c_1,c_2,c_3}\cU^{\vee}(-k))\simeq \mathrm{H}^{\bullet}(\CG,\Sigma^{c_1-k,c_2-k,c_3-k}\cU^{\vee})$. 

For $k= 1,2$ the vanishing~$\mathrm{H}^{\bullet}(\CG,\Sigma^{c_1-k,c_2-k,c_3-k}\cU^{\vee})=0$ follows from Proposition \ref{main_computations}. The case~$k=3$ follows from Serre duality.

For $k=0$ we additionally have $c_1+c_2+c_3\le 0$ so $\mathrm{H}^{\bullet}(\CG,\Sigma^{c_1,c_2,c_3}\cU^{\vee})=0$ again by Proposition~\ref{main_computations}. 
\end{proof}
\begin{corollary} \label{exceptionality2}
The collection $\mathfrak{E},\Sigma^{2,1}\cU^{\vee},\mathfrak{E}(1),\mathfrak{E}(2),\mathfrak{E}(3)$ is exceptional.
\end{corollary}
\begin{proof}
By Corollary \ref{exceptionality1} to prove the statement it is enough to compute for $1\le k\le 3$
\begin{align} \label{exceptionality21}
& \mathrm{Ext}^{\bullet}(\Sigma^{2,1}\cU^{\vee},\Sigma^{2,1}\cU^{\vee})= \Bbbk, \\ \label{21E}
& \mathrm{Ext}^{\bullet}(\Sigma^{2,1}\cU^{\vee},\mathfrak{E})= 0, \\ \label{E21}
& \mathrm{Ext}^{\bullet}(\mathfrak{E}(k),\Sigma^{2,1}\cU^{\vee})=0.
\end{align}
Let us first prove \eqref{exceptionality21}. Using the Littlewood–Richardson rule we decompose 
\begin{equation} \label{computation:21}
\Sigma^{2,1}\cU\otimes \Sigma^{2,1}\cU^{\vee}= \cO \oplus (\Sigma^{1,0,-1}\cU^{\vee})^{\oplus 2} \oplus \Sigma^{2,-1,-1}\cU^{\vee} \oplus \Sigma^{2,0,-2}\cU^{\vee}\oplus \Sigma^{1,1,-2}\cU^{\vee},
\end{equation}
so the statement follows from Proposition \ref{main_computations}.

Let us now prove \eqref{E21} for $k=1$. We need to compute $$\mathrm{Ext}^{\bullet}(\Sigma^{a}\cU^{\vee}(1),\Sigma^{2,1}\cU^{\vee})=\mathrm{H}^{\bullet}(\CG,\Sigma^{a}\cU\otimes \Sigma^{2,1}\cU^{\vee}(-1))$$
for $(0,0) \preceq a \preceq(1,1)$. Using the Littlewood–Richardson rule we decompose
\begin{equation*}
\Sigma^{a}\cU\otimes \Sigma^{2,1}\cU^{\vee}(-1)= \bigoplus \Sigma^{c_1,c_2,c_3}\cU^{\vee}(-1),
\end{equation*}
where by Lemma~\ref{lemma:lrr} we have $1\le c_1+c_2+c_3\le 3$ and $c_1\in [1,2],$ $c_2 \in [0,1],$ $c_3\in [-1,0]$. So by Proposition~\ref{main_computations} all cohomology groups vanish in these cases.

Let us prove \eqref{21E}. Denote by $\mathfrak{E}^{\vee}= (\Lambda^2\cU,\cU,\cO)$. Using \eqref{preliminaries_isomorphisms} we obtain $\mathfrak{E}^{\vee}(2)\in (\mathfrak{E}(1), \cO(2))$. Thus, using the isomorphism 
$$\mathrm{Ext}^{\bullet}(\Sigma^{2,1}\cU^{\vee},\cE)\simeq \mathrm{Ext}^{\bullet}(\cE^{\vee}, \Sigma^{2,1}\cU^{\vee}(-2))=\mathrm{Ext}^{\bullet}(\cE^{\vee}(2), \Sigma^{2,1}\cU^{\vee}),$$
where $\cE\in \mathfrak{E}$, we reduce the vanishing \eqref{21E} to the vanishing~$\mathrm{Ext}^{\bullet}(\cO(2), \Sigma^{2,1}\cU^{\vee})= 0$ and the vanishing~\eqref{E21} for $k=1$. The first follows from Proposition~\ref{main_computations} and the second one was already shown.

It remains to prove \eqref{E21} for $k=2,3$.
By Proposition \ref{proposition:21} and Corollary \ref{exceptionality1} to prove~$\mathrm{Ext}^{\bullet}(\mathfrak{E}(k),\Sigma^{2,1}\cU^{\vee})=0$ it is enough to show the vanishing $\mathrm{Ext}^{\bullet}(\mathfrak{E}(k),\Sigma^{2,1}\cU^{\vee}(-1))=0$. By Serre duality we have
$$\mathrm{Ext}^{\bullet}(\mathfrak{E}(k),\Sigma^{2,1}\cU^{\vee}(-1))= \mathrm{Ext}^{\bullet}(\Sigma^{2,1}\cU^{\vee}(3-k), \mathfrak{E}[8]).$$
Thus, the case $k=3$ immediately follows from the vanishing \eqref{21E} that was already shown. When $k=2$ we again apply Proposition \ref{proposition:21} and Corollary \ref{exceptionality1} and reduce the vanishing $\mathrm{Ext}^{\bullet}(\Sigma^{2,1}\cU^{\vee}(1), \mathfrak{E})=0$ to the vanishing~\eqref{21E} that was already shown. So we get the statement. 
\end{proof}
Now we consider $\Lambda^2\cQ$. We will see that it does not fit to the exceptional collection, but its right mutation $\mathbb{R}_{\Lambda^2\cU^{\vee}}\Lambda^2\cQ$ does, so the below computations will be useful.
\begin{lemma} \label{Q}
The vector bundle $\Lambda^2\cQ$ is exceptional.
\end{lemma}
\begin{proof}
We need to compute
$\mathrm{Ext}^{\bullet}(\Lambda^2\cQ, \Lambda^2\cQ)=
\mathrm{H}^{\bullet}(\CG,\Lambda^2\cU^{\perp}\otimes \Lambda^2\cQ)$.
For this we compute $\mathrm{Ext}$ from objects of the exact sequence \eqref{Koszul_dual} for $n=1$ to objects of the exact sequence \eqref{Koszul} for $n=0$.
Using Proposition \ref{main_computations} we obtain that $\mathrm{Ext}^{\bullet}(\Lambda^2\cQ, \Lambda^2\cQ)=
\mathrm{H}^{\bullet}(\CG,\Lambda^2\cU^{\perp}\otimes \Lambda^2\cQ)=\Bbbk$, so we deduce the statement. 
\end{proof}
\begin{lemma} \label{lemma_right_ort}
We have $\Lambda^2\cQ \in \langle \Sigma^{2,1}\cU^{\vee},\mathfrak{E}(1),\mathfrak{E}(2),\mathfrak{E}(3) \rangle^{\perp}$.
\end{lemma}
\begin{proof}
We need to show that 
\begin{equation*}
    \mathrm{Ext}^{\bullet}(\Sigma^{2,1}\cU^{\vee},\Lambda^2\cQ)= \mathrm{H}^{\bullet}(\CG, \Sigma^{2,1}\cU\otimes \Lambda^2\cQ)=0 \quad \mbox{and} \quad  \mathrm{Ext}^{\bullet}(\Sigma^{a}\cU^{\vee}(k),\Lambda^2\cQ)=\mathrm{H}^{\bullet}(\CG,\Sigma^{a}\cU(-k)\otimes \Lambda^2\cQ)=0,
\end{equation*}
where $(0,0) \preceq \alpha \preceq(1,1)$ and $1 \le k \le 3$.

First, let us prove that $\mathrm{H}^{\bullet}(\CG,\Sigma^{2,1}\cU\otimes \Lambda^2\cQ)=0$. Twisting \eqref{Koszul} for $n=0$ by~$\Sigma^{2,1}\cU$ we get:
\begin{equation} \label{exact_sequence_lambda2Q1}
0\to S^2\cU\otimes \Sigma^{2,1}\cU \to V\otimes \cU \otimes \Sigma^{2,1}\cU \to \Lambda^2 V\otimes \cO \otimes \Sigma^{2,1}\cU \to \Lambda^2\cQ\otimes \Sigma^{2,1}\cU \to 0.
\end{equation}
The spectral sequence for the exact sequence \eqref{exact_sequence_lambda2Q1}:
\begin{equation*} 
    \mathbb{E}_{1}^{-l,t}=\mathrm{H}^t(\CG,\Sigma^{2,1}\cU\otimes S^l\cU \otimes \Lambda^{2-l}V)\Rightarrow \mathrm{H}^{t-l}(\CG, \Sigma^{2,1}\cU\otimes \Lambda^2 \cQ)
\end{equation*} 
shows that it is enough to prove vanishing $\mathrm{H}^{\bullet}(\CG,\Sigma^{2,1}\cU\otimes S^l\cU)=0$ for $l\in [0,2]$:
\begin{itemize}
    \item for $l=0$ we have 
\begin{equation*} \mathrm{H}^{\bullet}(\CG,\Sigma^{2,1}\cU\otimes\cO)=\mathrm{Ext}^{\bullet}(\Sigma^{2,1}\cU^{\vee},\cO)=0,
\end{equation*}
where the vanishing follows from Corrollary \ref{exceptionality2};
\item for $l=1$ we have 
\begin{equation*}
  \mathrm{Ext}^{\bullet}(\Sigma^{2,1}\cU^{\vee},\cU)=\mathrm{Ext}^{\bullet}(\Sigma^{2,1}\cU^{\vee}(1),\Lambda^2\cU^{\vee})=\mathrm{Ext}^{\bullet}(\Lambda^2\cU^{\vee}(3),\Sigma^{2,1}\cU^{\vee})[8]=0,
\end{equation*}
where the first equality follows the isomorphism \eqref{preliminaries_isomorphisms}, the second equality is Serre duality and the last one follows from Corollary \ref{exceptionality2}, so we get that $$\mathrm{H}^{\bullet}(\CG,\Sigma^{2,1}\cU\otimes\cU)=  \mathrm{Ext}^{\bullet}(\Sigma^{2,1}\cU^{\vee},\cU)=0;$$
\item for $l=2$ we have
\begin{equation*}
\mathrm{Ext}^{\bullet}(\Sigma^{2,1}\cU^{\vee},S^2\cU)= \mathrm{Ext}^{\bullet}(S^2\cU(4),\Sigma^{2,1}\cU^{\vee}[8])=\mathrm{H}^{\bullet}(\CG,\Sigma^{2,1}\cU^{\vee}\otimes S^2\cU^{\vee}(-4))[8],
\end{equation*}
where the first equality follows from Serre duality;
using the Littlewood–Richardson rule we decompose 
\begin{equation*}
\begin{split}
  \mathrm{H}^{\bullet}(\CG,\Sigma^{2,1}\cU^{\vee}\otimes & S^2\cU^{\vee} (-4))= \\
  & =\mathrm{H}^{\bullet}(\CG,\Sigma^{0,-3,-4}\cU^{\vee} \oplus \Sigma^{-1,-2,-4}\cU^{\vee} \oplus \Sigma^{-1,-3,-3}\cU^{\vee}\oplus \Sigma^{-2,-2,-3}\cU^{\vee}),   
\end{split}
\end{equation*}
so the vanishing
$$\mathrm{H}^{\bullet}(\CG,\Sigma^{2,1}\cU\otimes S^2\cU)= \mathrm{Ext}^{\bullet}(\Sigma^{2,1}\cU^{\vee},S^2\cU)= \mathrm{H}^{\bullet}(\CG,\Sigma^{2,1}\cU^{\vee}\otimes S^2\cU^{\vee}(-4)[-8])=0$$ follows from Proposition~\ref{main_computations}.
\end{itemize}
Hence we get that
$$\mathrm{Ext}^{\bullet}(\Sigma^{2,1}\cU^{\vee},\Lambda^2\cQ)= \mathrm{H}^{\bullet}(\CG,\Sigma^{2,1}\cU\otimes \Lambda^2\cQ)=0.$$

Now let us prove that $\mathrm{H}^{\bullet}(\Sigma^a\cU(-k)\otimes \Lambda^2\cQ)=0$ for $(0,0) \preceq a \preceq(1,1)$ and $1 \le k \le 3$. Twisting~\eqref{Koszul} for $n=0$ by $\Sigma^{a}\cU(-k)$ we get
\begin{equation} \label{exact_sequence_lambda2Q2}
0\to S^2\cU\otimes \Sigma^{a}\cU(-k) \to V\otimes \cU \otimes \Sigma^{a}\cU(-k) \to \Lambda^2 V\otimes \cO \otimes \Sigma^{a}\cU(-k) \to \Lambda^2\cQ \otimes \Sigma^{a}\cU(-k) \to 0.
\end{equation}
Using the spectral sequence for the exact sequence \eqref{exact_sequence_lambda2Q2}:
\begin{equation} \label{spectral_sequence_Lambda2Q}
    \mathbb{E}_{1}^{-l,t}=\mathrm{H}^t(\CG,\Sigma^{a}\cU(-k)\otimes S^l\cU\otimes \Lambda^{2-l}V)\Rightarrow \mathrm{H}^{t-l}(\CG, \Sigma^{a}\cU(-k) \otimes \Lambda^2 \cQ)
\end{equation} 
we see that we need to compute  $\mathrm{H}^{\bullet}(\CG, \Sigma^{a}\cU(-k)\otimes S^l\cU)=0$ for $l\in [0,2]$:
\begin{itemize}
    \item for $l=0$ we have 
\begin{equation*} \mathrm{H}^{\bullet}(\CG,\Sigma^{a}\cU(-k)\otimes\cO)= \mathrm{Ext}^{\bullet}(\Sigma^{a}\cU^{\vee}(k),\cO)=0,
\end{equation*}
where the vanishing follows from Corrollary \ref{exceptionality2};
\item for $l=1$ using the isomorphism $\cU\simeq \Lambda^2\cU^{\vee}(-1)$ and Corrollary \ref{exceptionality2} we get 
\begin{equation*}
\mathrm{H}^{\bullet}(\CG,\Sigma^{a}\cU(-k')\otimes\cU)=\mathrm{Ext}^{\bullet}(\Sigma^{a}\cU^{\vee}(k'),\cU)=0
\end{equation*}
for $k=1,2$. For $k=3$ by Serre duality we have
\begin{equation*}
 \mathrm{Ext}^{\bullet}(\Sigma^{a}\cU^{\vee}(3),\cU)= \mathrm{Ext}^{\bullet}(\Lambda^2\cU^{\vee},\Sigma^{a}\cU^{\vee}[8]),
\end{equation*}
so that by Corollary \ref{exceptionality1} we get
\begin{equation*}
\mathrm{H}^{\bullet}(\CG,\Sigma^{a}\cU(-3)\otimes\cU)=\mathrm{Ext}^{\bullet}(\Sigma^{a}\cU^{\vee}(3),\cU)=
\begin{cases}
0, &\text{if $(0,0) \preceq a \preceq(1,0)$};\\
\Bbbk[-8],&\text{if $a=(1,1,0)$};
\end{cases}
\end{equation*}
\item for $l=2$ we have
\begin{equation*}
     \mathrm{Ext}^{\bullet}(\Sigma^{a}\cU^{\vee}(k),S^2\cU)= \mathrm{Ext}^{\bullet}(S^2\cU(4-k),\Sigma^{a}\cU^{\vee}[8])=\mathrm{H}^{\bullet}(\CG,\Sigma^{a}\cU^{\vee}\otimes S^2\cU^{\vee}(k-4))[-8],
\end{equation*}
where the first equality follows from Serre duality; using the Littlewood–Richardson rule we decompose
\begin{equation*}
\Sigma^{a}\cU^{\vee}\otimes S^2\cU^{\vee}(k-4)= \bigoplus \Sigma^{c_1,c_2,c_3}\cU^{\vee}(k-4),
\end{equation*}
where by Lemma~\ref{lemma:lrr} we have $2\le c_1+c_2+c_3\le 4$ and $c_1\in [2,3],$ $c_2 \in [0,1],$ $c_3\in [0,1]$; by Proposition \ref{main_computations} we get 
\begin{equation*}
\mathrm{H}^{\bullet}(\CG,\Sigma^{a}\cU(-k)\otimes S^2\cU)=\mathrm{Ext}^{\bullet}(\Sigma^{a}\cU^{\vee}(k),S^2\cU)=
\begin{cases}
V^{\vee}[-8],&\text{if $a=(1,1)$ and $k=3$};\\
0, &\text{in all other cases}.\\
\end{cases}
\end{equation*}
\end{itemize}
Using the spectral sequence \eqref{spectral_sequence_Lambda2Q} we obtain the vanishing 
$$\mathrm{Ext}^{\bullet}(\Sigma^{a}\cU^{\vee}(k),\Lambda^2\cQ)=\mathrm{H}^{\bullet}(\CG, \Sigma^{a}\cU(-k)\otimes \Lambda^2 \cQ)=0$$
from the above computations.
\end{proof}

\begin{lemma} \label{lemma:Lambda^2Q}
For $(0,0) \preceq a \preceq(1,1)$ and $k=0,1$ we have
\begin{equation*}
\mathrm{Ext}^{\bullet}(\Lambda^2\cQ(k),\Sigma^{a}\cU^{\vee})=\begin{cases}
\Bbbk, &\text{if $a=(1,1)$ and $k=0$};\\
0, &\text{in all other cases}.
\end{cases}
\end{equation*}
Moreover, we have $\mathrm{Ext}^{\bullet}(\Lambda^2\cQ(1),\Sigma^{2,1}\cU^{\vee})=0$.
\end{lemma}
\begin{proof}
Using Lemma \ref{lemma_right_ort} and the isomorphism \eqref{preliminaries_isomorphisms2} we obtain 
\begin{equation*}
  \mathrm{Ext}^{\bullet}(\Lambda^2\cQ(1),\Sigma^{2,1}\cU^{\vee})\simeq \mathrm{Ext}^{\bullet}(\Sigma^{2,1}\cU, \Lambda^2\cU^{\perp}(-1))\simeq  \mathrm{Ext}^{\bullet}(\Sigma^{2,1}\cU^{\vee}(-2), \Lambda^2\cQ(-2))=0.  
\end{equation*}
Now we need to compute
\begin{equation*}
   \mathrm{Ext}^{\bullet}(\Lambda^2\cQ(k),\Sigma^{a}\cU^{\vee})=\mathrm{H}^{\bullet}(\CG,\Sigma^{a}\cU^{\vee}(-k)\otimes \Lambda^2\cU^{\perp})
\end{equation*}
for $(0,0) \preceq a \preceq(1,1)$ and $k=0,1$.
To do this let us consider the exact sequence \eqref{Koszul} for $n=-1$
\begin{equation} \label{exact_sequence_Q_dual}
    0\to S^2\cU(-1) \to V\otimes \cU(-1) \to \Lambda^2 V\otimes \cO(-1) \to \Lambda^2\cU^{\perp} \to 0,
\end{equation}
here we use the isomorphism \eqref{preliminaries_isomorphisms2}.
Tensoring the exact sequence \eqref{exact_sequence_Q_dual} by $\Sigma^{a}\cU^{\vee}(-k)$ we get the following spectral sequence:
\begin{equation} \label{spectral_sequence_dual}
    \mathbb{E}_{1}^{-l,t}=\mathrm{H}^t(\CG,\Sigma^{a}\cU^{\vee}(-k)\otimes S^l\cU \otimes \Lambda^{2-l}V)\Rightarrow \mathrm{H}^{t-l}(\CG, \Sigma^{a}\cU^{\vee}(-k)\otimes \Lambda^2 \cU^{\perp}),
\end{equation} 
which shows that we need to compute $\mathrm{H}^{\bullet}(\CG,\Sigma^{a}\cU^{\vee}(-k)\otimes S^l\cU(-1))$ for $l\in [0,2]$:
\begin{itemize}
    \item for $l=0$ we have 
    \begin{equation*}
        \mathrm{H}^{\bullet}(\CG,\Sigma^{a}\cU^{\vee}\otimes \cO(-1-k))=\mathrm{Ext}^{\bullet}(\cO(k+1),\Sigma^{a}\cU^{\vee})=0
    \end{equation*}
    by Corollary \ref{exceptionality1};
    \item for $l=1$ we have 
    \begin{equation*}
        \mathrm{H}^{\bullet}(\CG,\Sigma^{a}\cU^{\vee}\otimes \cU(-1-k))=\mathrm{Ext}^{\bullet}(\cU^{\vee}(k+1),\Sigma^{a}\cU^{\vee})=0
    \end{equation*}
    by Corollary \ref{exceptionality1};
    \item for $l=2$ we have 
        \begin{equation*}
        \mathrm{H}^{\bullet}(\CG,\Sigma^{a}\cU^{\vee}\otimes S^2\cU(-1-k))= \mathrm{H}^{\bullet}(\CG,\Sigma^{c_1,c_2, c_3}\cU^{\vee}(-k)),
         \end{equation*}
    where $c_1\in [-1,0]$, $c_2\in [-1,0]$, $c_3\in [-3,-2]$ and $-5\le c_1+c_2+ c_3\le -3$ by Lemma \ref{lemma:lrr}; so by Proposition \ref{main_computations} we get
    \begin{equation*}
\mathrm{H}^{\bullet}(\CG,\Sigma^{a}\cU^{\vee}\otimes S^2\cU(-1-k))=
\begin{cases}
\Bbbk[-2],&\text{if $a=(1,1)$ and $k=0$};\\
0, &\text{in all other cases}.\\
\end{cases}
\end{equation*}
\end{itemize}
Using the spectral sequence \eqref{spectral_sequence_dual} we obtain the required statement.
\end{proof}
\begin{corollary} \label{corollary:Lambda^2Q}
We have $\mathrm{Ext}^{\bullet}(\Lambda^2\cQ(1), \Lambda^2\cQ)=0$.
\end{corollary}
\begin{proof}
We have $\mathrm{Ext}^{\bullet}(\Lambda^2\cQ(1), \Lambda^2\cQ)=\mathrm{Ext}^{\bullet}(\Lambda^2\cQ,\Lambda^2\cQ(-1))$ and to compute $\mathrm{Ext}^{\bullet}(\Lambda^2\cQ, \Lambda^2\cQ(-1))$ we use the exact sequence \eqref{Koszul_dual} for $n=0$. Applying Lemma \ref{lemma:Lambda^2Q} we see that it is enough to show  
$$\mathrm{Ext}^{\bullet}(\Lambda^2\cQ, S^2\cU^{\vee})=0.$$
Using the exact sequence \eqref{Koszul_dual} for $n=1$ and Proposition \ref{main_computations} we deduce the statement.
\end{proof}
Recall from the Introduction that we denote by $\cR$ the dual quotient bundle $(\Lambda^2\cU^{\perp}/ \Lambda^2\cU)^{\vee}$.
\begin{corollary} \label{mutation_R_Lambda2}
On $\CG$ we have $\mathbb{R}_{\Lambda^2\cU^{\vee}}\Lambda^2\cQ=\cR$ and conversely $\mathbb{L}_{\Lambda^2\cU^{\vee}}\cR=\Lambda^2\cQ$. The vector bundle~$\cR$ is exceptional. 
\end{corollary}
\begin{proof}
By Lemma \ref{lemma:Lambda^2Q} we have $\mathrm{Ext}^{\bullet}(\Lambda^2\cQ,\Lambda^2\cU^{\vee})=\Bbbk$ so we see that the mutation $\mathbb{R}_{\Lambda^2\cU^{\vee}}\Lambda^2\cQ$ is given by the exact sequence \eqref{exact_sequence:R_Lambda^2Q}, and similarly for $\mathbb{L}_{\Lambda^2\cU^{\vee}}\cR$. The last statement follows from Proposition \ref{proposition:mutations} and Lemma \ref{Q}.
\end{proof}
The main result of this section is the following corollary.
\begin{corollary} \label{corollary_exceptionality}
The collection $\mathfrak{E},\cR,\Sigma^{2,1}\cU^{\vee},\mathfrak{E}(1),\cR(1),\mathfrak{E}(2),\mathfrak{E}(3)$ is exceptional.
\end{corollary}
\begin{proof}
By Corollary \ref{mutation_R_Lambda2} the vector bundle $\cR$ is exceptional, thus, by Corollary \ref{exceptionality2} we need to show that
\begin{align} \label{R1}
& \cR\in\  ^{\perp}\langle \mathfrak{E}\rangle,\\ \label{R2}
& \cR\in \langle \Sigma^{2,1}\cU^{\vee},\mathfrak{E}(1),\mathfrak{E}(2),\mathfrak{E}(3) \rangle^{\perp},\\ \label{R3}
& \cR(1)\in\  ^{\perp}\langle \mathfrak{E}, \cR, \Sigma^{2,1}\cU^{\vee} \rangle. 
\end{align}
Applying Corollary \ref{mutation_R_Lambda2}, Corollary \ref{exceptionality2} and Lemma \ref{lemma_right_ort} we obtain \eqref{R2}. Applying Corollary \ref{mutation_R_Lambda2}, Corollary \ref{exceptionality2} and Lemma \ref{lemma:Lambda^2Q} we get $\cR(1)\in\  ^{\perp}\langle \mathfrak{E}, \Sigma^{2,1}\cU^{\vee} \rangle$ and \eqref{R1}.

So the last thing we need to check is the vanishing $\mathrm{Ext}^{\bullet}(\cR(1), \cR)=0$. By Corollary \ref{mutation_R_Lambda2}, Lemma~\ref{lemma_right_ort} and Lemma \ref{lemma:Lambda^2Q} it is enough to show  $\mathrm{Ext}^{\bullet}(\Lambda^2\cQ(1), \Lambda^2\cQ)=0$ and this is the statement of Corollary~\ref{corollary:Lambda^2Q}.
\end{proof}

\subsection{Residual category}
In this subsection we compute the residual category for the Lefschetz collection 
\eqref{collection}. Recall Definition \ref{def:lefschetz}(iii). To describe the residual category for the the Lefschetz collection~\eqref{collection} we first compute several left mutations.
\begin{lemma} \label{lemma:mutation21}
On $\CG$ we have $\mathbb{L}_{\mathfrak{E}}\Sigma^{2,1}\cU^{\vee}\simeq\Sigma^{2,1}\cU^{\vee}(-1)[2]$.
\end{lemma}
\begin{proof}
Using Proposition \ref{main_computations} and decomposition \eqref{computation:21} we obtain $$\mathrm{Ext}^{\bullet}(\Sigma^{2,1}\cU^{\vee}, \Sigma^{2,1}\cU^{\vee}(-1))=\mathrm{H}^{\bullet}(\Sigma^{2,1}\cU\otimes \Sigma^{2,1}\cU^{\vee}(-1))= \Bbbk[-2],$$
so that we have a morphism $\Sigma^{2,1}\cU^{\vee}\to \Sigma^{2,1}\cU^{\vee}(-1)[2]$. To prove the statement it is enough to check that 
\begin{enumerate}
    \item $\mathrm{Cone}(\Sigma^{2,1}\cU^{\vee}\to \Sigma^{2,1}\cU^{\vee}(-1)[2])\in \mathfrak{E}$;
    \item $\Sigma^{2,1}\cU^{\vee}(-1)\in \mathfrak{E}^{\perp}$.
\end{enumerate}
By Corollary \ref{corollary_exceptionality} we get (ii). By Proposition \ref{proposition:21} we have $\Sigma^{2,1}\cU^{\vee}\in \langle \Sigma^{2,1}\cU^{\vee}(-1),\mathfrak{E},\cO(1)\rangle$. By Corollary~\ref{corollary_exceptionality} we have $\mathrm{Ext}^{\bullet}(\cO(1), \Sigma^{2,1}\cU^{\vee})=0$ so we deduce that $\Sigma^{2,1}\cU^{\vee}\in \langle \Sigma^{2,1}\cU^{\vee}(-1),\mathfrak{E}\rangle$ and (i) follows.
\end{proof}
\begin{remark}
In fact, the mutation  $\mathbb{L}_{\mathfrak{E}}\Sigma^{2,1}\cU^{\vee}$ is given by the following self-dual exact sequence on~$\CG$
\begin{equation*}
    0\to \Sigma^{2,1}\cU^{\vee}(-1)\to (V\oplus \Bbbk)\otimes \cU^{\vee}\to (V^{\vee}\oplus \Bbbk)\otimes\Lambda^2\cU^{\vee}\to \Sigma^{2,1}\cU^{\vee}\to 0.
\end{equation*}
\end{remark}
\begin{lemma} \label{mutation:R}
We have $\cR(-1)\in \langle \mathfrak{E}, \mathfrak{E}(1), \cR(1)\rangle$.
\end{lemma}
\begin{proof}
Using the exact sequence \eqref{exact_sequence:R_Lambda^2Q} twisted by $\cO(-1)$ and the exact sequence \eqref{Koszul_dual} for $n=0$ we get
\begin{equation} \label{Rin}
  \cR(-1)\in  \langle \cU, \mathfrak{E}, S^2\cU^{\vee} \rangle.  
\end{equation}
Using the exact sequences \eqref{exact_sequence:E_dual} and \eqref{tautological_exact_sequence} we obtain $S^2\cU^{\vee}\in \langle \cU, \cO, \cE^{\vee}\rangle$. By Corollary \ref{proposition_E_E} and the exact sequence \eqref{tautological_exact_sequence_1} twisted by $\cO(1)$ we get $S^2\cU^{\vee}\in \langle \cU, \cO, S^2\cU(1),  \cU^{\vee}(1)\rangle$. Using the exact sequence \eqref{Koszul} for $n=1$ and the exact sequence \eqref{exact_sequence:R_Lambda^2Q} twisted by $\cO(-1)$ we obtain $S^2\cU^{\vee}\in \langle \cU, \mathfrak{E}, \mathfrak{E}(1), \cR(1) \rangle$ so from~\eqref{Rin} we get $\cR(-1)\in \langle \cU, \mathfrak{E}, \mathfrak{E}(1), \cR(1)\rangle$.

By Corollary~\ref{corollary_exceptionality} we have $\mathrm{Ext}^{\bullet}(\cR(-1),\Lambda^2\cU^{\vee}(-1))=0$, so using the isomorphism \eqref{preliminaries_isomorphisms} we deduce the statement. 
\end{proof}
\begin{lemma} \label{lemma:mutationR(1)}
Up to shift we have $\mathbb{L}_{\mathfrak{E},\mathfrak{E}(1)}(\cR(1))\simeq \cR(-1).$
\end{lemma}
\begin{proof}
By Corollary~\ref{corollary_exceptionality} we have $\cR(-1)\in \langle \mathfrak{E},\mathfrak{E}(1) \rangle^{\perp}$. Thus, from Lemma \ref{mutation:R} we deduce that there exits a morphism $\mathrm{Ext}^{\bullet}(\cR(1),\cR(-1))\otimes \cR(1) \to \cR(-1)$
and that for its cone we have $$\mathrm{Cone}(\mathrm{Ext}^{\bullet}(\cR(1),\cR(-1))\otimes \cR(1) \to \cR(-1))\in \langle \mathfrak{E},\mathfrak{E}(1)\rangle,$$
so we get the statement.
\end{proof}
Recall, see \cite[Theorem 2.8]{KS}, that the residual category $\mathsf{Res}$ for the Lefschetz decomposition \eqref{collection} admits the following autoequivalence:
\begin{equation*}
    \tau(-) \coloneqq \mathbb{L}_{\mathfrak{E}}(-\otimes \cO(1)).
\end{equation*}
\begin{theorem} \label{residual}
The residual category $\mathsf{Res}$ for the Lefschetz decomposition \eqref{collection} is generated by three completely orthogonal exceptional objects with the following action of $\tau$:
\begin{equation*}
\xymatrix@C=0.5pc{
     \mathsf{Res} \ar@{=}[r]&\langle \mathbb{L}_{\mathfrak{E}}\cR \ar@/_2.2pc/[rr]_{\tau},& \Sigma^{2,1}\cU^{\vee}(-1) \ar@(dl,dr)[]^{\tau},& \cR(-1) \ar@/_2pc/[ll]_{ \tau}\rangle.
     }
\end{equation*}
\end{theorem}
\begin{proof}
By definition, $\mathsf{Res}$ is the orthogonal to the collection $\langle \mathfrak{E}, \mathfrak{E}(1),\mathfrak{E}(2),\mathfrak{E}(3)\rangle$ in $\Db(\CG)$.
Therefore, it is generated by the exceptional collection $\langle \mathbb{L}_{\mathfrak{E}}\cR ,\ \mathbb{L}_{\mathfrak{E}}\Sigma^{2,1}\cU^{\vee},\ \mathbb{L}_{\mathfrak{E},\mathfrak{E}(1)}(\cR(1))\rangle$, i.e. by Lemma \ref{lemma:mutation21} and Lemma \ref{lemma:mutationR(1)},
\begin{equation*}
 \mathsf{Res}=\langle \mathbb{L}_{\mathfrak{E}}\cR ,\Sigma^{2,1}\cU^{\vee}(-1),\cR(-1)\rangle.
\end{equation*}
By definition, we have $\tau(\cR(-1))=\mathbb{L}_{\mathfrak{E}}\cR$. By Lemma \ref{lemma:mutation21} we get~$\tau(\Sigma^{2,1}\cU^{\vee}(-1))=\Sigma^{2,1}\cU^{\vee}(-1)[2]$. By Lemma \ref{lemma:mutationR(1)} we obtain~$\tau( \mathbb{L}_{\mathfrak{E}}\cR)=\cR(-1)$ up to shift . 

Semiorthogonality is evident. The complete orthogonality follows from the fact that $\tau$ is the autoequivalence. 
\end{proof}

\end{document}